\documentclass[reqno,12pt]{amsart}      
\usepackage{amsfonts}
\usepackage{colordvi}
\usepackage[usenames,dvipsnames,svgnames,table]{xcolor}
\usepackage{multirow}
\usepackage{multicol}
\usepackage{amsmath,amssymb,amsthm} 
\usepackage{mathrsfs}
\usepackage{graphicx}
\usepackage{tikz}
\usepackage{ulem}  % rule out text

\theoremstyle{plain}      
\newtheorem{theorem}{Theorem} %[section]     
\newtheorem{corollary}[theorem]{Corollary}     
\newtheorem{lemma}[theorem]{Lemma}     
\newtheorem{proposition}[theorem]{Proposition}     
\newtheorem{conjecture}[theorem]{Conjecture}     
\newtheorem{algorithm}[theorem]{Algorithm}     

\theoremstyle{remark}      
\newtheorem{example}[theorem]{Example} 
\newtheorem{remark}[theorem]{Remark} 

\newtheorem{definition}[theorem]{Definition}     

\newcounter{FNC}[page]
\def\fauxfootnote#1{{\addtocounter{FNC}{2}\Magenta{$^\fnsymbol{FNC}$}%
     \let\thefootnote\relax\footnotetext{\Magenta{$^\fnsymbol{FNC}$#1}}}}

\def\Color#1#2{#2}
\newcommand{\defcolor}[1]{\RoyalBlue{#1}}
\newcommand{\demph}[1]{\defcolor{{\sl #1}}}

\def\Nred#1{\Color{0 0.6 0.6 0}{#1}}
\newcommand{\cbd}{{\Black{\bullet}\Nred{\bullet}}}

\newcommand{\A}{{\mathbb A}}
\newcommand{\CC}{{\mathbb C}}
\newcommand{\FF}{{\mathbb F}}
\newcommand{\PP}{{\mathbb P}}
\newcommand{\QQ}{{\mathbb Q}}
\newcommand{\ZZ}{{\mathbb Z}}

\newcommand{\Fln}{{\mathbb F}\ell_n}
\newcommand{\Fl}{{\mathbb F}\ell}
\newcommand{\Gr}{\mbox{\it Gr\,}}
\newcommand{\Gal}{\mbox{\rm Gal}}
\newcommand{\Mon}{\mbox{\rm Mon}}

\newcommand{\calE}{{\mathcal E}}
\newcommand{\calF}{{\mathcal F}}
\newcommand{\calG}{{\mathcal G}}
\newcommand{\calO}{{\mathcal O}}
\newcommand{\calT}{{\mathcal T}}
\newcommand{\calV}{{\mathcal V}}
\newcommand{\calX}{{\mathcal X}}
\newcommand{\calY}{{\mathcal Y}}
\newcommand{\calZ}{{\mathcal Z}}

\newcommand{\lhra}{\ensuremath{\lhook\joinrel\relbar\joinrel\relbar\joinrel\rightarrow}}
\newcommand{\calEdot}{\calE_{\bullet}}
\newcommand{\calFdot}{\calF_{\bullet}}
\newcommand{\calGdot}{\calG_{\bullet}}
\newcommand{\Edot}{E_{\bullet}}
\newcommand{\Fdot}{F_{\bullet}}
\newcommand{\Gdot}{G_{\bullet}}
\newcommand{\Mdot}{M_{\bullet}}
\newcommand{\blambda}{{\boldsymbol{\lambda}}}
\newcommand{\bmu}{{\boldsymbol{\mu}}}
\newcommand{\bnu}{{\boldsymbol{\nu}}}

\newcommand{\bI}{\includegraphics{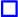}}
\newcommand{\I}{\includegraphics{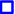}}
\newcommand{\II}{\includegraphics{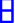}}
\newcommand{\T}{\includegraphics{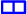}}
\newcommand{\TI}{\includegraphics{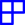}}
\newcommand{\mTI}{\includegraphics{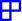}}
\newcommand{\sI}{{\includegraphics{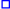}}}
\newcommand{\sII}{{\includegraphics{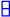}}}

\newcommand{\shIII}{{\includegraphics{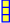}}}
\newcommand{\sT}{{\includegraphics{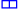}}}
\newcommand{\sTI}{{\includegraphics{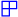}}}
\newcommand{\sTII}{{\includegraphics{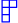}}}
\newcommand{\bTII}{{\includegraphics{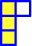}}}
\newcommand{\scTII}{{\includegraphics{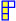}}}
\newcommand{\sTT}{{\includegraphics{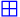}}}
\newcommand{\mTT}{{\includegraphics{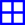}}}
\newcommand{\TT}{{\includegraphics{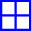}}}
\newcommand{\sTh}{{\includegraphics{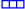}}}
\newcommand{\shTh}{{\includegraphics{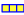}}}
\newcommand{\sThI}{{\includegraphics{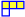}}}
\newcommand{\sFI}{{\includegraphics{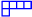}}}
\newcommand{\shF}{{\includegraphics{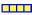}}}
\newcommand{\shFI}{{\includegraphics{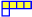}}}
\newcommand{\FI}{{\includegraphics{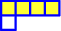}}}
\newcommand{\FT}{{\includegraphics{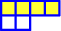}}}
\newcommand{\ThThI}{{\includegraphics{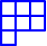}}}
\newcommand{\ThTh}{{\includegraphics{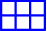}}}
\newcommand{\sThTh}{{\includegraphics{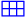}}}

\begin{document}     

%%%%%%%%%%%%%%%%%%%%%%%%%%%%%%%%%%%%%%%%%%%%%%%%%%%%%%%%%%%%%%%%%%%%%%%%%

\title[Schubert Galois groups in $\Gr(4,9)$]{Classification of Schubert Galois
  groups in $\Gr(4,9)$} 

%%%%%%%%%%%%%%%%%%%%%%%%%%%%%%%%%%%%%%%%%%%%%%%%%%%%%%%%%%%%%%%%%%%%%%%%%%%%%%%%%
\author[Mart\'in del Campo]{Abraham Mart\'in del Campo}
\address{Abraham Mart\'in del Campo\\
         Centro de Investigaci\'on en Matem\'aticas, A.C.\\
         Jalisco S/N, Col. Valenciana\\
         36023 \ Guanajuato, Gto.  M\'exico}
\email{abraham.mc@cimat.mx}
\urladdr{http://personal.cimat.mx:8181/\~{}abraham.mc}
%%%%%%%%%%%%%%%%%%%%%%%%%%%%%%%%%%%%%%%%%%%%%%%%%%%%%%%%%%%%%%%%%%%%%%%%%%%%%%%%%
\author[Sottile]{Frank Sottile}
\address{Frank Sottile\\
         Department of Mathematics\\
         Texas A\&M University\\
         College Station\\
         Texas \ 77843\\
         USA}
\email{sottile@math.tamu.edu}
\urladdr{www.math.tamu.edu/\~{}sottile}
%%%%%%%%%%%%%%%%%%%%%%%%%%%%%%%%%%%%%%%%%%%%%%%%%%%%%%%%%%%%%%%%%%%%%%%%%%%%
\author[Williams]{Robert Lee Williams}
\address{Robert Lee Williams\\
         Department of Mathematics\\
         Rose-Hulman Institute of Technology\\
         Terre Haute\\
         IN \ 47803\\
         USA}
\email{williams1eipi0@gmail.com}
\urladdr{~}
%%%%%%%%%%%%%%%%%%%%%%%%%%%%%%%%%%%%%%%%%%%%%%%%%%%%%%%%%%%%%%%%%%%%%%%%%%%%%%%%%
\thanks{Work of Mart\'in del Campo supported in part by CONACyT under grant C\'atedra-1076}
\thanks{Work of Sottile supported in part by the National Science Foundation under grant DMS-1501370}
%%%%%%%%%%%%%%%%%%%%%%%%%%%%%%%%%%%%%%%%%%%%%%%%%%%%%%%%%%%%%%%%%%%%%%%%%%%%
% 14N15 Classical problems, Schubert calculus 
% 12F10  Field extensions: Separable extensions, Galois theory 
% 12F12 Field extensions:  Inverse Galois theory
%%%%%%%%%%%%%%%%%%%%%%%%%%%%%%%%%%%%%%%%%%%%%%%%%%%%%%%%%%%%%%%%%%%%%%%%%%%%%%%%%

\begin{abstract}
 We classify Schubert problems in the Grassmannian of 4-planes in 9-dimen\-sion\-al space
 by their Galois groups.
 Of the 31,806 essential % This has been checked
 Schubert problems in this Grassmannian, there are only 149 whose Galois group does not
 contain the alternating group.  
 We identify the Galois groups of these 149---each is an imprimitive permutation group.
 These 149 fall into two families according to their geometry.
 This study suggests a possible classification of Schubert problems  whose Galois group is not the full symmetric
 group, and is a first step towards the inverse Galois problem for Schubert calculus.
\end{abstract}

\maketitle

%%%%%%%%%%%%%%%%%%%%%%%%%%%%%%%%%%%%%%%%%%%%%%%%%%%%%%%%%%%%%%%%%%%%%%%%%%%%%%%%%
\section*{Introduction}
In  ``{\it Trait\'e des substitutions et des \'equations alg\'ebriques}'',
Jordan~\cite{J1870} explained how a problem in enumerative geometry has a Galois group that
acts on its solutions.
If the set of solutions of an enumerative problem possess additional geometric 
structure, then its Galois group must preserve that structure, and thus cannot be the full symmetric group
on its solutions---we call such a  problem/Galois group \demph{enriched}. 
Jordan studied several enriched enumerative problems, such as the 27 lines on a smooth cubic surface in ${\mathbb P}^3$.
Later, Harris~\cite{Ha79} showed that Jordan's enriched problems 
had Galois group as large as possible given  the structure of the set of solutions.
He also showed that natural  generalizations of each of Jordan's enriched problems, as well as some other enumerative
problems, such as Chasles' problem of 3264 plane conics tangent to five given conics~\cite{Ch1864},
have full symmetric Galois groups. 

Until recently, Galois groups of enumerative problems were difficult to study, for there were few methods  
available. Vakil's geometric Littlewood-Richardson rule~\cite{Va06a,Va06b} leads to a recursive 
method that may show the Galois group of a Schubert problem in a Grassmannian (a 
\demph{Schubert Galois group}) contains the alternating group on its solutions (is \demph{at least alternating}). 
Numerical algebraic geometry can compute a monodromy group~\cite{HRS,LS09}, which equals the Galois
group~\cite{Ha79,Hermite}.
Symbolic computation can compute cycle types of Frobenius elements in Galois groups over $\QQ$.
With many thousands to hundreds of millions of computable Schubert problems in Grassmannians, Schubert calculus
forms a laboratory for studying Galois groups in enumerative geometry. 

Write $\defcolor{\Gr(k,n)}$ for the Grassmannian of $k$-planes in $\CC^n$.
Vakil~\cite{Va06b} used his method to show that every Schubert Galois group in $\Gr(2,n)$ for $n\leq 16$ and in $\Gr(3,n)$ for
$n\leq 9$ is at least alternating.
%
%  We need to deduce this later in the text.  It is easy.
%
Vakil's method implies that the
Galois group of any Schubert problem with two or three solutions is the full symmetric group.
Derksen found an enriched problem in $\Gr(4,8)$ with six solutions, which Vakil generalized to an infinite
family of enriched Schubert problems with members in every Grassmannian $\Gr(k,n)$ for 
$4\leq k\leq n{-}4$~\cite[Section~3.14]{Va06b}. 
Numerical computation of monodromy showed~\cite{LS09} that many simple (explained
in~\S\ref{SS:SchubertCalculus}) Schubert problems 
have Galois group the full symmetric group, including one with 17,589 solutions in $\Gr(3,9)$.
Combinatorial and analytic arguments starting with Vakil's method showed~\cite{BdCS} that, for any $n$, every Schubert
Galois group in $\Gr(2,n)$ is at least alternating.
Geometric and combinatorial methods were used in~\cite{SW_double} to show that every Schubert Galois group in $\Gr(3,n)$, for
any $n$, is  2-transitive. 
%
%  This needs a definition later in the text
%
A Schubert problem is essential (see the discussion around~\eqref{Eq:essential} in \S\ref{SS:SchubertCalculus})
if it is not equivalent to one on a smaller Grassmannian.
Galois groups of all 3501 essential Schubert problems in $\Gr(4,8)$ were studied in~\cite[Section~4]{MSJ}
and~\cite{SW_double}.  %Checked
All except 14 are at least alternating.
Each of those 14 has an imprimitive Galois group and they fall into three families with Derksen's example forming one
family. 

Here, we study the Galois groups of all 31,806 %Checked
essential Schubert problems in the next Grassmannian, $\Gr(4,9)$,
determining that all except 149 are at least alternating.
Each of the 149 has an imprimitive Galois group that is a wreath product of two symmetric groups.
All 149 and 13 of the 14 in $\Gr(4,8)$ have a common structure---they are a fibration of Schubert problems
(see Section~\ref{S:geometry}) in either $\Gr(2,4)$ or $\Gr(2,5)$, which explains their Galois group. 
Among those enriched problems, only Derksen's is not a fibration, and its structure was explained by Vakil.
The simplicity of this classification for $\Gr(4,8)$ and $\Gr(4,9)$ suggests the possibility of classifying all enriched
Schubert problems in Grassmannians.

Each Schubert Galois group over $\CC$ is a normal subgroup of the Galois group over $\QQ$, and we conjecture that
the two groups are equal.
Computing cycle types of Frobenius elements enables us to determine the Galois group over $\QQ$ of every essential Schubert
problem in $\Gr(4,9)$ with fewer than 300 solutions---26,051 problems in all.
It is worth noting that each enriched Schubert problem in $\Gr(4,9)$ has ten or fewer solutions.
Other than the 149 enriched problems, all remaining 26,353 have full symmetric Galois groups.
We highlight this observed dichotomy from our study of all small Grassmannians.

%%%%%%%%%%%%%%%%%%%%%%%%%%%%%%%%%%%%%%%%%%%%%%%%%%%%%%%%%%%%%%%%%%%%%%%%%%%%%%%%%
\begin{theorem}
 Every known Schubert Galois group is either
 \begin{enumerate}
  \item the full symmetric group on the solutions, and hence maximally transitive, or
  \item not $2$-transitive.
    %
    % Not all of these are imprimitive, see below
    %
 \end{enumerate}
\end{theorem}
%%%%%%%%%%%%%%%%%%%%%%%%%%%%%%%%%%%%%%%%%%%%%%%%%%%%%%%%%%%%%%%%%%%%%%%%%%%%%%%%%

Many, but not all, of the Schubert Galois groups that are not the full symmetric group are imprimitive.
Recently, Esterov~\cite{Esterov} has shown a similar dichotomy for systems of sparse polynomial equations; either their
Galois group is full symmetric or it is imprimitive.

The known Schubert Galois groups are very particular permutation groups.
Most are symmetric groups $S_d$ acting naturally on the set $\defcolor{[d]}:=\{1,\dotsc,d\}$.
The Galois group of Derksen's example is the induced action of
$S_4$ on the six equipartitions of $[4]$, 
\[
  12|34\,,\ 13|24\,,\ 14|23\,,\ 23|14\,,\ 24|13\,,\ 34|12\,.
\] 
This action is imprimitive as it preserves the partition
\[
   \{\, 12|34\,,\, 34|12\,\}\ \sqcup\  
   \{\, 13|24\,,\, 24|13\,\}\ \sqcup\ 
   \{\, 14|23\,,\, 23|14\,\}\,. 
\]
Members of Vakil's infinite family have a similar action of $S_n$ on certain partitions of $[n]$.
Not all are imprimitive, for example, one has Galois group $S_5$ acting on the 10 partitions of $[5]$ into two
subsets, one of cardinality two and one of cardinality three.
This primitive permutation group is not $2$-transitive.
The remaining Galois groups we found are wreath products of two symmetric groups
$\defcolor{S_m\wr S_d}=(S_m)^d\rtimes S_d$. 
Each essential enriched Schubert problem in $\Gr(4,9)$ has Galois group 
one of $S_2\wr S_2$, $S_3\wr S_2$, $S_5\wr S_2$, or $S_2\wr S_3$.\smallskip

Mathematicians from Hilbert to Arnold have stimulated the development of mathematics by proposing to address the first
nontrivial or next unknown instance of a general question.
This investigation is in their spirit.
As it touches on geometry, combinatorics, number theory, and group theory, it represents  the unity of
mathematics. \smallskip

We begin in Section~\ref{S:background} by sketching some background, including Schubert problems in Grassmannians, Galois
groups of branched covers, and permutation groups.
In Section~\ref{S:computations} we discuss our computations, which give a lower bound for each Schubert Galois group.
Section~\ref{S:geometry} contains a detailed study of the structure of some Schubert problems, identifying
classes of enriched Schubert problems whose Galois group is a subgroup of a nontrivial wreath
product, and determining the Galois group of each enriched problem in $\Gr(4,9)$.
This explains general methods to study Schubert Galois groups and it points to a possible
classification of enriched Schubert problems in Grassmannians. 

Some of this paper is based on the 2017 Ph.D.\ thesis of Williams~\cite{Williams}.
The computations used a Maple script of Vakil\footnote{\tt http://math.stanford.edu/\~{}vakil/programs/galois}, as well as
software developed by the authors and  by Christopher Brooks, Aaron Moore, James Ruffo, and Luis Garc\'ia-Puente.

%%%%%%%%%%%%%%%%%%%%%%%%%%%%%%%%%%%%%%%%%%%%%%%%%%%%%%%%%%%%%%%%%%%%%%%%%%%%%%%%%
\section{Background}\label{S:background}

We work over the complex numbers, $\CC$, although our results hold for any algebraically closed field of characteristic
zero. 

%%%%%%%%%%%%%%%%%%%%%%%%%%%%%%%%%%%%%%%%%%%%%%%%%%%%%%%%%%%%%%%%%%%%%%%%%%%%%%%%%
\subsection{Schubert Calculus in Grassmannians}\label{SS:SchubertCalculus}
The Schubert calculus in Grassmannians concerns all problems of enumerating the linear subspaces of a vector space 
that satisfy incidence conditions imposed by other, fixed linear subspaces.
The simplest non-trivial Schubert problem asks for the 2-planes in $\CC^4$ that meet four
general 2-planes nontrivially.
Passing to lines in projective 3-space, this asks for the lines that meet four general lines.

To understand the solutions, consider first three pairwise skew lines \Blue{$\ell_1$}, \Red{$\ell_2$}, and \Green{$\ell_3$}. 
They lie on a unique hyperboloid (see Figure~\ref{F:4lines}).
%%%%%%%%%%%%%%%%%%%%%%%%%%%%%%%%%%%%%%%%%%%%%%%%%%%%%%%%%%%%%%%%%%%%%%%%%%%%%%%%%
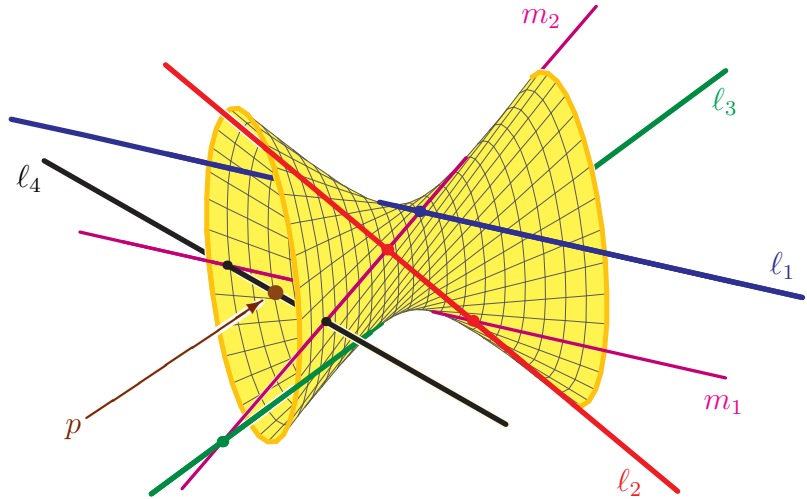
\begin{figure}[htb]
%
%   If we need to squeeze a few more lines out for the page limit, it is possible
%   to shrink this by 10-20%.
%
\centerline{
  \begin{picture}(314,192)
   \put(3,0){\includegraphics[height=6.6cm]{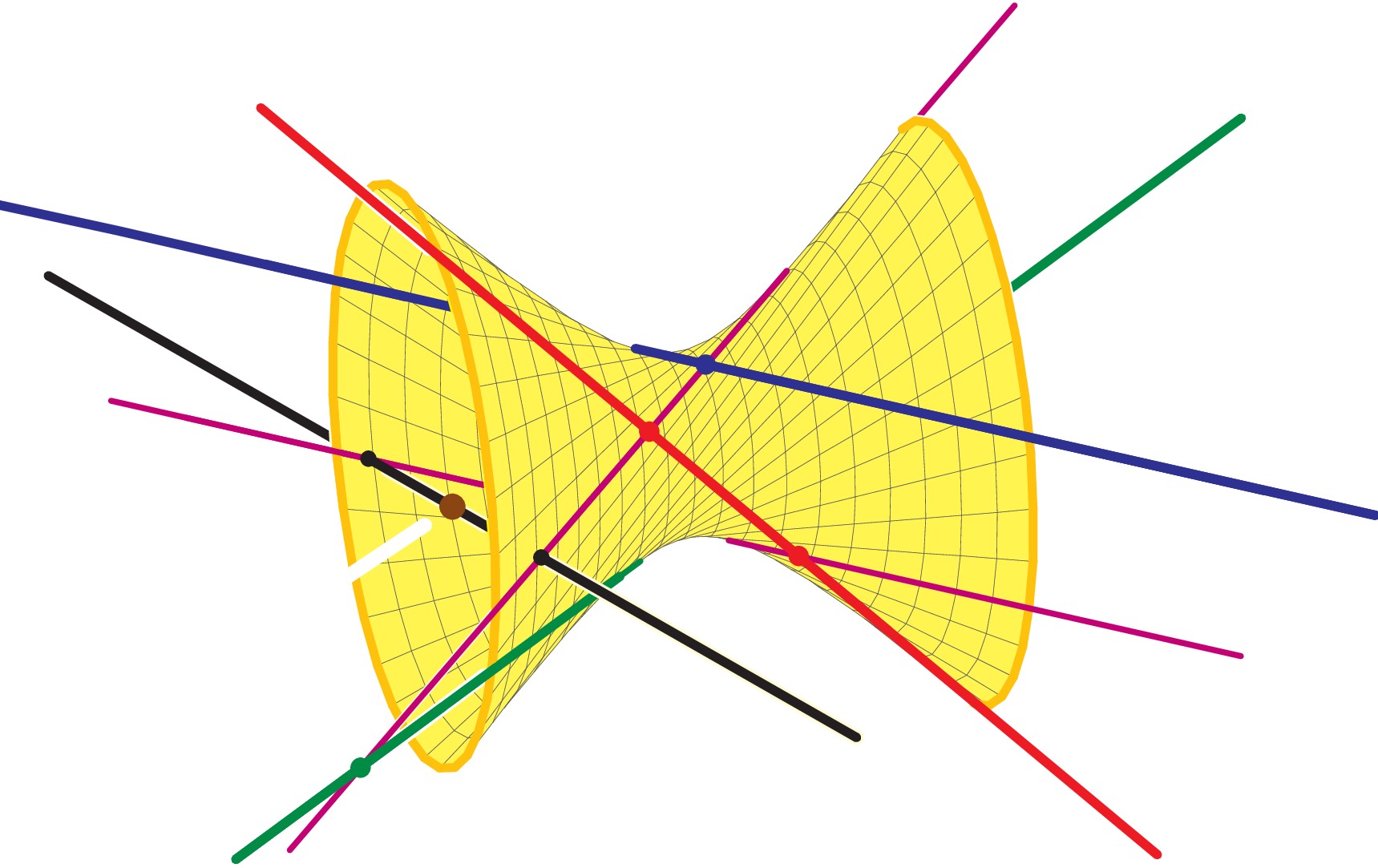}}
   \put(230,  3){\Red{$\ell_2$}}
   \put(288, 85){\Blue{$\ell_1$}}
   \put(266,147){\Green{$\ell_3$}}
   \put(  3,118){$\ell_4$}
   \put(263, 34){\Magenta{$m_1$}}
   \put(194,180){\Magenta{$m_2$}}
   \thicklines
  \put(30,31.2){\White{\line(3,2){66.6}}}
  \put(30,29.8){\White{\line(3,2){66.6}}}
%   \thinlines
   \put(29.5,30.5){\Brown{\vector(3,2){66.6}}}
   \put( 21.5, 25){\Brown{$p$}}
  \end{picture}
}
%   If a smaller picture is needed.......
%\centerline{
%  \begin{picture}(189,108)(-1,5)
%   \put(3,0){\includegraphics[height=4.1cm]{figures/4lines}}
%   \put(159,  5){\Red{$\ell_2$}}
%   \put(180, 53){\Blue{$\ell_1$}}
%   \put(159,103){\Green{$\ell_3$}}
%   \put(  2, 72){$\ell_4$}
%   \put(160, 22){\Magenta{$m_1$}}
%   \put(120,113){\Magenta{$m_2$}}
%
%   \thicklines
%   \put(21,20.1){\White{\line(3,2){39}}}
%   \thinlines
%   \put(21,20.1){\Brown{\vector(3,2){39}}}
%   \put(14, 16){\Brown{$p$}}
%  \end{picture}
%}
 \caption{Problem of four lines}\label{F:4lines}
\end{figure}
%%%%%%%%%%%%%%%%%%%%%%%%%%%%%%%%%%%%%%%%%%%%%%%%%%%%%%%%%%%%%%%%%%%%%%%%%%%%%%%%%
This hyperboloid has two rulings by lines:  \Blue{$\ell_1$}, \Red{$\ell_2$}, and \Green{$\ell_3$} lie in one, and
the lines that meet all three form the second ruling.
The fourth line, $\ell_4$, meets the hyperboloid in two points, and the lines \Magenta{$m_1$} and \Magenta{$m_2$}
in the second ruling through these points are the two solutions to our problem of four lines.

Since $m_1$ and $m_2$ lie in the same ruling, they do not meet.
Figure~\ref{F:4lines} also illustrates that the Galois group of this Schubert problem is the symmetric group $S_2$.
Rotating the line $\ell_4$ by $180^\circ$ about the point $p$ interchanges two solution lines, so the Galois
group contains a transposition.
For another way to see this, observe that if $\ell_4$ moves 
to become tangent to the hyperboloid, then  the two solution lines will coincide.
Let $q\in \ell_4$ be a point not on the hyperboloid and $H$ a plane containing $\ell_4$ that is not tangent to the hyperboloid.
  Then the pencil of lines $\defcolor{\PP(H/q)}:=\{\ell\mid q\in\ell\subset H\}$ contains two lines that are tangent to the 
  hyperboloid. (To see this in the real picture of Figure~\ref{F:4lines}, $q$ should lie outside of the hyperboloid.)
This implies that the local monodromy in the pencil near the tangent line contains a 2-cycle.
It also implies that letting any line move in a general pencil will produce a 2-cycle in the monodromy group.
We later use this observation to establish that some Galois groups are as large as possible.
This argument of deducing a simple transposition from a solution of multiplicity 2 was used by Harris~\cite{Ha79}
and by Byrnes and Stevens~\cite{BS_homotopy} to study Galois groups.

For future reference, if $\ell\subset\Lambda$ are linear subspaces of dimensions $r{-}1$ and $r{+}1$, then
  $\PP(\Lambda/\ell)\simeq\PP^1$ is the \demph{pencil} of $r$-planes $L$ such that $\ell\subset L\subset\Lambda$.

For additional material on the Schubert calculus on Grassmannians, see~\cite{Fu97}.
Let $V$ be a complex vector space of dimension $n$.
Write $\defcolor{\Gr(k,n)}$ or $\defcolor{\Gr(k,V)}$ for the Grassmannian of $k$-planes in $V$.
This is an algebraic manifold of dimension $k(n{-}k)$, and  $\Gr(k,n)$ is isomorphic to
$\Gr(n{-}k,n)$. 
The ambient space for the problem of four lines is $\Gr(2,4)$.
Incidence conditions on the $k$-planes of $\Gr(k,n)$ are indexed by
\demph{partitions}, which are weakly decreasing sequences of nonnegative integers
\[
   \lambda\ \colon\  n{-}k\ \geq\  \lambda_1\ \geq\ 
    \lambda_2\ \geq\ \dotsb\ \geq\ \lambda_k\ \geq\ 0\,.
\]
For example, both $(4,4,3,1)$ and $(3,1,0,0)$ are partitions for $\Gr(4,9)$.
Trailing 0s are often omitted.
We represent partitions by their Young diagrams, which are left-justified arrays of
boxes, with $\lambda_i$ boxes in row $i$.
For these partitions, we have 
\[
   (4,4,3,1)\ =\ \,\raisebox{-11pt}{\includegraphics{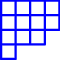}}
   \qquad\mbox{and}\qquad
   (3,1)\ =\ \,\raisebox{-4pt}{\includegraphics{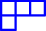}}\ .
\]
More generally, given positive integers $a<b$ and a partition $\lambda$, we say that $\lambda$ is a partition for $\Gr(a,b)$
when $\lambda_1\leq b{-}a$ and $\lambda_{a+1}=0$.

A condition $\lambda$ is imposed by a (complete) \demph{flag}, which is a sequence of linear subspaces,
 \begin{equation}\label{Eq:FullFlag}
   \defcolor{\Fdot}\ \colon\ 
   \{0\}\ \subsetneq\ F_1\ \subsetneq\ F_2\ \subsetneq\ \dotsb\ \subsetneq\ F_n\ =\ \CC^n\,,
 \end{equation}
where $\dim(F_j)=j$.
The condition $\lambda$ on a $k$-plane $H\in \Gr(k,n)$ imposed by the flag $\Fdot$ is
 \begin{equation}\label{Eq:SchubertCondition}
   \dim( H\cap F_{n-k+j-\lambda_j})\ \geq\ j 
   \qquad\mbox{for } j=1,\dotsc,k\,.
 \end{equation}
For $\lambda=(3,1)$ on $\Gr(4,9)$, as $9-4+1-3=3$ and $9-4+2-1=6$, this is
\[
    \dim( H\cap F_3)\geq 1
   \quad\mbox{ and }\quad
    \dim (H\cap F_6)\geq 2\,,
\]
as the conditions for the trailing 0s, $\dim (H\cap F_8)=3$ and $\dim( H\cap F_9)=4$, always hold.

The set of all $H\in \Gr(k,n)$ satisfying condition~\eqref{Eq:SchubertCondition} is the
\demph{Schubert variety} \defcolor{$\Omega_\lambda\Fdot$}.
This is an irreducible subvariety of $\Gr(k,n)$ of codimension $\defcolor{|\lambda|}:=\lambda_1+\dotsb+\lambda_k$.
In the problem of four lines, the condition that a 2-plane $H$ in $\CC^4$ meets the 2-dimensional subspace of
$\Fdot$ is encoded by the partition $(1,0)=\I$.
The Schubert variety $\Omega_\sI\Fdot$ is a hypersurface in the four-dimensional Grassmannian $\Gr(2,4)$.
Observe that only one of the subspaces in $\Fdot$ is needed to define  $\Omega_\sI\Fdot$ and only two of the subspaces in
the flag $\Fdot$ are used to define $\Omega_{(3,1)}\Fdot$. 
In general, only the subspaces of a flag $\Fdot$ corresponding to the corners
($\{j\mid \lambda_j>\lambda_{j+1}\}$) of the partition $\lambda$ are needed to define $\Omega_\lambda\Fdot$.
%%%%%%%%%%%%
When a Schubert variety is defined by specifying the necessary subspaces, we may write them in place of the
complete flag $\Fdot$.
Thus, we may write $\Omega_\sI F_2$ or $\Omega_\sI(F_2\subset F_4)$ instead of $\Omega_\sI\Fdot$
for the same hypersurface in $\Gr(2,4)$. 
%%%%%%%%%%%%%

A \demph{Schubert problem} in $\Gr(k,n)$ is a list  $\defcolor{\blambda}=(\lambda^1,\dotsc,\lambda^s)$ of partitions with 
$k(n{-}k)=|\lambda^1|+\dotsb+|\lambda^s|$.
The problem of four lines is the Schubert problem $(\I,\I,\I,\I)$, which we write in multiplicative form as $\I^4$;
the exponent $4$ indicates that $\I$ is repeated four times.
A Schubert problem is \demph{simple} if all except one or two of its conditions are hypersurface
Schubert conditions, $\I$.
Using the multiplicative notation, a simple Schubert problem in
$\Gr(k,n)$ has the form $\lambda\cdot\mu\cdot\I^{s-2}$, where $k(n{-}k)=|\lambda|+|\mu|+s{-}2$.

Let $\blambda$ be a Schubert problem in $\Gr(k,n)$.
An \demph{instance} of $\blambda$ is given by a choice  $\defcolor{\calFdot}:=(\Fdot^1,\dotsc,\Fdot^s)$ of flags, and it
corresponds to the intersection
\begin{equation}\label{Eq:intersection}
  \defcolor{\Omega_{\blambda}\calFdot}\ :=\ 
   \Omega_{\lambda^1}\Fdot^1 \,\bigcap\, 
   \Omega_{\lambda^2}\Fdot^2 \,\bigcap\, \dotsb \,\bigcap\, 
   \Omega_{\lambda^s}\Fdot^s\,. 
 \end{equation}
If the flags in $\calFdot$ are general, then this intersection is transverse~\cite{K74} and consists of finitely many
points.
A consequence of transversality is that if $\calFdot$ is general, then the inequalities in~\eqref{Eq:SchubertCondition}
hold with equality.
This number \demph{$d(\blambda)$}  of points in the intersection~\eqref{Eq:intersection} is independent of the choice of general flags.
It may be computed using algorithms from the Schubert calculus~\cite{Fu97,KL72}. 
For brevity, we sometimes denote Schubert problems by $\blambda = d(\blambda)$ to specify its number of solutions, e.g.,
the problem of four lines can be denoted as $\I^4 = 2$, to indicate that the problem has two solutions.
%########################################################
%> interface(quiet=true):
%Elapsed  Time :=10992.49  seconds
%NumberofProblems:=81533
A Schubert problem $\blambda$ is \demph{nontrivial} if $d(\blambda)>1$.
There are $81,533$ nontrivial Schubert problems in $\Gr(4,9)$.
We have a Maple script\footnote{Available at {\tt http://www.math.tamu.edu/\~{}sottile/research/stories/GIVIX}}
that takes about 3 hours to enumerate these and compute $d(\blambda)$.\smallskip

%%%%%%%%%%%%%%%%%%%%%%%%%%%%%%%%%%%%%%%%%%%%%%%%%%%%%%%%%%%%%%%%%%%%%%%%%%%%%%%%%

The set \defcolor{$\Fln$} of all flags in $\CC^n$ is a smooth irreducible rational algebraic variety of dimension
$\binom{n}{2}$. 
Given a Schubert problem $\blambda=(\lambda^1,\dotsc,\lambda^s)$, we have the incidence variety
 \begin{equation}\label{Eq:TotalSpace}
  \raisebox{-35pt}{\begin{picture}(307,75)(-5,-4)
   \put(39,60){$\defcolor{\calX_\blambda}\ :=\ 
           \{(H,\Fdot^1,\dotsc,\Fdot^s)\mid H\in\Omega_{\lambda^i}\Fdot^i\ \mbox{ for}\ i=1,\dotsc,s\}$\,.}
   \put(40,55){\vector(-1,-2){20}}   \put(50,55){\vector(1,-2){20}}
   \put(20,33){$p$} \put(65,33){$\pi$}
   \put(-8,0){$\Gr(k,n)$}   \put(60,0){$(\Fln)^s$}
   \end{picture}}
 \end{equation}
%
%  This is explained, in slightly different notation
%
The total space $\calX_\blambda$ is irreducible as the map  $p\colon\calX_\blambda\to \Gr(k,n)$ is a fiber bundle with
irreducible fibers.
This is explained  with slightly different notation in~\cite[Section~2.2]{SW_double} as follows:
For any point $H\in\Gr(k,n)$ and any partition $\lambda$, the set of flags $\Fdot\in\Fl_n$ such that 
$H\in\Omega_\lambda\Fdot$ forms a Schubert subvariety \defcolor{$\Psi_\lambda H$} of $\Fl_n$, which is irreducible.
The inverse image $p^{-1}(H)$ of $H\in\Gr(k,n)$ is the product of the Schubert varieties $\Psi_{\lambda^i}H$, one in each
of the flag manifold factors, and  is thus irreducible. 
For an $s$-tuple $\calFdot=(\Fdot^1,\dotsc,\Fdot^s)$ of flags, the fiber $p(\pi^{-1}(\calF))$ is the
intersection~\eqref{Eq:intersection}, which forms the subvariety  $\Omega_{\blambda}\calFdot$.
Thus $\calX_\blambda\to(\Fln)^s$ is the family of all instances of the Schubert problem $\blambda$.
By Kleiman's Theorem~\cite{K74}, there is a dense open subset $U\subset(\Fln)^s$ consisting of $s$-tuples of flags for
which the intersection~\eqref{Eq:intersection} is transverse and consists of $d(\blambda)$ reduced points.
Over the set $U$, the map $\pi\colon\calX_\blambda\to(\Fln)^s$ is a covering space of degree $d(\blambda)$.

The variety $\calX_\blambda\subset \Gr(k,n)\times(\Fln)^s$ may be defined in local coordinates by
polynomials with integer coefficients. 
The dimension conditions~\eqref{Eq:SchubertCondition} are formulated as rank conditions on matrices, which
become the vanishing of minors of matrices whose entries are variables or constants.
This and other formulations are explained in~\cite{HaHS,HS_Sq,LMSVV,MSJ}.\smallskip

It is sometimes the case that Schubert problems $\blambda$ and $\bmu$ in possibly different Grassmannians are equivalent in
the following way:
Given a general instance $\calFdot$ of $\blambda$, there is an instance $\calEdot$ of $\bmu$ and a natural bijection
involving linear algebraic constructions of the two sets of solutions $\Omega_\blambda\calFdot$ and $\Omega_\bmu\calEdot$.
Finally, all general instances $\calEdot$ of $\bmu$ occur in this way.

%%%%%%%%%%%%%%%%%%%%%%%%%%%%%%%%%%%%%%%%%%%%%%%%%%%%%%%%%%%%%%%%%%%%%%%%%%%%%%%%%
\begin{example}\label{Ex:first_reduction}
  Consider the Schubert problem $\T^2\cdot\I^2=2$ in $\Gr(2,5)$.
  An instance of it is given by two 2-planes, $L_1,L_2$, and two 3-planes $\Lambda_3,\Lambda_4$ in $\CC^5$, and the
  solutions are those 2-planes $H$ that meet each of these non-trivially. 
  We write this instance as
  \[
    \Omega_{\sT}L_1 \,\cap\, \Omega_{\sT}L_2\,\cap\, \Omega_{\sI}\Lambda_3\,\cap\,\Omega_{\sI}\Lambda_4\,.
  \]
  Assuming that the linear subspaces are general, a solution $H$ is spanned by its intersections with any two of these
  linear spaces.
  Thus $H\subset\langle L_1, L_2\rangle=:M\simeq\CC^4$.
  The 2-plane $H$ also meets each of $L_3:=\Lambda_3\cap M$ and  $L_4:=\Lambda_4\cap M$, which are 2-planes.
  Thus $H$ is a solution to the instance of the Schubert problem $\I^4=2$ in $\Gr(2,M)$ given by $L_1,\dotsc,L_4$.
  Lastly, all general instances of $\I^4=2$ on $\Gr(2,M)$ occur in this way. \hfill$\diamond$
\end{example}
%%%%%%%%%%%%%%%%%%%%%%%%%%%%%%%%%%%%%%%%%%%%%%%%%%%%%%%%%%%%%%%%%%%%%%%%%%%%%%%%%

We do not aim to classify Schubert problems up to this notion of equivalence---a consequence of the arguments
in~\cite{BCDLT} is that all Schubert problems $\blambda$ with $\delta(\blambda)=1$ are equivalent.
We will however use equivalent in a restricted sense.
In Example~\ref{Ex:first_reduction},  $\T^2\cdot\I^2=2$ in $\Gr(2,5)$ was shown to be equivalent to $\I^4=2$ in 
$\Gr(2,4)$.
Two of the conditions, $\T$ and $\T$, imposed a linear algebraic condition on solutions which reduced solving the original
problem to solving an instance of $\I^4=2$ in a $\Gr(2,4)$.
Passing from one Schubert problem to an equivalent problem in a Grassmannian of a smaller dimensional
vector space 
coming from linear algebraic constraints imposed by one or two conditions was crucial for the arguments in~\cite{BdCS}
and~\cite{SW_double}.
We use the word \demph{reduction} to refer to this process.

In~\cite[Proposition~6]{SW_double}, it was shown that such a reduction was possible for a Schubert problem $\blambda$ if
there were partitions $\mu,\nu$ from $\blambda$ such that one of the following did not hold:
 \begin{equation}\label{Eq:essential}
  \begin{tabular}{rcl}
   (a)&\ &$\mu_1<n-k$, \\
   (b)&& $\mu_k=0$, \\
   (c)&& $\mu_i+\nu_{k+1-i}< n-k$ for every $i=1,\dotsc,k$, and \\
   (d)&& $\mu_i+\nu_{k-i}\leq n-k$ for every $i=1,\dotsc,k{-}1$.
  \end{tabular}
 \end{equation}
 For example, in the Schubert problem  $\T^2\cdot\I^2=2$ on $\Gr(2,5)$, the first two partitions do not satisfy (d) for
 $i=1$, as when $\mu=\nu=\T$,  $\mu_1+\nu_1=2+2\not\leq 3$---this is the source of the reduction presented in
 Example~\ref{Ex:first_reduction}.
 A Schubert problem $\blambda$ in $\Gr(k,n)$ is \demph{essential} if condition~\eqref{Eq:essential} holds for every
 $\mu,\nu$ from  $\blambda$, so that it is not equivalent to a Schubert problem on a smaller Grassmannian
   in this way\footnote{In~\cite{SW_double} the term `reduced' is used instead of `essential'.}. 
 On $\Gr(4,9)$ only $31,806$ of the $81,533$ nontrivial Schubert problems are essential.
%
%On Bloc601k   to compute the problems with at least four solutions, including the non-essential ones
%Elapsed  Time := 1398.68  seconds
%NumberofProblems:=47909
%
% On Mac laptop all essential problems with at least 2 solutions
%Elapsed  Time := 1430.30  seconds
%NumberofProblems:=31806
%
% On Mac laptop all essential problems with at least 4 solutions
%Elapsed  Time := 1380.33  seconds
%NumberofProblems:=30784

%%%%%%%%%%%%%%%%%%%%%%%%%%%%%%%%%%%%%%%%%%%%%%%%%%%%%%%%%%%%%%%%%%%%%%%%%%%%%%%%%
\subsection{Galois groups of branched covers}\label{SS:Galois}
Suppose that $\pi\colon X\to Y$ is a dominant map of complex irreducible varieties of the same dimension.
Then there is a positive integer $d$ and a dense open subset $U\subset Y$ consisting of regular
values $u$ of $\pi$ whose fiber $\pi^{-1}(u)$ consists of $d$ reduced points.
Thus, over $U$, the map $\pi\colon \pi^{-1}(U)\to U$ is a covering space of degree $d$.
Call $\pi\colon X\to Y$ a \demph{branched cover} of degree $d$.

Since $\pi(X)$ is dense in $Y$ and both varieties are irreducible, there is an inclusion of function fields
$\pi^*\colon\CC(Y)\hookrightarrow\CC(X)$. 
As the map $\pi$ has degree $d$, the field extension $\CC(X)/\pi^*\CC(Y)$ has degree $d$.
Let \defcolor{$\Gal_{\pi}$} be the Galois group of the Galois closure of $\CC(X)$ over $\pi^*\CC(Y)$.
Harris~\cite{Ha79} showed that
this is also the monodromy group of the covering space $\pi\colon \pi^{-1}(U)\to U$.
The monodromy group acts transitively on a fiber.
Indeed, as $X$ is irreducible, $\pi^{-1}(U)$ is path-connected.
Then for any two points $x,x'$ in a given fiber $\pi^{-1}(u)$ over a point $u\in U$, there is a path in $\pi^{-1}(U)$
connecting $x$ to $x'$.
Its image in $U$ is a loop based at $u$ whose corresponding monodromy permutation sends $x$ to $x'$.

Branched covers $\pi\colon X\to Y$ are common in enumerative geometry and are the source of Jordan's observation that
enumerative problems have Galois groups~\cite{J1870}. 
If $Z\subset Y$ is an irreducible subvariety that meets the regular locus $U$ of $\pi$, then $\pi^{-1}(Z)\to Z$ (or rather
the closure of $\pi^{-1}(Z\cap U)$ in $X$) is a branched cover and its monodromy group is a subgroup of $\Gal_{\pi}$.
This holds even if $\pi^{-1}(Z)$ is reducible.

The realization that the algebraic Galois group is the geometric monodromy group goes back at least to
Hermite~\cite{Hermite}, with a modern treatment given in~\cite{Ha79}.
Vakil~\cite[Section~3.5]{Va06b} gives a purely algebraic construction of 
a group $\Mon_\pi$ of a branched cover over any field.
When the field is $\CC$, this is equal to the monodromy group of the branched cover.
When $\pi\colon X\to Y$ is defined over $\QQ$, so that the complex varieties are obtained by extending scalars from $\QQ$
to $\CC$, then Vakil's construction shows that $\Mon_\pi$ is equal to the Galois group for
$X(\overline{\QQ})\to Y(\overline{\QQ})$, where $\overline{\QQ}$ is the algebraic closure of $\QQ$.
This version of Galois equals monodromy is also explained in~\cite[Sect.~1.1]{GGEGA}.

We may similarly define a Galois group $\Gal_\pi(\QQ)$ using the Galois 
closure $L$ of $\QQ(X)$ over $\pi^*(\QQ(Y))$.
If $K=L\cap\CC$, then Vakil's construction shows that $\Gal_\pi=\Gal(L/K(Y))$, so that the function field
  $K(Y)=\QQ(Y)\otimes K$
  is the $\Gal_\pi$-fixed field of $L$, which shows that $\Gal_\pi$ is a normal subgroup of $\Gal_\pi(\QQ)$, and
  it equals $\Gal_\pi(\QQ)$ if and only if $L$ does not contain scalars, in that  $\QQ= L\cap\CC$.
For example, when $Y=\A^1$ and $X=\calV(x^3-y)$, then $\Gal_\pi=\ZZ/3\ZZ$, but $\Gal_\pi(\QQ)=S_3$, the
difference being that the primitive third roots of unity lie in $\CC$ and not in $\QQ$.
%
%  This needs an extensive explanation.  It is standard field theory, though.
%
This distinction will be important in Section~\ref{SS:Frobnenius}. %which distinction??

The \demph{Schubert Galois group} \demph{$\Gal_\blambda$} of a Schubert problem $\blambda$ is
the Galois group of the branched cover $\pi\colon\calX_\blambda\to(\Fln)^s$ defined in~\eqref{Eq:TotalSpace}.
Choosing a regular value $\calFdot\in(\Fln)^s$ so that $\pi^{-1}(\calFdot)$  consists of $d(\blambda)$ reduced points, 
the group $\Gal_\blambda$ is a transitive subgroup of the symmetric group $S_{d(\blambda)}$. 
The Grassmannian and flag manifold are rational varieties defined over $\QQ$, so we also have the group
$\Gal_\blambda(\QQ)$ which contains $\Gal_\blambda$ as a normal subgroup.
Based on the results of Sections~\ref{SS:Frobnenius} and~\ref{S:geometry}, we make the following conjecture.

%%%%%%%%%%%%%%%%%%%%%%%%%%%%%%%%%%%%%%%%%%%%%%%%%%%%%%%%%%%%%%%%%%%%%%%%%%%%%%%%%
\begin{conjecture}\label{Conj:one}
  For any Schubert problem $\blambda$, $\Gal_{\blambda}(\QQ)=\Gal_{\blambda}$.
\end{conjecture}
%%%%%%%%%%%%%%%%%%%%%%%%%%%%%%%%%%%%%%%%%%%%%%%%%%%%%%%%%%%%%%%%%%%%%%%%%%%%%%%%%

The following simple proposition, in particular constructions appearing in its proof, will be
important to establish the results of Section~\ref{S:geometry}.

%%%%%%%%%%%%%%%%%%%%%%%%%%%%%%%%%%%%%%%%%%%%%%%%%%%%%%%%%%%%%%%%%%%%%%%%%%%%%%%%%
\begin{proposition}\label{P:Two}
  The Galois groups of\/ $\I^4=2$ in $\Gr(2,4)$, as well as
  $\raisebox{-3.5pt}{\TI}\cdot \I^3=2$, $\raisebox{-3.5pt}{\II}\cdot\I^4=2$,
  $\T^2\cdot\I^2=2$, $\T\cdot\I^4=3$, and $\I^6=5$ in $\Gr(2,5)$
  are all the corresponding symmetric group. 
\end{proposition}
%%%%%%%%%%%%%%%%%%%%%%%%%%%%%%%%%%%%%%%%%%%%%%%%%%%%%%%%%%%%%%%%%%%%%%%%%%%%%%%%%

These are all the Schubert problems $\bnu$ in $\Gr(2,4)$ and $\Gr(2,5)$ with $d(\bnu)>1$.

%%%%%%%%%%%%%%%%%%%%%%%%%%%%%%%%%%%%%%%%%%%%%%%%%%%%%%%%%%%%%%%%%%%%%%%%%%%%%%%%%
\begin{proof}
  All Schubert Galois groups in $\Gr(2,n)$ contain the alternating group~\cite{BdCS}.
  We show that each contains a transposition, which will complete the proof.
  As these are monodromy groups of branched covers of degree $d$, any sufficiently small loop in the base around a point
  whose fiber consists of $d{-}1$ distinct points, so that one point is a double point,
  induces a simple transposition~\cite[Section~II.3]{Ha79}.
  For the branched cover of a Schubert problem 
  %
  % Any sufficiently small loop in the base around such a simple double point will induce a
  % simple transposition.  We included some citations to the literature on this point.
  %
  %  Harris/Byrnes/Esterov, all use this.
  %
  %  Harris has this statement as a Lemma in page 698
  %
 $\calX_\blambda\to(\Fln)^s$, 
 it will suffice to show there is a choice of flags $\calFdot$ whose corresponding instance $\Omega_{\blambda}\calFdot$ 
 has a unique double solution.
 We describe a configuration of flags with one flag lying in a pencil such that the pencil of
 instances contains an instance with a double solution.

 For $\I^4=2$ in $\Gr(2,4)$, note that only the two-dimensional subspaces of the flags matter.
 As explained following Figure~\ref{F:4lines}, if we start with a configuration of four flags (lines
 $\ell_1\dotsc,\ell_4$ in Figure~\ref{F:4lines}) and then let $\ell_4$ move in a general pencil, the two members of
 that pencil that are tangent to the quadric each give an instance with a double solution.

 The Schubert problems in $\Gr(2,5)$ with two solutions are not essential in the sense of~\cite[Proposition~6]{SW_double}:
 For each, one of (b), (c), or (d) of~\eqref{Eq:essential} does not hold, and  
 they reduce to the only non-trivial Schubert problem $\I^4=2$ in a $\Gr(2,4)$.
 We showed this for  $\T^2\cdot\I^2=2$ in Example~\ref{Ex:first_reduction}.
 Thus, their Galois groups are isomorphic to that of $\I^4=2$ by~\cite[Proposition~6]{SW_double}.
 This completes the proof in these cases.

 For  $\T\cdot\I^4=3$, the relevant subspaces in the flags $\Fdot^1,\dotsc,\Fdot^5$ are
 $F^1_2,F^2_3,\dotsc,F^5_3$, and the Schubert problem asks for the
 2-planes $H$ that have a nontrivial intersection with each.
 If the first two subspaces are in a degenerate configuration where $\ell:=F^1_2\cap F^2_3$ has dimension 1 so that
 $\Lambda:=\langle F^1_2,F^2_3\rangle$ has dimension 4, then
 \[
   \Omega_{\sT}F^1_2\ \cap\ \Omega_{\sI} F^2_3\ =\ \Omega_{\sTh}\ell \ \cup\ \Omega_{\sTI}(F^1_2\subset\Lambda)\,,
 \]
 %
 %  We will add a remark on this use of terminology.  Also refer to Schubert and to G2n: Brooks, Abraham, and Frank
 %
 so that the Schubert problem breaks into two subproblems.
 The one involving $\Omega_{\sTh}\ell$ has the unique solution
 $\langle \ell,F^3_3\rangle\cap\langle \ell,F_3^4\rangle\cap\langle \ell,F_3^5\rangle$, while the one involving
 $\Omega_{\sTI}(F^1_2\subset\Lambda)$ is an instance of $\raisebox{-2pt}{\mTI}\cdot\I^3=2$.
 This is equivalent to an instance of $\I^4=2$ on $\Gr(2,\Lambda)\simeq\Gr(2,4)$, namely, the $2$-planes in $\Lambda$ that meet 
 each of the four $2$-planes $F^1_2$, $F^3_3\cap\Lambda$, $F^4_3\cap\Lambda$, $F^5_3\cap\Lambda$.
 Because all instances of $\I^4=2$ on $\Gr(2,\Lambda)$ may occur in this way, there is an instance with a double solution, which
 proves this case.
 %
 % This is very important and related to the double solution.
 %  We need to make this explicit, perhaps by a lemma or a strong remark. 
 %  Done with a strong remark.
 %

 Finally, the Schubert problem $\I^6=5$ asks for the 2-planes $H$ that meet each of six 3-planes $F^1_3,\dotsc,F^6_3$.
 Suppose that $F^1_3$ and $F^2_3$ are in degenerate position so that $\defcolor{L^{12}}:=F^1_3\cap F^2_3$ is a 2-plane and
 $\defcolor{\Lambda^{12}}:=\langle F^1_3, F^2_3\rangle$ is a 4-plane.
 Then $\Omega_{\sI}\Fdot^1\cap\Omega_{\sI}\Fdot^2=\Omega_{\sII}\Lambda^{12}\cup\Omega_{\sT}L^{12}$.
 Supposing also that $F^3_3$ and $F^4_3$ are in a similar degenerate position, and we define  $\defcolor{L^{34}}$ and
 $\defcolor{\Lambda^{34}}$ similarly, then the Schubert problem becomes
 \[
   % The referee would like to explain: \Omega_{\sII}\Lambda^{12}
   % \comm{p7, l 29}{This notation should be explained}
   %
   % We now had explained the notation when we defined Schubert varieties.
   %
   (\Omega_{\sII}\Lambda^{12}\: \cup\: \Omega_{\sT}L^{12})\ \cap\ 
    (\Omega_{\sII}\Lambda^{34}\: \cup\: \Omega_{\sT}L^{34})\ \cap\ 
    \Omega_{\sI}F^5_3\ \cap\ \Omega_{\sI}F^6_3\,.
 \]
 This gives four subproblems, which are the intersection of $ \Omega_{\sI}F^5_3 \cap \Omega_{\sI}F^6_3$ with one of
 \[
    \Omega_{\sII}\Lambda^{12}\cap\Omega_{\sII}\Lambda^{34}\,,\ 
    \Omega_{\sII}\Lambda^{12}\cap\Omega_{\sT}L^{34}\,,\ 
    \Omega_{\sT}L^{12}\cap\Omega_{\sII}\Lambda^{34}\,,\ \mbox{ or }\ 
    \Omega_{\sT}L^{12}\cap\Omega_{\sT}L^{34}\,.
  \]
  %
  %  The referee noted a typo above.  I fixed it
  %
  The first three each have a unique solution---for example the first intersection is
  \[
    \Omega_{\sII}\Lambda^{12}\cap\Omega_{\sII}\Lambda^{34}\cap\Omega_{\sI}F^5_3 \cap \Omega_{\sI}F^6_3\ =\ \{H\}\,,
   \]
  where $H$ is the span of the two one-dimensional linear subspaces 
  \[
    F^5_3\cap\Lambda^{12}\cap\Lambda^{34}
    \qquad\mbox{and}\qquad
    F^6_3\cap\Lambda^{12}\cap\Lambda^{34}\,.
  \]
  The last subproblem gives an instance of $\T^2\cdot\I^2=2$ in $\Gr(2,5)$.
  As in that case, letting one of $F^5_3$ or $F^6_3$ move in a pencil completes the proof in this case.
 The special positions of the flags used here give the generically transverse intersections that are claimed, as shown
 in~\cite{So97} (and also used by Schubert in~\cite{Sch1886c}.)
\end{proof}
%%%%%%%%%%%%%%%%%%%%%%%%%%%%%%%%%%%%%%%%%%%%%%%%%%%%%%%%%%%%%%%%%%%%%%%%%%%%%%%%%

%%%%%%%%%%%%%%%%%%%%%%%%%%%%%%%%%%%%%%%%%%%%%%%%%%%%%%%%%%%%%%%%%%%%%%%%%%%%%%%%%
\begin{remark}\label{R:Galois_Comment}
  For each of the five nontrivial Schubert problems $\bnu$ in $\Gr(2,5)$, the proof of Proposition~\ref{P:Two} exhibits
  a choice $\calFdot$ of flags such that $\Omega_\bnu\calFdot$ consists of $d(\bnu)$ points.
  Furthermore, the proof showed if we let one of the subspace $F^i_3$ corresponding to a condition $\I$ move in a
  general pencil while fixing the other flags, then that pencil of instances contains an instance $\calFdot'$ with a unique
  double point.
  Consequently, monodromy in that pencil around  $\calFdot'$ is a simple transposition in the Schubert Galois group
  $\Gal_\bnu$.
  This uses the observation that the linear algebra constructions of intersection with a subspace and image under a linear
  map preserves linear algebraic objects, such as pencils of subspaces.

  A consequence is that the same holds for any choice of general flags $\calGdot$ with $\Omega_\bnu\calGdot$
  consisting of $d(\bnu)$ points.
  That is, if a subspace $G^i_3$ corresponding to a condition $\I$ moves in a general pencil, then there will be an
  instance  $\calGdot'$ in that pencil with a unique double point and therefore monodromy in that pencil around
  $\calGdot'$ is a simple transposition.
  This may perhaps also be understood as the general singular configuration of flags $\calFdot$ gives one point of multiplicity two in
  $\Omega_{\bnu}\calFdot$ and $d(\bnu){-}2$ simple points.
  \hfill{$\diamond$}
\end{remark}
%%%%%%%%%%%%%%%%%%%%%%%%%%%%%%%%%%%%%%%%%%%%%%%%%%%%%%%%%%%%%%%%%%%%%%%%%%%%%%%%%

%%%%%%%%%%%%%%%%%%%%%%%%%%%%%%%%%%%%%%%%%%%%%%%%%%%%%%%%%%%%%%%%%%%%%%%%%%%%%%%%%
\subsection{Permutation groups}\label{SS:Permutation}
Fix a positive integer $d$.
A permutation group of degree $d$ is a subgroup of $S_d$, that is, it is a group $G$ together with a faithful action on
$\defcolor{[d]}:=\{1,\dotsc,d\}$.
Write the image of $a\in[d]$ under $g\in G$ as $g(a)$.
(The set $[d]$ may be replaced by any set of cardinality $d$.)
This permutation group is \demph{transitive} if for all $i,j\in[d]$, there is a $g\in G$ with $g(i)=j$.
More generally, for any $1\leq t\leq d$, a permutation group $G$ is \demph{$t$-transitive} if for any distinct 
$i_1,\dotsc,i_t\in [d]$ and distinct $j_1,\dotsc,j_t\in[d]$, there is a $g\in G$ with $g(i_m)=j_m$ for $m=1,\dotsc,t$.
The full symmetric group $S_d$ is $d$-transitive and its alternating subgroup $A_d$ is $(d{-}2)$-transitive.
These are the only highly transitive permutation groups. 
This is explained in~\cite[Section~4]{Cameron} and summarized in the following proposition, which follows from the
O'Nan-Scott Theorem~\cite{OS} and the classification of finite simple groups.\medskip

%%%%%%%%%%%%%%%%%%%%%%%%%%%%%%%%%%%%%%%%%%%%%%%%%%%%%%%%%%%%%%%%%%%%%%%%%%%%%%%%%
\noindent{\bf Theorem~4.11} of~\cite{Cameron}{\bf .}{\it \ 
  The only $6$-transitive groups are symmetric groups $S_n$ and alternating groups
  $A_{n+2}$ for $n\geq 6$.
 The only $4$-transitive groups are the appropriate symmetric and alternating groups, and the Mathieu groups
 $M_{11}$, $M_{12}$, $M_{23}$, and $M_{24}$.
 All $2$-transitive permutation groups are  known.}\medskip
%%%%%%%%%%%%%%%%%%%%%%%%%%%%%%%%%%%%%%%%%%%%%%%%%%%%%%%%%%%%%%%%%%%%%%%%%%%%%%%%%

Tables~7.3 and 7.4 in~\cite{Cameron} list the 2-transitive permutation groups.
Let $G$ be a transitive permutation group of degree $d$.
A \demph{block} is a subset $S$ of $[d]$ such that for every $g\in G$ either $g(S)=S$ or $g(S)\cap S=\emptyset$.
The orbits of a block generate a partition of $[d]$ into blocks.
The group $G$ is \demph{primitive} if its only blocks are $[d]$ or singletons; otherwise it is \demph{imprimitive}.
Any 2-transitive permutation group is primitive, and primitive permutation groups that are not symmetric or
alternating are rare---the set of degrees $d$ of such nontrivial primitive permutation
groups has density zero in the natural numbers~\cite[Section~4.9]{Cameron}. 

Let $[d]\times[f]$ be the set of ordered pairs $\{(a,b)\mid a\in[d]\,,\ b\in[f]\}$ and 
$\pi\colon[d]\times[f]\to [f]$ the projection map.
The \demph{wreath product  $S_d\wr S_f$}  is the symmetry group of the fibration $\pi\colon[d]\times[f]\to [f]$.
That is, it is the largest permutation group acting on the set $[d]\times[f]$ that also preserves the partition given
by the fibers of $\pi$; thus, it also has an induced action on $[f]$.  
The action on $[d]\times[f]$ is imprimitive when both $d$ and $f$ are at least 2.
As an abstract group, it is the  semidirect product $(S_d)^f\rtimes S_f$.
Its elements are ordered pairs $((g_1,\dotsc,g_f),h)$ where $g_1,\dotsc,g_f\in S_d$ and $h\in S_f$, with product defined by 
\[
    ((g_1,\dotsc,g_f),h)((\gamma_1,\dotsc,\gamma_f),k)\ :=\ 
    ((g_1\gamma_{h^{-1}(1)},\dotsc,g_f\gamma_{h^{-1}(f)}),hk)\,.
\]
Its action on $[d]\times[f]$ of ordered pairs is as follows,
\[
    ((g_1,\dotsc,g_f),h)(a,b)\ :=\  ( g_{h(b)}(a), h(b) )\,.
\]
For permutation groups $G$ and $H$ of degrees $d$ and $f$, respectively, their
wreath product $G\wr H$ is the obvious subgroup of $S_d\wr S_f$.
Every imprimitive permutation group is a subgroup of a wreath product of symmetric
groups~\cite[Theorem~1.8]{Cameron}. 

The symmetric group $S_d$ is the only $2$-transitive permutation group of degree $d$ that contains a $2$-cycle.
Jordan gave a useful generalization.

%%%%%%%%%%%%%%%%%%%%%%%%%%%%%%%%%%%%%%%%%%%%%%%%%%%%%%%%%%%%%%%%%%%%%%%%%%%%%%%%%
\begin{proposition}[Jordan~\cite{J1870}]
 \label{P:Jordan}
 If $G\subset S_d$ is primitive and contains a $p$-cycle for some prime number $p<d{-}2$, then $G$
 contains the alternating group $A_d$.
\end{proposition}
%%%%%%%%%%%%%%%%%%%%%%%%%%%%%%%%%%%%%%%%%%%%%%%%%%%%%%%%%%%%%%%%%%%%%%%%%%%%%%%%%

%%%%%%%%%%%%%%%%%%%%%%%%%%%%%%%%%%%%%%%%%%%%%%%%%%%%%%%%%%%%%%%%%%%%%%%%%%%%%%%%%
\section{Lower bounds for Schubert Galois groups in $\Gr(4,9)$}\label{S:computations}

We use two methods to compute lower bounds for all Schubert Galois groups in $\Gr(4,9)$.
The first is a recursive criterion due to Vakil~\cite{Va06b} based on his geometric Littlewood-Richardson
rule~\cite{Va06a}.
When the criterion holds for a Schubert problem $\blambda$, the group $\Gal_\blambda$ is at least alternating.
The second method computes cycle types of elements in Schubert Galois group $\Gal_\blambda(\QQ)$ over
$\QQ$. 
%
%  The next part is cryptic.  Make it less so.
%
Assuming Conjecture~\ref{Conj:one}, that $\Gal_\blambda=\Gal_\blambda(\QQ)$, then in every Schubert problem $\blambda$ in
$\Gr(4,9)$ that can be computed, the lower bound from the second method shows that $\Gal_\blambda$ is as large as possible.
  It is the symmetric group $S_{d(\blambda)}$ when $\blambda$ is at least alternating, and in the other cases, it is a
  wreath product of symmetric groups which is the upper bound as determined in Section~\ref{S:geometry}.

%%%%%%%%%%%%%%%%%%%%%%%%%%%%%%%%%%%%%%%%%%%%%%%%%%%%%%%%%%%%%%%%%%%%%%%%%%%%%%%%%
\subsection{Vakil's Method}\label{SS:Vakil}
This exploits the classical method of special position in enumerative geometry to obtain information about Galois 
groups. 
Coupled with Vakil's geometric Littlewood-Richardson rule, it provides a remarkably effective method to show that nearly
all Schubert Galois groups in $\Gr(4,9)$ are at least alternating.
We describe Vakil's method, sketch the algorithm, and discuss the result of using two independent implementations to test
all essential Schubert problems in $\Gr(4,9)$. 

Suppose that $\pi\colon X\to Y$ is a branched cover of degree $d$ with regular locus $U\subset Y$.
We follow Vakil~\cite[Section~3.4]{Va06b}, who described how the monodromy action over an irreducible subvariety
$Z\subset Y$ that meets $U$ affects the Galois group $\Gal_{\pi}$.
Suppose that $Z\hookrightarrow Y$ is the closed embedding of a Cartier divisor that meets $U$, where $Y$ is smooth in
codimension one along $Z$.
%
% The referee does not like this.  Maybe we should make it clear that we are quoting Vakil, and following his notation.
%   This is Section 3.4 of Vakil
%
Let \defcolor{$W$} be the closure in $X$ of $\pi^{-1}(Z\cap U)$, and consider the  diagram
 \begin{equation}\label{Eq:fiber_diagram}
  \raisebox{-20pt}{
  \begin{picture}(60,45)
   \put(5,35){$W$} \put(18,35){$\lhra$} \put(48,35){$X$}
   \put(0,20){$p$}\put(10,32){\vector(0,-1){20}}
      \put(52,32){\vector(0,-1){20}}\put(55,20){$\pi$}
   \put(5, 0){$Z$} \put(18, 0){$\lhra$} \put(48, 0){$Y$}
  \end{picture}
  }
 \end{equation}
where $p\colon W\to Z$ has degree $d$.
When $W$ is irreducible or has two components Vakil shows that the following holds.
\begin{enumerate}
 \item[(a)] If $W$ is irreducible, then the monodromy group $\Gal_{p}$ is a subgroup of $\Gal_{\pi}$.
 \item[(b)] If $W=W_1\cup W_2$ with each $p_i\colon W_i\to Z$ a branched cover of degree $d_i$, then the monodromy group
   for $p$ is a subgroup of both $\Gal_{\pi}$ and of the product $\Gal_{p_1}\times \Gal_{p_2}$, and it
   maps surjectively onto each factor $\Gal_{p_i}$.
   %
   %  The monodromy subgroup is a subgroup of $(\Gal_{p_1}\times \Gal_{p_2}) \cap \Gal_{\pi}$
   %
\end{enumerate}
In the above situation, Vakil gave criteria for deducing that $\Gal_{\pi}$ is at least
alternating, based on purely group-theoretic arguments including Goursat's Lemma.\medskip

%%%%%%%%%%%%%%%%%%%%%%%%%%%%%%%%%%%%%%%%%%%%%%%%%%%%%%%%%%%%%%%%%%%%%%%%%%%%%%%%%
%
%  Give a precise citation for these results
%
\noindent{\bf Vakil's Criteria.} (Theorem 3.2 and Remark 3.4 in~\cite{Va06b})
{\it
  Suppose that we have a diagram as in~\eqref{Eq:fiber_diagram}.
  The Galois group  $\Gal_{\pi}$ is at least alternating if one of the following holds.
\begin{enumerate}
\item[(i)] We are in Case (a) and $\Gal_{p}$ is at least alternating.
  %
  %  This ^^^^ is misstated
  %
\item[(ii)] We are in Case (b),  $\Gal_{p_1}$ and $\Gal_{p_2}$ are at least alternating,
         and either $d_1\neq d_2$ or $d_1 = d_2 = 1$.
\end{enumerate}
}\medskip
%%%%%%%%%%%%%%%%%%%%%%%%%%%%%%%%%%%%%%%%%%%%%%%%%%%%%%%%%%%%%%%%%%%%%%%%%%%%%%%%%

%%%%%%%%%%%%%%%%%%%%%%%%%%%%%%%%%%%%%%%%%%%%%%%%%%%%%%%%%%%%%%%%%%%%%%%%%%%%%%%%%
%
%   Simply skip the description of this.  Say as little as possible/necessary
%
%
The geometric Littlewood-Richardson rule is presented in Vakil's original paper~\cite{Va06a} and also in some detail from a
different perspective in~\cite{LMSVV}.
Its implications for Schubert Galois groups are explained in Section~3 of~\cite{Va06b}. 
We refer the reader to those sources for details, stating only the consequences of Vakil's construction.

Given a Schubert problem $\blambda$, Vakil~\cite{Va06a} constructs a \demph{checkerboard tournament
    $\calT_{\blambda}$}, which is a rooted tree encoding the structure of the branched cover $\calX_\blambda\to(\Fln)^s$ as some 
  subspaces in the flags degenerate in a particular way.
  The nodes of $\calT_\blambda$ are certain \demph{checkerboards} $\cbd$ and each encodes a branched cover
  $X_{\cbd}\to Y_{\cbd}$ with $Y_{\cbd}$ smooth resulting from a degeneration of $\calX_\blambda\to(\Fln)^s$.
  The children of a node represent further degenerations, obtained by restricting  $X_{\cbd}\to Y_{\cbd}$  to a Cartier
  divisor $Z$ of $Y_{\cbd}$ as in~\eqref{Eq:fiber_diagram} and taking the irreducible components of $W$.
  Each node $\cbd$ has either one or two children, denoted $\cbd'$ and $\cbd''$.

To state the main results of~\cite{Va06a,Va06b} which concern us, for a node $\cbd$ of $\calT_{\blambda}$, let
  \defcolor{$\delta(\cbd)$} be the number of leaves above $\cbd$.
  This satisfies the recursion that if $\cbd$ is a leaf, then $\delta(\cbd)=1$,
  if $\cbd$ has a single child $\cbd'$, then $\delta(\cbd)=\delta(\cbd')$, and that 
  if $\cbd$ has two children $\cbd'$ and $\cbd''$, then $\delta(\cbd)=\delta(\cbd')+\delta(\cbd'')$.

%%%%%%%%%%%%%%%%%%%%%%%%%%%%%%%%%%%%%%%%%%%%%%%%%%%%%%%%%%%%%%%%%%%%%%%%%%%%%%%%%
\begin{proposition}\label{P:Vakil}
 Let $\blambda$ be a Schubert problem in $\Gr(k,n)$ and construct $\calT_{\blambda}$ as above.
 \begin{enumerate}
  \item If $\cbd$ is the root of $\calT_{\blambda}$, then $\delta(\cbd)=d(\blambda)$.
   %
   %  Every tree has a unique root
   %
  \item If for every node $\cbd$ of $\calT_{\blambda}$ with two children $\cbd'$ and $\cbd''$, either
    $\delta(\cbd')\neq\delta(\cbd'')$ or
   %
   %  This condition is not symmetric in the two children.  Fix it.
   %
   $\min\{\delta(\cbd'),\delta(\cbd'')\}=1$, then the Schubert Galois group $\Gal_{\blambda}$ is at least alternating.
 \end{enumerate}
\end{proposition}
%%%%%%%%%%%%%%%%%%%%%%%%%%%%%%%%%%%%%%%%%%%%%%%%%%%%%%%%%%%%%%%%%%%%%%%%%%%%%%%%%

%\input{Vakil_Proof.tex}
%%%%%%%%%%%%%%%%%%%%%%%%%%%%%%%%%%%%%%%%%%%%%%%%%%%%%%%%%%%%%%%%%%%%%%%%%%%%%%%%%%%%%%%%%%%%%%%%%%%%
%%%%%%%%%%%%%%%%%%%%%%%%%%%%%%%%%%%%%%%%%%%%%%%%%%%%%%%%%%%%%%%%%%%%%%%%%%%%%%%%%
\noindent{\it Sketch of proof of Proposition~\ref{P:Vakil}.}
 The checkerboard game $\calT_{\lambda,\mu}$ 
 encodes a sequence of `bend-and-sometimes-break' flat degenerations of the intersection
 $\Omega_\lambda\Fdot\cap \Omega_\mu\Mdot$ of Schubert varieties into a union of Schubert varieties $\Omega_\nu\Fdot$ for
 $\nu$ a leaf of $\calT_{\lambda,\mu}$. 
 These follow a sequence of $\binom{n}{2}$ specializations of the pair $(\Fdot,\Mdot)$ of flags, starting at the root
 with the flags in linear general position and ending at the leaves with the flags coinciding.
 Let \defcolor{$Y_r$} be the set of pairs of flags in the $r$th special position.
 It is an orbit of $GL(n,\CC)$ on $\Fln\times\Fln$.
 Furthermore, $Y_{r+1}$ is a subset of the closure $\overline{Y_r}$ of $Y_r$ with 
 $\overline{Y_{r+1}}\hookrightarrow\overline{Y_r}$ the inclusion of a Cartier divisor and 
 $\overline{Y_r}$ is smooth in codimension one along $Y_{r+1}$.

 The checkerboard $\cbd$ labeling a node in $\calT_{\lambda,\mu}$ at height $r$ encodes a position of a $k$-plane with
 respect to a pair of flags $(\Fdot,\Mdot)\in Y_r$.
 For $(F,M)\in Y_r$, the set $X_{\cbd}(\Fdot,\Mdot)$ of all such $k$-planes
 forms a \demph{checkerboard variety}, which is irreducible.
 The family $\calX_{\cbd}\to Y_r$ whose fiber over $(\Fdot,\Mdot)\in Y_r$ is the checkerboard variety 
 $X_{\cbd}(\Fdot,\Mdot)$ has the following property.
 Let  $\overline{\calX_{\cbd}}\to \overline{Y_r}$ be its closure in $\Gr(k,n)\times\Fln\times\Fln$ and 
 $W\to Y_{r+1}$ the restriction of this closure to $Y_{r+1}$ as with~\eqref{Eq:fiber_diagram}.
 We have that $\calX_{\cbd}\cup W$ is flat over $Y_r\cup Y_{r+1}$, and that 
 $W$ has either one or two components, depending upon whether or not $\cbd$ has one or two children in
 $\calT_{\lambda,\mu}$, and these components are families $\calX_{\cbd'}$, where $\cbd'$ is a child of $\cbd$.

 This collection of families of degenerations of $\Omega_\lambda\Fdot\cap \Omega_\mu\Mdot$ is the 
 \demph{geometric Littlewood-Richardson rule}.
 This is because if \defcolor{$c^\nu_{\lambda,\mu}$} is the number of leaves of $\calT_{\lambda,\mu}$ labeled $\nu$, then
 the properties of these families show that
 \begin{equation}\label{Eq:LRR}
   [\Omega_\lambda\Fdot\cap \Omega_\mu\Mdot]\ =\ \sum_{\nu} c^\nu_{\lambda,\mu} [\Omega_\nu\Fdot]\,,
 \end{equation}
 where, for $V\subset \Gr(k,n)$, \defcolor{$[V]$} is the cohomology class Poincar\'e dual to the fundamental cycle of $V$
 in the homology of $\Gr(k,n)$.
 As classes of Schubert varieties form a basis for cohomology and the class of an intersection of Schubert varieties in
 general position is the product of their classes, the numbers $c^\nu_{\lambda,\mu}$ are the Littlewood-Richardson numbers.

 Given a Schubert problem $\blambda=(\lambda^1,\dotsc,\lambda^s)$, we may splice the families in the geometric
 Littlewood-Richardson rule for $\calT_{\lambda^{s-1},\lambda^s}$ into the total family $\calX_{\blambda}\to(\Fln)^s$ for
 $\blambda$~\eqref{Eq:TotalSpace} as follows. 
 Given a node $\cbd$ at height $r$ in $\calT_{\lambda^{s-1},\lambda^s}$, we have a family 
 $\calX_{\cbd,\blambda}\to (\Fln)^{s-2}\times Y_r$ whose fiber over a point
 $(\Fdot^1,\dotsc,\Fdot^{s-2},\Fdot,\Mdot,)\in (\Fln)^{s-2}\times Y_r$ is
\[
    \Omega_{\lambda^1}\Fdot^1\,\bigcap\, \dotsb\,\bigcap\, \Omega_{\lambda^{s-2}}\Fdot^{s-2}
   \ \bigcap\  X_{\cbd}(\Fdot,\Mdot)\,.
\]
 A leaf $\mu$ of $\calT_{\lambda^{s-1},\lambda^s}$ corresponds to the total family for the Schubert problem
 $(\lambda^1,\dotsc,\lambda^{s-2},\mu)$ over $(\Fln)^{s-1}$.
 As before, we may splice the tree $\calT_{\lambda^{s-2},\mu}$ of deformations into these families at the leaves of
 $\calT_{\lambda^{s-1},\lambda^s}$. 
 Continuing in this fashion creates families at each node of the checkerboard tournament $\calT_{\blambda}$.

 For each node $\cbd$ of $\calT_{\blambda}$ that is at the $r$th level in a checkerboard game
 $\calT_{\lambda^{s-t},\nu}$, which is itself at the $t$-th stage in the construction of $\calT_{\blambda}$, we have a
 branched cover $\calX_{\cbd,\blambda}\to (\Fln)^{s-t-1}\times Y_r$.
 Restricting its closure to $Z=(\Fln)^{s-t-1}\times Y_{r+1}$, we obtain a diagram as in~\eqref{Eq:fiber_diagram} where $W$
 has one or two components, corresponding to the one or two children of $\cbd$.
 At a leaf of $\calT_{\blambda}$, the corresponding family has fiber $\Omega_{[n{-}k]^k}\Fdot=\{F_k\}$ over $\Fdot$, 
 by~\eqref{Eq:SchubertCondition}. 
 In particular, the families at the leaves of $\calT_{\blambda}$ all have degree $1$.

 For each node $\cbd$ of $\calT_{\blambda}$ let \defcolor{$d(\cbd)$} be the degree of the family
 $\calX_{\cbd,\blambda}\to (\Fln)^{s-t-1}\times Y_r$.
 Then these degrees $d$ satisfy the same recursion over $\calT_{\blambda}$ and initial conditions as do the numbers
 $\delta$ of leaves above a given node, which proves that $\delta(\cbd)$ is the degree of the family
 $\calX_{\cbd,\blambda}\to (\Fln)^{s-t-1}\times Y_r$  at a node $\cbd$.
 Statement (1) is this observation for the root node of $\calT_{\blambda}$.
 Statements (2) and (3) follow by Vakil's criterion, applied recursively at each node $\cbd$ of
 $\calT_{\blambda}$ with two children.\hfill$\Box$\medskip
%%%%%%%%%%%%%%%%%%%%%%%%%%%%%%%%%%%%%%%%%%%%%%%%%%%%%%%%%%%%%%%%%%%%%%%%%%%%%%%%%@

%%%%%%%%%%%%%%%%%%%%%%%%%%%%%%%%%%%%%%%%%%%%%%%%%%%%%%%%%%%%%%%%%%%%%%%%%%%%%%%%%%%%%%%%%%%%%%%%%%%%

Vakil observed that Proposition~\ref{P:Vakil} part (2) yields an algorithm to show that a Schubert Galois
group is at least alternating.
He used it for the computations reported in~\cite{Va06b}.

%%%%%%%%%%%%%%%%%%%%%%%%%%%%%%%%%%%%%%%%%%%%%%%%%%%%%%%%%%%%%%%%%%%%%%%%%%%%%%%%%
\begin{algorithm}[Vakil's Algorithm]\label{A:Vakil}
 {\rm \ 

 {\bf Input:} A Schubert problem $\blambda$.

 {\bf Output:} Either ``$\Gal_\blambda$ is at least alternating'' or 
    ``Cannot determine if $\Gal_\blambda$ is at least alternating''.

 {\bf Do:} Construct $\calT_\blambda$ and recursively determine $\delta(\cbd)$ for all nodes $\cbd$ of $\calT_\blambda$.
  If at a node $\cbd$ there are two children $\cbd'$ and $\cbd''$, such that $\delta(\cbd')\neq 1$ and
  $\delta(\cbd')=\delta(\cbd'')$, then stop and output 
 ``Cannot determine if $\Gal_\blambda$ is at least alternating''.

  If we have either  $\min\{\delta(\cbd'),\delta(\cbd'')\}=1$ or
  $\delta(\cbd')\neq\delta(\cbd'')$ for all nodes $\cbd$ of $\calT_\blambda$ with
  two children, then stop and output ``$\Gal_\blambda$ is at least alternating''.
}
\end{algorithm}
%%%%%%%%%%%%%%%%%%%%%%%%%%%%%%%%%%%%%%%%%%%%%%%%%%%%%%%%%%%%%%%%%%%%%%%%%%%%%%%%%

Vakil implemented this in a Maple script which is available from his
website\footnote{\tt http://math.stanford.edu/\~{}vakil/programs/galois}.
A revised version is available from the website accompanying this
article\footnote{{\tt http://www.math.tamu.edu/\~{}sottile/research/stories/GIVIX}}.

%%%%%%%%%%%%%%%%%%%%%%%%%%%%%%%%%%%%%%%%%%%%%%%%%%%%%%%%%%%%%%%%%%%%%%%%%%%%%%%%%
\begin{remark}\label{R:assymmetry}
  The construction of
 the checkerboard tournament $\calT_\blambda$ depends upon the ordering of the partitions in $\blambda$.
 Consequently, the outcome of Vakil's Algorithm~\ref{A:Vakil} depends on the ordering of the partitions in
 $\blambda$. \hfill$\diamond$ 
\end{remark}
%%%%%%%%%%%%%%%%%%%%%%%%%%%%%%%%%%%%%%%%%%%%%%%%%%%%%%%%%%%%%%%%%%%%%%%%%%%%%%%%%

Table~\ref{Ta:Vakil} summarizes the result of running Vakil's Maple script
on all Schubert problems in some small Grassmannians.
For each, it records the total number of Schubert problems tested $(\#)$,
the number of problems $\blambda$ for which it could not decide
if $\Gal_\blambda$ was at least alternating (??), and the
time of computation in seconds or d:h:m:s format.
On each Grassmannian, all Schubert problems $\blambda$ were tested, including those with $d(\blambda)=0$ and with
$d(\blambda)=1$, as well as all non-essential Schubert problems, except for 
$\Gr(3,10)$, $\Gr(3,11)$, and $\Gr(4,9)$ for which many non-essential problems were not tested.
%%%%%%%%%%%%%%%%%%%%%%%%%%%%%%%%%%%%%%%%%%%%%%%%%%%%%%%%%%%%%%%%%%%%%%%%%%%%%%%%%
\begin{table}[htb]
\caption{Performance of Vakil's Maple script on different $\Gr(k,n)$}\label{Ta:Vakil}
%%%%%%%%%%%%%%%%%%%%%%%%%%%%%%%%%%%%%%%%%%%%%%%%%%%%%%%%%%%%%%%%%%%%%%%%%%%%%%%%%
% Problems in \Gr(2,n)
\begin{tabular}{|c||c|c|c|c|c|c|c|c|c|}
   \multicolumn{10}{c}{Schubert Problems in $\Gr(2,n)$}\\\hline
  $n$&4&5&6&7&8&9&10&11&12\\\hline\hline
  \#&1&4&10&23&47&90&164&288&488\\\hline
   ??&-&-&-&-&-&-&-&-&-\\\hline
  sec&.01&.03&.12&.36&1.28&3.6&11.4&40&199\\\hline
\end{tabular}\medskip

%%%%%%%%%%%%%%%%%%%%%%%%%%%%%%%%%%%%%%%%%%%%%%%%%%%%%%%%%%%%%%%%%%%%%%%%%%%%%%%%%
% Problems in \Gr(3,n)
\begin{tabular}{|c||c|c|c|c|c|c|c|}
   \multicolumn{8}{c}{Schubert Problems in $\Gr(3,n)$}\\\hline
  $n$&5&6&7&8&9&10&11\\\hline\hline
  \#&4&39&270&1337&5786&22011&77305\\\hline
  ??&-&2&3&7&14&24&48\\\hline
 d:h:m:s&0.29&5.8&30.2&3:47&47:5&7:28:18&11:14:55:36\\\hline
\end{tabular}\medskip

%%%%%%%%%%%%%%%%%%%%%%%%%%%%%%%%%%%%%%%%%%%%%%%%%%%%%%%%%%%%%%%%%%%%%%%%%%%%%%%%%
% Problems in \Gr(4,n)
\begin{tabular}{|c||c|c|c|c|}
   \multicolumn{5}{c}{Schubert Problems in $\Gr(4,n)$}\\\hline
  $n$&6&7&8&9\\\hline\hline
   \#&10&270&3802&38760\\\hline
   ??&-&3&33&233\\\hline
  d:h:m:s&9.1&3:35&4:23:27&9:14:38:5\\\hline
\end{tabular}

\end{table}
%%%%%%%%%%%%%%%%%%%%%%%%%%%%%%%%%%%%%%%%%%%%%%%%%%%%%%%%%%%%%%%%%%%%%%%%%%%%%%%%%

Students Christopher Brooks and Aaron Moore worked with us to implement Vakil's algorithm in Python.
We ran the resulting software on all Schubert problems in Table~\ref{Ta:Vakil}, with nearly the same result.
Our Python script was inconclusive for 81 of the 198,099 nontrivial essentially new Schubert problems in $\Gr(3,12)$.
There was a slightly different set of Schubert problems in $\Gr(4,8)$ and $\Gr(4,9)$ for which the two implementations were
unable to determine if they were at least alternating.
The reason for this was mentioned in Remark~\ref{R:assymmetry}---the two implementations construct the checkerboard
tournament $\calT_\blambda$ differently, and this matters for those Schubert problems.

%%%%%%%%%%%%%%%%%%%%%%%%%%%%%%%%%%%%%%%%%%%%%%%%%%%%%%%%%%%%%%%%%%%%%%%%%%%%%%%%%
\subsection{Frobenius Algorithm}\label{SS:Frobnenius}
The Frobenius algorithm gives a lower bound on a Galois group over $\QQ$ by computing cycle types of
Frobenius elements.
It exploits an asymmetry in Gr\"obner basis calculations---it is much faster to
first reduce an ideal modulo a prime $p$ and then compute an eliminant than to first compute the eliminant and then reduce
modulo $p$.

%367+81

Let \defcolor{$K$} be the splitting field of an irreducible univariate polynomial $f\in\ZZ[x]$ and
\defcolor{$\calO_K$} be the ring of elements in $K$ that are integral over $\ZZ$.
Dedekind showed that for every prime $p\in\ZZ$ not dividing the discriminant of $f$, there is a unique element 
$\defcolor{\sigma_p}\in\defcolor{\Gal(K/\QQ)}$ in the Galois group of $K$ over $\QQ$ such that for every prime
\defcolor{$\varpi$} of $\calO_K$ above $p$, i.e.~$\langle\varpi\rangle \cap \ZZ = \langle p \rangle$,
and every $z\in\calO_K$, we have 
$\sigma_p(z)\equiv z^p \mod\varpi$~\cite[Theorem~4.37]{Jacobson85}.
Thus $\sigma_p$ lifts the Frobenius map $z \mapsto z^p$ on $\calO_K/p\calO_K$ to
$\calO_K$ and thus to its field of fractions $K$.
It is not necessary that $f$ be monic, but in that case, we must replace $\ZZ$ by the ring obtained by inverting the
primes which divide the leading coefficient of $f$.

The cycle type of this \demph{Frobenius element $\sigma_p$} (as a permutation of the roots of $f$) is given by the degrees
of the irreducible factors of $\defcolor{f_p}:=f\mod p$, as the irreducible factors give primes $\varpi$ above $p$.
The condition that $p$ does not divide the discriminant of $f$ is equivalent to $f_p$ being squarefree.
This gives a method to compute cycle types of elements of $\Gal(K/\QQ)$.
For a prime $p$, factor the reduction $f_p$, and if no factor is repeated, record the degrees of the factors.
This is particularly effective due to the Chebotarev Density Theorem, which asserts that Frobenius elements are uniformly
distributed for sufficiently large primes $p$.

Let $\pi\colon X\to Y$  be a branched cover of degree $d$ defined over $\QQ$ with $Y$ a rational variety.
For any regular value $y\in U(\QQ)\subset Y(\QQ)$ of $\pi$, if \defcolor{$K_y$} is the field of definition of all points in
the fiber $\pi^{-1}(y)$, then $K_y/\QQ$ is Galois and $\Gal(K_y/\QQ)$ is a subgroup of $\Gal_{\pi}(\QQ)$.  
As explained in~\cite[Chapter~VII.2]{Lang}, there is a Zariski-dense open subset $U$ of $Y$ for which the
Galois groups are equal to $\Gal_\pi(\QQ)$ for all $y\in U(\QQ)$.
For $p$ a prime and $y\in U(\QQ)$, write $\defcolor{\sigma_p(y)}$ for the Frobenius element acting on the field extension
$\Gal(K_y/\QQ)$ for the fiber above $y$.
Ekedahl~\cite{Ekedahl} showed that for a sufficiently large prime $p$, the Frobenius
elements $\sigma_p(y)\in\Gal(K_y/\QQ)\subset\Gal_{\pi}(\QQ)$ are uniformly distributed in $\Gal_{\pi}(\QQ)$ for $y\in U(\QQ)$.
Thus we may study $\Gal_{\pi}(\QQ)$ by fixing a prime $p$ and computing cycle types of Frobenius elements in
$\Gal(K_y/\QQ)$ at points  $y\in U(\QQ)$.

By Jordan's Theorem (Proposition~\ref{P:Jordan}), knowing cycle types of Frobenius elements can be used to show a Galois
group is full symmetric.

%%%%%%%%%%%%%%%%%%%%%%%%%%%%%%%%%%%%%%%%%%%%%%%%%%%%%%%%%%%%%%%%%%%%%%%%%%%%%%%%%
\begin{corollary}\label{C:IsSn}
 Suppose that $G\subset S_d$ contains a $d$-cycle, a $(d{-}1)$-cycle, and an element $\sigma$ with a unique longest cycle
 of length a prime $p<d{-}2$.
 Then $G=S_d$.
\end{corollary}
%%%%%%%%%%%%%%%%%%%%%%%%%%%%%%%%%%%%%%%%%%%%%%%%%%%%%%%%%%%%%%%%%%%%%%%%%%%%%%%%%

%%%%%%%%%%%%%%%%%%%%%%%%%%%%%%%%%%%%%%%%%%%%%%%%%%%%%%%%%%%%%%%%%%%%%%%%%%%%%%%%%
\begin{proof}
 As $G$ contains both a $d$- and a $(d{-}1)$-cycle, it is $2$-transitive (hence primitive), and it is not a subgroup of the 
 alternating group.
 As $\sigma^{(p{-}1)!}$ is a $p$-cycle (it is the inverse of the $p$-cycle in $\sigma$), Jordan's Theorem implies that
 $G=S_d$. 
\end{proof}
%%%%%%%%%%%%%%%%%%%%%%%%%%%%%%%%%%%%%%%%%%%%%%%%%%%%%%%%%%%%%%%%%%%%%%%%%%%%%%%%%

Symbolic computation, together with Frobenius elements, Ekedahl's Theorem, and
Corollary~\ref{C:IsSn} give an effective method to study Galois groups in enumerative geometry, presented in
Algorithm~\ref{A:Frobenius} below. 
Suppose that $f_1,\dotsc,f_N \in \CC[x_1,\dotsc,x_m]$ are polynomials that generate an ideal $I$ whose variety
$\defcolor{\calV(I)}\subset\CC^m$ is zero-dimensional and has degree $d$. 
The monic generator $f(x_1)$ of the ideal $I\cap\CC[x_1]$ is an \demph{eliminant} of $I$.
We use the following Shape Lemma (adapted from~\cite{BMMT}).

%%%%%%%%%%%%%%%%%%%%%%%%%%%%%%%%%%%%%%%%%%%%%%%%%%%%%%%%%%%%%%%%%%%%%%%%%%%%%%%%%
\begin{proposition}[Shape Lemma]
 Suppose that $f_1,\dotsc,f_N$ have integer coefficients and let $I$, $f$, and $d$ be as above.
 Then $f\in\QQ[x_1]$, and after clearing denominators, we may assume that $f\in\ZZ[x_1]$.
 If $f$ has degree $d$ and is square-free, then the splitting field of $f$ is the field generated by the coordinates of the 
 points in $\calV(I)\subset\CC^m$. 
\end{proposition}
%%%%%%%%%%%%%%%%%%%%%%%%%%%%%%%%%%%%%%%%%%%%%%%%%%%%%%%%%%%%%%%%%%%%%%%%%%%%%%%%%

%%%%%%%%%%%%%%%%%%%%%%%%%%%%%%%%%%%%%%%%%%%%%%%%%%%%%%%%%%%%%%%%%%%%%%%%%%%%%%%%%
\begin{proof}
  Under these hypotheses, the ideal $I$ is generated by $f$ and by
  polynomials of the form $x_i-g_i$ with $g_i\in\QQ[x_1]$ and $\deg(g_i)<d$, for $i=2,\dotsc,m$.
  Thus we have the isomorphism of rings
 \[
   \QQ[x_1,\dotsc,x_m]/I\ =\ 
   \QQ[x_1,\dotsc,x_m]/\langle f, x_i-g_i\mid i=2,\dotsc,m\rangle\ \simeq\ 
   \QQ[x_1]/\langle f\rangle\ .
 \]
 This isomorphism is induced by the inclusion $\QQ[x_1]\subset\QQ[x_1,\dotsc,x_m]$, which corresponds geometrically to the
 projection to the first coordinate. 
 This induces a map $\calV(I)\to\calV(f)$ which is an isomorphism of schemes over $\QQ$,
 and has inverse $x_1\mapsto (x_1,g_2(x_1),\dotsc, g_m(x_1))$.
 Thus coordinates of the points in  $\calV(I)\subset\CC^m$ both lie in and generate the $\QQ$-algebra generated by the roots of
 $f$, which is the splitting field of $f$.
\end{proof}
%%%%%%%%%%%%%%%%%%%%%%%%%%%%%%%%%%%%%%%%%%%%%%%%%%%%%%%%%%%%%%%%%%%%%%%%%%%%%%%%%

%%%%%%%%%%%%%%%%%%%%%%%%%%%%%%%%%%%%%%%%%%%%%%%%%%%%%%%%%%%%%%%%%%%%%%%%%%%%%%%%%
\begin{remark}
 Given polynomials $f_1,\dotsc,f_N\in\ZZ[x_1,\dotsc,x_m]$, Gr\"obner basis software (e.g.~Singular~\cite{Singular} or
 Macaulay2~\cite{M2}) can compute the dimension and degree of the variety $\calV(I)\subset\CC^m$ of the ideal $I$ they
 generate, and compute eliminants. 
 These computations take place in the ring $\QQ[x_1,\dotsc,x_m]$.
 The software may also reduce the polynomials modulo a prime $p$, and perform the same computations in
 $\FF_p[x_1,\dotsc,x_m]$ (here $\defcolor{\FF_p}:=\ZZ/p\ZZ$).
 In prime characteristic, the computations will typically be many orders of magnitude faster.
 This is because the height (number of digits) of coefficients over $\QQ$ becomes enormous, while
 the coefficients in computations over $\FF_p$ have height bounded by $\log_2 p$.

 The dimension of $\calV(I)$ in characteristic $p$ is at least its dimension in characteristic zero.
 When both have dimension zero, the degree in characteristic $p$ is at most the degree in characteristic zero, and
 the eliminant in $\FF_p[x_1,\dotsc,x_m]$ is the reduction modulo $p$ of the eliminant in
 $\QQ[x_1,\dotsc,x_m]$. 
 Finally, the software packages compute the factorization of a univariate polynomial into irreducible factors in
 $\FF_p[x_1]$.\hfill$\diamond$ 
\end{remark}
%%%%%%%%%%%%%%%%%%%%%%%%%%%%%%%%%%%%%%%%%%%%%%%%%%%%%%%%%%%%%%%%%%%%%%%%%%%%%%%%%

%%%%%%%%%%%%%%%%%%%%%%%%%%%%%%%%%%%%%%%%%%%%%%%%%%%%%%%%%%%%%%%%%%%%%%%%%%%%%%%%%
\begin{algorithm}[Frobenius Algorithm]\label{A:Frobenius}
 {\rm \ 

 {\bf Input:} A branched cover $\pi\colon X\to Y$ of degree $d\geq 8$ with $Y=\A^n$ and $X\subset\A^m\times Y$
   given by a family $f_1,\dotsc,f_N \in \ZZ[x_1,\dotsc,x_m;y_1,\dotsc,y_n]$ of integer polynomials, a positive integer
   $M$ and a prime number $p$.

 {\bf Output:} Either ``$\Gal_{\pi}(\QQ)=S_d$'' or a list $L$ of cycle types of Frobenius elements.

 {\bf Initialize:} Set ${\tt counter}:=1$, $L:=[\ ]$ (the empty list), and 
      $c_d,c_{d-1},c_{\mbox{\rm\scriptsize prime}}:=0$.

 {\bf Do:} 
    Choose a random $y\in\QQ^n$ and for each $i=1,\dotsc,N$ let $g_i\in\ZZ[x_1,\dotsc,x_m]$ be the result of clearing
    denominators and removing common divisors from the coefficients of the polynomial $f_i(x;y)$.
    Let $I$ be the ideal in $\FF_p[x_1,\dotsc,x_m]$ generated by the reductions modulo $p$ of $g_1,\dotsc,g_N$.
    If $\dim(I)>0$, then return to the start of this loop, choosing a new random $y\in\QQ^n$.

    Otherwise, compute an eliminant $g(x_1)\in \FF_p[x_1]$ of $I$ and its irreducible factorization
\[
    g(x_1)\ =\ h_1(x_1) h_2(x_1) \dotsb h_s(x_1)\,.
\]
    If $\deg(g)<d$ or if two factors coincide, then return to the start of this loop.

    Otherwise, append to $L$ the cycle type given by the degrees of the factors.

    If $s=1$ so that $g$ is irreducible, set $c_d:=1$.

    If $s=2$ and one factor has degree $d{-}1$, set $c_{d-1}:=1$.

    If the maximal degree of a factor is a prime between $d{-}2$ and $d/2$, set  
     $c_{\mbox{\rm\scriptsize prime}}:=1$.

    If $c_d\cdot c_{d-1}\cdot c_{\mbox{\rm\scriptsize prime}}=1$, then stop and output ``$\Gal_{\pi}(\QQ)=S_d$''.

   Set ${\tt counter}:={\tt counter}+1$.  
   If ${\tt counter}\geq M$, then stop and output $L$, 
     otherwise  return to start of this loop.
}
\end{algorithm}
%%%%%%%%%%%%%%%%%%%%%%%%%%%%%%%%%%%%%%%%%%%%%%%%%%%%%%%%%%%%%%%%%%%%%%%%%%%%%%%%%

%%%%%%%%%%%%%%%%%%%%%%%%%%%%%%%%%%%%%%%%%%%%%%%%%%%%%%%%%%%%%%%%%%%%%%%%%%%%%%%%%
\begin{remark}
For $d\leq 7$, a more involved, but elementary, decision procedure is used to detect if $\Gal_{\pi}(\QQ)=S_d$ using
  cycles types of Frobenius elements.\hfill$\diamond$ 
\end{remark}
%%%%%%%%%%%%%%%%%%%%%%%%%%%%%%%%%%%%%%%%%%%%%%%%%%%%%%%%%%%%%%%%%%%%%%%%%%%%%%%%%

%%%%%%%%%%%%%%%%%%%%%%%%%%%%%%%%%%%%%%%%%%%%%%%%%%%%%%%%%%%%%%%%%%%%%%%%%%%%%%%%%
\begin{proof}[Proof of correctness]
 In each iteration of the {\bf Do} loop, a random element $y\in \QQ^n$ is chosen, and the algorithm tries to compute the
 reduction modulo $p$ of the eliminant of polynomials that define the fiber $\pi^{-1}(y)$, and then its
 irreducible factorization. 
 If there is no eliminant, if it does not satisfy the Shape Lemma, or if it is not square-free, then the algorithm returns
 to the start of the loop, choosing another element of $\QQ^n$.

 Otherwise, the eliminant $f\in\ZZ[x_1]$ of the ideal generated by $g_1,\dotsc,g_N$
  has discriminant that is not divisible by $p$ and satisfies $g \equiv f\mod p$, where $g\in\FF_p[x_1]$ is the eliminant
  of $I$.
 The algorithm saves the degrees of the factors of $g$, which
 is the cycle type of a Frobenius element of $\Gal_{\pi}(\QQ)$, by Ekedahl's Theorem.
 (If $f$ is reducible, then we apply Dedekind's theory to each irreducible factor.)
 The variables $c_d,c_{d-1}$, and $c_{\mbox{\rm\scriptsize prime}}$, which record that a $d$-cycle, a $(d{-}1)$-cycle,
 or a permutation with a unique longest cycle of length a prime at most $d{-}2$ have been observed, are
 updated. 
 Once each has been observed, the algorithm terminates and returns ``$\Gal_{\pi}(\QQ)=S_d$'', which holds, by
 Corollary~\ref{C:IsSn}. 

 If after $M$ iterations, the three cycles from the hypothesis of Corollary~\ref{C:IsSn} have not been observed,
 then the algorithm terminates and returns the list $L$ of observed cycle types of Frobenius elements.
 If there have not yet been $M$ iterations, then {\tt counter} is incremented and the algorithm returns to the start of the
 loop. 
 The algorithm must terminate, and in either case, it returns correct output.
\end{proof}
%%%%%%%%%%%%%%%%%%%%%%%%%%%%%%%%%%%%%%%%%%%%%%%%%%%%%%%%%%%%%%%%%%%%%%%%%%%%%%%%%

Given a Schubert problem $\blambda$, the total family $\calX_\blambda\to(\Fln)^s$ of the Schubert problem is a branched
cover defined over $\ZZ$ of degree $d(\blambda)$.
As sketched at the end of Subsection~\ref{SS:SchubertCalculus}, this may be formulated by polynomials with integer
coefficients, and so the Frobenius algorithm may be used to study the Schubert Galois group $\Gal_\blambda(\QQ)$.

We wrote software implementing the Frobenius algorithm to study Schubert
Galois groups in small Grassmannians, particularly $\Gr(4,9)$.
That software, along with a more complete description and its output is found on our
webpage\footnote{{\tt http://www.math.tamu.edu/\~{}sottile/research/stories/GIVIX}}.
We provide a summary.

%%%%%%%%%%%%%%%%%%%%%%%%%%%%%%%%%%%%%%%%%%%%%%%%%%%%%%%%%%%%%%%%%%%%%%%%%%%%%%%%%
\subsubsection{$\Gr(2,n)$}
 As all Schubert problems in $\Gr(2,n)$, for any $n$, are at least alternating~\cite{BdCS}, we did not test any Schubert
 problems in these Grassmannians.

%%%%%%%%%%%%%%%%%%%%%%%%%%%%%%%%%%%%%%%%%%%%%%%%%%%%%%%%%%%%%%%%%%%%%%%%%%%%%%%%%
\subsubsection{$\Gr(3,n)$}
 Vakil's algorithm was inconclusive for 98 Schubert problems in $\Gr(3,n)$ for $n\leq 11$, as indicated in
 Table~\ref{Ta:Vakil}.
 Our Python implementation found a further 81 inconclusive  Schubert problems in $\Gr(3,12)$.
 Our implementation of the Frobenius algorithm showed that each of these 179 Schubert problems has
 $\Gal_\blambda(\QQ)=S_{d(\blambda)}$.

%%%%%%%%%%%%%%%%%%%%%%%%%%%%%%%%%%%%%%%%%%%%%%%%%%%%%%%%%%%%%%%%%%%%%%%%%%%%%%%%%
\begin{theorem}
  Every Schubert problem in $\Gr(3,n)$ for $n\leq 12$ has at least alternating Galois group over $\QQ$.
\end{theorem}
%%%%%%%%%%%%%%%%%%%%%%%%%%%%%%%%%%%%%%%%%%%%%%%%%%%%%%%%%%%%%%%%%%%%%%%%%%%%%%%%%

The results that Schubert Galois groups in $\Gr(3,n)$ are 2-transitive~\cite{SW_double},
and those for  $\Gr(2,n)$  are at least alternating~\cite{BdCS}, constitute
evidence for the conjecture that every Schubert Galois group in $\Gr(2,n)$ and in $\Gr(3,n)$ is a symmetric group.

%%%%%%%%%%%%%%%%%%%%%%%%%%%%%%%%%%%%%%%%%%%%%%%%%%%%%%%%%%%%%%%%%%%%%%%%%%%%%%%%%
\subsubsection{$\Gr(4,n)$}
Since $\Gr(4,6)\simeq \Gr(2,6)$ and $\Gr(4,7)\simeq \Gr(3,7)$, the next Grassmannian to study is $\Gr(4,8)$. 
Five of its 33 inconclusive Schubert problems from Table~\ref{Ta:Vakil} are non-essential, and the remaining
28 were studied in~\cite[Section~6]{SW_double}.
Two are at least alternating, and twelve more were found to be full symmetric by the
Frobenius algorithm. 
The remaining 14 are enriched, and their Galois groups over $\QQ$ were determined.

For $\Gr(4,9)$, our software tested each of the 233 inconclusive Schubert problems from Table~\ref{Ta:Vakil}.
Of these, 79 were shown to be full symmetric.
The remaining 154 appeared to be enriched.
For each of these 154, we tried to compute cycle types of 50,000 Frobenius elements, which showed that the
Galois group over $\QQ$ was likely equal to a particular permutation group, either $S_2\wr S_2$,  $S_2\wr S_3$,
$S_3\wr S_2$, $S_5\wr S_2$, or $S_4$ acting on the six equipartitions of the set $[4]$, as explained in the Introduction
  concerning Derksen's example.
Five of these, including the one with Galois group $S_4$, were not essential---they come from Schubert problems in
$\Gr(4,8)$.  
This is also described on our webpage\footnote{{\tt http://www.math.tamu.edu/\~{}sottile/research/stories/GIVIX}}.
We tabulate how many enriched Schubert problems were found for each group.
\[
\begin{tabular}{|c|c|c|c|c|}\hline
  $S_2\wr S_2$ &$S_2\wr S_3$& $S_3\wr S_2$& $S_5\wr S_2$& $S_4$\\\hline
  99 & 20  &25&9&1\\\hline
\end{tabular}
\]
We also used the Frobenius algorithm to test  26,051 Schubert problems $\blambda$ in $\Gr(4,9)$ with
$d(\blambda)$ at most 300; for all except the 154 enriched ones,
we found that $\Gal_{\blambda}(\QQ)=S_{d(\blambda)}$.
%
%   This is not documented anywhere, not even on the web page for the paper.
%

For two Schubert problems, 
$\sTh\cdot\raisebox{-2pt}{\sTII}^3\cdot\sFI=6$ and
$\sI\cdot\raisebox{-2pt}{\sTII}\cdot\sTT\cdot\sFI\cdot\sThTh=6$,
we computed nearly two million Frobenius elements with $p=10007$.
We display the results in Table~\ref{Ta:two}.
%%%%%%%%%%%%%%%%%%%%%%%%%%%%%%%%%%%%%%%%%%%%%%%%%%%%%%%%%%%%%%%%%%%%%%%%%%%%%%%%%
\begin{table}[htb]
\caption{Frequency of cycle types for two Schubert problems.}
\label{Ta:two}
%%%%%%%%%%%%%%%%%%%%%%%%%%%%%%%%%%%%%%%%
%
% Schubert Problem : [6, [3, 7, 8, 9], [4, 6, 7, 9], [4, 6, 7, 9], [4, 6, 7, 9], [2, 6, 8, 9]]
%
\[
\begin{tabular}{|c|r|r|} \hline 
  \multicolumn{3}{|c|}{$\sTh\cdot\raisebox{-2pt}{\sTII}^3\cdot\sFI=6$}\rule{0pt}{13pt}\\\hline 
  \multicolumn{3}{|c|}{Cycles found in 1998176 samples}\\\hline 
  Cycle Type & Frequency & Fraction \\
  \hline
   (6)           & 333487 & 8.01099  \\ \hline
   (3,3)         & 332912 & 7.99718  \\ \hline
   (2,4)         & 249863 & 6.00219  \\ \hline
   (2,2,2)       & 291238 & 6.99609  \\ \hline
   (1,1,4)       & 250175 & 6.00968  \\ \hline
   (1,1,2,2)     & 373920 & 8.98227  \\ \hline
   (1,1,1,1,2)   & 125098 & 3.00509  \\ \hline
   (1,1,1,1,1,1) &  41483 & 0.99650  \\ \hline
\end{tabular}
%%%%%%%%%%%%%%%%%%%%%%%%%%%%%%%%%%%%%%%% After 2,000,000 trials (1998176 successful)
%  evalf(1998176/41483);
%(8.01099-8)/.08,(7.99718-8)/.08,(6.00219-6)/.06,(6.99609-7)/.07,(6.00968-6)/.05,(8.98227-9)/.09,(3.00509-3)/.03,(.9965-1)/.01;
% 0.137375, -0.03525, 0.0365, -0.055857, 0.1936, -0.197, 0.16967, -0.350
% 
\qquad
%%%%%%%%%%%%%%%%%%%%%%%%%%%%%%%%%%%%%%%%
%
% Schubert Problem : [6, [5, 7, 8, 9], [4, 6, 7, 9], [4, 5, 8, 9], [2, 6, 8, 9], [3, 4, 8, 9]]
%
\raisebox{-7pt}{%
\begin{tabular}{|c|r|r|} \hline 
  \multicolumn{3}{|c|}{$\sI\cdot\raisebox{-2pt}{\sTII}\cdot\sTT\cdot\sFI\cdot\sThTh=6$}\rule{0pt}{13pt}\\\hline 
  \multicolumn{3}{|c|}{Cycles found in 1997950 samples}\\\hline 
  Cycle Type & Frequency & Fraction \\
  \hline
   (6) & 332926 & 11.99763  \\ \hline
   (3,3) & 111610 & 4.02208 \\ \hline
   (2,4) & 500352 & 18.03115  \\ \hline
   (2,2,2) & 167264 & 6.02768  \\ \hline
   (1,2,3) & 332565 & 11.98462  \\ \hline
   (1,1,2,2) & 248642 & 8.96030 \\ \hline
   (1,1,1,3) & 111434 & 4.01574 \\ \hline
   (1,1,1,1,2) & 165746 & 5.97298 \\ \hline
  (1,1,1,1,1,1) & 27411 & 0.98781\\ \hline
\end{tabular}}
%%%%%%%%%%%%%%%%%%%%%%%%%%%%%%%%%%%%%%%  After 2,000,000 trials, 1997950 successful
%%%  evalf(1997950/27411);
%(11.99763-12)/.12,(4.02208-4)/.04,(18.03115-18)/.18,(6.02768-6)/.06,(11.98462-12)/.12;
%(8.96030-9)/.09,(4.01574-4)/.04,(5.97298-6)/.06,(0.98781-1)/.01;
%-0.01975, 0.552, 0.1730556, 0.46133, -0.1281667, -0.441, 0.3935, -0.450333, -1.219
\]
%%%%%%%%%%%%%%%%%%%%%%%%%%%%%%%%%%%%%%%%%%%%%%%%%%%%%%%%%%%%%%%%%%%%%%%%%%%%%%%%%

\end{table}
%%%%%%%%%%%%%%%%%%%%%%%%%%%%%%%%%%%%%%%%%%%%%%%%%%%%%%%%%%%%%%%%%%%%%%%%%%%%%%%%%
The discrepancies between two million and the number of computed Frobenius elements were instances where either
the ideal $I$ in $\FF_{10007}[x]$ was positive-dimensional, the eliminant did not have the
expected degree, or it was not square-free.
Both Schubert Galois groups are subgroups of $S_6$.
Dividing the total number of cycles computed by the number whose cycle type is $(1,1,1,1,1,1)$ (these Frobenius elements
were the identity) gives $48.1686$ and $72.8886$, respectively. 
The divisors of $|S_6|=6!=720$ closest to these numbers are 48 and 72, and the observed cycle types are consistent with
their Galois groups being $S_2\wr S_3$ and $S_3\wr S_2$, which have orders 48 and 72, respectively.
For each cycle type, we determined the fraction (out of 48 and 72) of Frobenius elements with that cycle type.
Other than the identity in the second  group, the observed fraction was within $0.5\%$ of
the actual distribution in the expected Galois group.
In the next section, we show that these Schubert problems have Galois group equal to $S_2\wr S_3$ and 
$S_3\wr S_2$, respectively.

%\newpage
%%%%%%%%%%%%%%%%%%%%%%%%%%%%%%%%%%%%%%%%%%%%%%%%%%%%%%%%%%%%%%%%%%%%%%%%%%%%%%%%%
\section{Fibrations of Schubert Problems}\label{S:geometry}

The essential enriched Schubert problems in $\Gr(4,9)$ share a common structure which explains their Galois groups:
their branched covers $\calX_\blambda\to(\Fln)^s$ form decomposable projections in the terminology of Am\'endola and
  Rodriguez~\cite{AR}. 
More precisely, over a dense open subset of the base $(\Fln)^s$ their solutions form a fiber bundle with base and fibers
Schubert problems in smaller Grassmannians. 
This is similar to the structure identified by Esterov~\cite{Esterov} for systems of sparse polynomials.
As the branched cover is decomposable, the corresponding Galois group is  a subgroup
of a wreath product~\cite{SDSS,PirolaSchlesinger} and therefore imprimitive. 
Thus any fibered Schubert problem has Galois group a subgroup of a wreath product.

There are two families of enriched Schubert problems in $\Gr(4,8)$ that are fibrations, each in different way.
This persists to $\Gr(4,9)$.
We treat each type of fibration in each of the next two subsections.
The Schubert problems in Section~\ref{S:typeI} are instances of a more general construction, called \demph{composition},
which is studied in~\cite{SWY}.
The Schubert problems in Section~\ref{S:typeII} are not compositions, and we do not yet know of a general construction for
them.

%%%%%%%%%%%%%%%%%%%%%%%%%%%%%%%%%%%%%%%%%%%%%%%%%%%%%%%%%%%%%%%%%%%%%%%%%%%%%%%%%
\begin{definition}\label{D:fibration}
 Let $\blambda$, $\bmu$, and $\bnu$ be Schubert problems in $\Gr(k{+}a,n{+}b)$, $\Gr(k,n)$, and  $\Gr(a,b)$, respectively.
 Then \demph{$\blambda$ is fibered over $\bmu$ with fiber $\bnu$}
 if the following holds:
 \begin{enumerate}
   \item For every general instance $\calFdot\in(\Fl_{n+b})^s$ of $\blambda$, there is a subspace $V\subset\CC^{n+b}$ of
         dimension $n$ and an instance $\calEdot$ of $\bmu$ in $\Gr(k,V)$ such that for every 
         $H\in\Omega_{\blambda}\calFdot$, we have $H\cap V\in\Omega_{\bmu}\calEdot$.
  \item If we set $\defcolor{W}:=\CC^{n+b}/V$, then for any $h\in \Omega_{\bmu}\calEdot$, there is an instance
    $\calGdot(h)$ of $\bnu$ in $\Gr(a,W)$ such that if $H\in\Omega_\blambda\calFdot$ with $\defcolor{h}:=H\cap V$, then
    $H/h\in\Gr(a,W)$ is a solution to $\Omega_{\bnu}\calGdot(h)$.
  \item The map $H\mapsto (h,H/h)$, where $h:=H\cap V$, is a bijection between the sets of solutions 
    $\Omega_{\blambda}\calFdot$ and
    \[
       \coprod_{h\in \Omega_{\bmu}\calEdot}\{h\}\times \Omega_{\bnu}\calGdot(h)\ .
    \]
  \item For a given subspace $V\simeq\CC^{n}$ of $\CC^{n+b}$, all general instances $\calEdot$ of $\bmu$ in $\Gr(k,V)$
    may be obtained from flags $\calFdot$ that induce the space $V$ in (1). 
    For a general such $\calEdot$ and $h\in\Omega_{\bmu}\calEdot$, the set of instances $\calGdot(h)$ of $\bnu$
    in $\Gr(a,W)$ which arise 
    also contains an open dense subset of the set of general instances of $\calGdot(h)$.
  \end{enumerate}
 This is called a fibration  as the bijection of (3) realizes $\Omega_{\blambda}\calFdot$ as a fibration over $\Omega_{\bmu}\calEdot$
   with fiber $\Omega_{\bnu}\calGdot(h)$ over $h\in\Omega_{\bmu}\calEdot$.
\hfill$\diamond$
\end{definition}
%%%%%%%%%%%%%%%%%%%%%%%%%%%%%%%%%%%%%%%%%%%%%%%%%%%%%%%%%%%%%%%%%%%%%%%%%%%%%%%%%

 Oftentimes, given $V$ we identify a subspace $W(h)\subset\CC^{n+b}$ complementary to
 $V$ and an instance $\calGdot(h)$ of $\bnu$ in $\Gr(a,W(h))$ such that
 $H\in\Omega_{\blambda}\calFdot \Rightarrow H=(H\cap V)\oplus(H\cap W(h))$, and 
 \begin{equation}\label{Eq:Sols_Fibration}
   \Omega_\blambda \calFdot\ =\ 
    \{ h\oplus K \mid h\in \Omega_{\bmu}\calEdot \mbox{ and } K\in \Omega_\bnu \calGdot(h)\}\,.
 \end{equation}
 This suffices as subspaces $W$ complimentary to $V$ are canonically identified with $\CC^{n+b}/V$ by the composition
 $W\to\CC^{n+b}\twoheadrightarrow\CC^{n+b}/V$, and this implies part (3) of Definition~\ref{D:fibration}.
 The more general statement of Definition~\ref{D:fibration} is adapted from~\cite{SWY}.

%%%%%%%%%%%%%%%%%%%%%%%%%%%%%%%%%%%%%%%%%%%%%%%%%%%%%%%%%%%%%%%%%%%%%%%%%%%%%%%%%
\begin{lemma}\label{L:Fibered}
 If $\blambda$ is a Schubert problem fibered over $\bmu$ with fiber $\bnu$, then
 $d(\blambda) = d(\bmu) d(\bnu)$ and  $\Gal_\blambda$ is a subgroup of the wreath product
 $\Gal_\bnu \wr \Gal_\bmu$ whose projection to $\Gal_\bmu$ is surjective.
 Furthermore, the kernel of the surjection $\Gal_\blambda\to\Gal_\bmu$ is a subgroup of $\Gal_{\bnu}^{d(\bmu)}$
 that is stable under the action of $\Gal_\bmu$ and whose projection to each $\Gal_{\bnu}$ factor is surjective.
\end{lemma}
%%%%%%%%%%%%%%%%%%%%%%%%%%%%%%%%%%%%%%%%%%%%%%%%%%%%%%%%%%%%%%%%%%%%%%%%%%%%%%%%%

%%%%%%%%%%%%%%%%%%%%%%%%%%%%%%%%%%%%%%%%%%%%%%%%%%%%%%%%%%%%%%%%%%%%%%%%%%%%%%%%%
\begin{proof}
 The bijection from part (3) of Definition~\ref{D:fibration} shows that $d(\blambda) = d(\bmu) d(\bnu)$.
 Let $\calFdot$ be a general instance of $\blambda$ and $\calEdot$ the induced instance of $\bmu$.
  Since the group $\Gal_{\blambda}$ preserves the fibration $\Omega_{\blambda}\calFdot\to\Omega_{\bmu}\calEdot$ of finite
  sets, we have the inclusion $\Gal_{\blambda}\subset\Gal_\bnu \wr \Gal_\bmu$, together with a map
  $\Gal_{\blambda}\to \Gal_\bmu$ whose kernel \defcolor{$\Gamma$} is a subgroup of $\Gal_\bnu^{d(\bmu)}$.

 By part (4) of Definition~\ref{D:fibration}, if $\CC^n\simeq V\subset\CC^{n+b}$, then all general instances $\calEdot$ of
  $\bmu$ in $\Gr(k,V)$ arise by the construction of part (1).
  Thus any based loop through general instances of $\bmu$ lifts to a loop through general instances of $\blambda$.
  This implies the surjectivity of the map $\Gal_{\blambda}\to\Gal_{\bmu}$, and that the kernel $\Gamma$ is stable under
  conjugation by elements of $\Gal_{\bmu}$.

 To study this kernel $\Gamma$, fix a general instance $\calEdot$ of $\bmu$ and $h\in\Omega_{\bmu}\calEdot$.
 As the flags $\calGdot(h)$ induced by flags that induce $\calEdot$ are general, the same arguments as in the previous
 paragraph implies that the map $\Gamma\to\Gal_\nu$ acting on solutions $H/h$ to $\Omega_\bnu\calGdot(h)$ is surjective.
\end{proof}
%%%%%%%%%%%%%%%%%%%%%%%%%%%%%%%%%%%%%%%%%%%%%%%%%%%%%%%%%%%%%%%%%%%%%%%%%%%%%%%%%

The difficulty in establishing that $\Gal_{\blambda}=\Gal_\bnu \wr \Gal_\bmu$ is to show that
$\Gamma=(\Gal_\bnu)^{d(\bmu)}$, i.e.\ that monodromy acts sufficiently independently on each $\Omega_{d(\bnu)}\calGdot(h)$
for $h\in\Omega_{\bmu}\calEdot$. 

For partitions $\lambda$ and $\mu$, let  \defcolor{$\lambda + \mu$} be their component-wise sum, 
$(\lambda+\mu)_i = \lambda_i + \mu_i$, which is always a partition. 
Let \defcolor{$(\lambda,\mu)$} be the decreasing rearrangement of the parts of $\lambda$ and $\mu$.
In this paper, we will only consider $\lambda+\mu$ when  $\mu_{a+1} =0$ and $\lambda_1=\cdots= \lambda_a = r$, so that $\lambda$ has $r$
columns with the last of height at least $a$.
Then the columns of $\lambda+\mu$ are the $r$ columns of $\lambda$ followed by the columns of $\mu$.
We will also only consider $(\lambda,\mu)$ when  $\lambda_r  \geq \mu_1$ and $\lambda_{r+1} = 0$, so that the rows of
$(\lambda,\mu)$ are the rows of $\lambda$ followed by the rows of $\mu$.

Notice that if \defcolor{$\lambda^t$} denotes the conjugate partition obtained from $\lambda$ by interchanging rows
with columns (e.g., matrix transpose), then  $(\lambda + \mu)^t = (\lambda^t,\mu^t)$.
Note that $|\lambda + \mu|=|(\lambda, \mu)|=|\lambda|+|\mu|$.

%%%%%%%%%%%%%%%%%%%%%%%%%%%%%%%%%%%%%%%%%%%%%%%%%%%%%%%%%%%%%%%%%%%%%%%%%%%%%%%%%
\begin{example}
 We illustrate these definitions.
 Write \defcolor{$1^{a}$} for the partition $(1,\dotsc,1)$ with $a$ parts of size 1 and \defcolor{$c$} 
  for the partition $(c)$ with one part of size $c$.
 If $\mu=(3,1)$, we display Young diagrams for $1^3+\mu$, $(4, \mu)$, 
 $(4, \mu)^t=1^4+(2,1,1)$, as well as $(2,2)+ 1$ and  $((2,2), 1)$.
\[
    \includegraphics{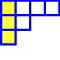}\qquad
    \includegraphics{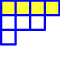}\qquad
    \includegraphics{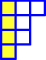}\qquad
    \includegraphics{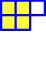}\qquad
    \includegraphics{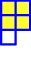}
\]
 In these, we have shaded the  portions $1^3$, $4$, $1^4$, and $(2,2)$.\hfill{$\diamond$}
\end{example}
%%%%%%%%%%%%%%%%%%%%%%%%%%%%%%%%%%%%%%%%%%%%%%%%%%%%%%%%%%%%%%%%%%%%%%%%%%%%%%%%%
%\comm{Even though we are using this definition for specific cases, it may be useful to the reader if we show an example when $\lambda$ has more than one part, e.g. $(\lambda, \mu)$ for $\lambda= (4,2)$ and $\mu = (3,1)$ is equal to $(4,3,2,1)$}

%%%%%%%%%%%%%%%%%%%%%%%%%%%%%%%%%%%%%%%%%%%%%%%%%%%%%%%%%%%%%%%%%%%%%%%%%%%%%%%%%
\subsection{Fibrations of Type I}\label{S:typeI}
We study the Galois groups of 120 enriched Schubert problems which all have similar constructions.
In a Schubert problem $\bnu = (\nu^1,\dotsc, \nu^s)$ in $\Gr(a,b)$, some of the partitions
$\nu^i$ could be $0$, and therefore impose no conditions.
While these do not affect the geometry of a Schubert problem, this flexibility is important for 
the following result.

%%%%%%%%%%%%%%%%%%%%%%%%%%%%%%%%%%%%%%%%%%%%%%%%%%%%%%%%%%%%%%%%%%%%%%%%%%%%%%%%%
\begin{theorem}\label{T:First_Family}
 Suppose that $a < b$ and $\bnu$ is a Schubert problem in $\Gr(a,b)$. 
 Then
\[
   \blambda\ :=\  \bigl((b{-}a{+}1, \nu^1)\,,\, (b{-}a{+}1, \nu^2)\,,\, 
      1^{a+1}+\nu^3\,,\, 1^{a+1}+\nu^4\,,\, \nu^5, \dotsc, \nu^s\bigr)
\]
is a Schubert problem in $\Gr(2{+}a, 4{+}b)$ that is fibered over $\I^4$ in $\Gr(2,4)$ with fiber $\bnu$.
Its Galois group $\Gal_{\blambda}$ is a subgroup of $\Gal_{\bnu}\wr S_2$ as in Lemma~\ref{L:Fibered}.
\end{theorem}
%%%%%%%%%%%%%%%%%%%%%%%%%%%%%%%%%%%%%%%%%%%%%%%%%%%%%%%%%%%%%%%%%%%%%%%%%%%%%%%%%

%%%%%%%%%%%%%%%%%%%%%%%%%%%%%%%%%%%%%%%%%%%%%%%%%%%%%%%%%%%%%%%%%%%%%%%%%%%%%%%%%
\begin{example}\label{Ex:G25}
  Let $a=2$ and $b=5$.
   The Schubert problems $(\I^4,\I^2)$, $(0,0,\I^2,\I^4)$, and $(0,\I,0,\I,\I^4)$ 
   in $\Gr(2,5)$  are the same
 geometrically, and each has five solutions. 
 The construction of Theorem~\ref{T:First_Family} gives the following Schubert problems in $\Gr(4,9)$,
\[
   \raisebox{-7.5pt}{\includegraphics{figures/41}}^2 \cdot
   \raisebox{-14.5pt}{\includegraphics{figures/211}}^2  \cdot\bI^2\, ,
   \qquad 
   \includegraphics{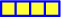}^2\cdot\raisebox{-14.5pt}{\includegraphics{figures/211}}^2 \cdot\bI^4\, ,
   \quad\mbox{and}\quad
   \includegraphics{figures/4}\cdot\raisebox{-7.5pt}{\includegraphics{figures/41}}\cdot
   \raisebox{-14.5pt}{\includegraphics{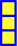}}\cdot
   \raisebox{-14.5pt}{\includegraphics{figures/211}} \cdot\bI^4\, .
\]
 By Theorem~\ref{T:First_Family}, each is fibered over $\I^4=2$ in $\Gr(2,4)$ with fiber $\I^6=5$ in $\Gr(2,5)$.
 By Lemma~\ref{L:Fibered}, each has ten solutions and has Galois group a subgroup of $S_5\wr S_2$.

 By Theorem~\ref{T:First_Family}, every nontrivial Schubert problem in $\Gr(2,5)$ gives an enriched Schubert
 problem in $\Gr(4,9)$ that is fibered over $\I^4=2$.
 Omitting trivial conditions of 0 and writing in multiplicative form, these Schubert problems in $\Gr(2,5)$ are
 \[
  2\ =\ \raisebox{-5.5pt}{\TI}\cdot\I^3\ =\ \raisebox{-5.5pt}{\II}\cdot\I^4\ =\ \T^2\cdot\I^2\,,\ \ 
  3\ =\ \T\cdot\I^4\,,
  \quad\mbox{and}\quad
  5\ =\ \I^6\,,
 \]
 and each has full symmetric Galois group, by Proposition~\ref{P:Two}.
 These and others geometrically equivalent to them give enriched Schubert problems whose Galois groups are
 subgroups of wreath products $S_d\wr S_2$, for $d=2,3,5$.
 Our computation of Frobenius elements reported in Section~\ref{SS:Frobnenius} and detailed on our
 webpage\footnote{{\tt http://www.math.tamu.edu/\~{}sottile/research/stories/GIVIX}} shows
 that each of these has Galois group over $\QQ$ equal to the wreath product. 
 Table~\ref{Ta:107} gives the number of Schubert problems in $\Gr(4,9)$ fibered over $\I^4=2$ with the given
 fiber.\hfill{$\diamond$}
\end{example}
%%%%%%%%%%%%%%%%%%%%%%%%%%%%%%%%%%%%%%%%%%%%%%%%%%%%%%%%%%%%%%%%%%%%%%%%%%%%%%%%%

%%%%%%%%%%%%%%%%%%%%%%%%%%%%%%%%%%%%%%%%%%%%%%%%%%%%%%%%%%%%%%%%%%%%%%%%%%%%%%%%% 
\begin{table}[htb]
  \caption{Number of Schubert problems in $\Gr(4,9)$ fibered over ${\protect\I}^4=2$ arising
     from Theorem~\ref{T:First_Family} and Example~\ref{Ex:G25} when $a=2$ and $b=5$.}
  \label{Ta:107}
%\[
  \begin{tabular}{|c|c|c|c|c|c|}\hline  \raisebox{-8.pt}{\rule{0pt}{20.5pt}}
    Fiber & $\raisebox{-5.5pt}{\TI}\cdot\I^3$&$\raisebox{-5.5pt}{\II}\cdot\I^4$&
                     $\T^2\cdot\I^2$&$\T\cdot\I^4$&$\I^6$\\\hline \rule{0pt}{12.5pt}
   Number &20&21&26&21&9\\\hline \rule{0pt}{12.5pt}
   Galois Group over $\CC$ &$S_2\wr S_2$&$S_2\wr S_2$&$S_2\wr S_2$&$S_3\wr S_2$&$S_5\wr S_2$\\\hline
  \end{tabular}
%\]
\end{table}
%%%%%%%%%%%%%%%%%%%%%%%%%%%%%%%%%%%%%%%%%%%%%%%%%%%%%%%%%%%%%%%%%%%%%%%%%%%%%%%%% 

%%%%%%%%%%%%%%%%%%%%%%%%%%%%%%%%%%%%%%%%%%%%%%%%%%%%%%%%%%%%%%%%%%%%%%%%%%%%%%%%%
\begin{corollary}\label{C:First_Family}
  The $97$ Schubert problems of Table~\ref{Ta:107} all have the claimed Galois groups over $\CC$.
\end{corollary}
%%%%%%%%%%%%%%%%%%%%%%%%%%%%%%%%%%%%%%%%%%%%%%%%%%%%%%%%%%%%%%%%%%%%%%%%%%%%%%%%%

We prove this after giving a proof of Theorem~\ref{T:First_Family}.

%%%%%%%%%%%%%%%%%%%%%%%%%%%%%%%%%%%%%%%%%%%%%%%%%%%%%%%%%%%%%%%%%%%%%%%%%%%%%%%%%
\begin{remark}
 A (partial) flag $\Fdot$ in $\CC^n$ is a nested sequence of subspaces, where not all  dimensions need to occur.
 Equivalently, a flag is an increasing filtration of $\CC^n$ by linear subspaces.
 If all possible dimensions occur in $\Fdot$, it is \demph{complete}, as in~\eqref{Eq:FullFlag}.
 For any subspace $h\subset\CC^n$, the sequence of subspaces $h{+}F_i$
 for $F_i\in\Fdot$ is an increasing filtration and thus forms another flag
 \defcolor{$h{+}\Fdot$} in $\CC^n$ with smallest subspace $h$.
 If $V\subset\CC^n$ is a subspace, then the subspaces $V\cap F_i$ form a flag \defcolor{$V\cap\Fdot$} in $V$.
 If $\phi\colon\CC^n\twoheadrightarrow V$ is surjective, then the image of $\Fdot$ is a flag in $V$.
 In the constructions of these filtrations (flags), we may have coincidences; $h+F_i=h+F_j$ or $V\cap F_i=V\cap F_j$ or
 $\phi(F_i)=\phi(F_j)$ with $i\neq j$. 
 Many of  our arguments require that we determine the description of particular elements of the flags resulting from these constructions. 
 \hfill{$\diamond$}
%
%   This may be generalized further
%
\end{remark}
%%%%%%%%%%%%%%%%%%%%%%%%%%%%%%%%%%%%%%%%%%%%%%%%%%%%%%%%%%%%%%%%%%%%%%%%%%%%%%%%%

% \renewcommand{\old}[1]{}
%%%%%%%%%%%%%%%%%%%%%%%%%%%%%%%%%%%%%%%%%%%%%%%%%%%%%%%%%%%%%%%%%%%%%%%%%%%%%%%%%
\begin{proof}[Proof of Theorem~\ref{T:First_Family}]

 We first show that $\blambda$ is a Schubert problem.
   As $\bnu$ is a Schubert problem in $\Gr(a,b)$, we have $a(b{-}a)=\sum_i|\nu^i|$.
 Then $\sum_i|\lambda^i|$ equals  
\begin{multline*}
   (b{-}a+1)+|\nu^1|+(b{-}a+1)+|\nu^2|+ (a{+}1)+|\nu^3|+ (a{+}1)+|\nu^4| +\sum_{i=5}^s|\nu^i|\\
  \ =\ 2(b{-}a+1)+2(a{+}1)+ a(b{-}a)\ =\ (a{+}2)(b{+}4{-}(a{+}2))\,,
\end{multline*}
 so that $\blambda$ is a Schubert problem in $\Gr(2{+}a,4{+}b)$.

 Let $\calFdot:=(\Fdot^1,\dotsc,\Fdot^s)$ be general flags in $\CC^{4+b}$ giving an instance of $\blambda$
 with solutions $\Omega_\blambda\calFdot$.
 We will often use that the flags in $\calFdot$ are in linear general position
 when making assertions about dimensions of intersections and spans, and leave the verification to the reader.
  Generality also implies that the dimension inequalities~\eqref{Eq:SchubertCondition} that define the Schubert
   varieties hold with equality.
 We establish the four parts of Definition~\ref{D:fibration} for $\bmu=\I^4$ in separate headings below.\medskip

 {\bf Part (1).}
   Set  $\defcolor{V} := F^1_2+ F^2_2$.
   As the flags are general, this is a direct sum and $V\simeq\CC^4$.
   Define $\calEdot$ by $\Edot^i:=\Fdot^i\cap V$ for $i=1,\dotsc,4$.
   Fixing $F^1_2$ and $F^2_2$ and thus $V$, but otherwise letting $\calFdot$ vary, every possible collection of complete
   flags $\calEdot$ in $V$ can occur, which implies the first half of part (4) in Definition~\ref{D:fibration}.

   Let $H\in\Omega_\blambda\calFdot$.
   Since $H\in\Omega_{(b{-}a{+}1, \nu^i)}\Fdot^i$ for $i=1,2$, condition~\eqref{Eq:SchubertCondition} for $j=1$ is 
   \[
     \dim (H\cap F^i_{4+b-(2+a)+1-(b{-}a{+}1)})\ =\ 1\,.
   \]
   As $2=4+b-(2+a)+1-(b{-}a{+}1)$, these subspaces are $F^1_2$ and $F^2_2$, so that $\defcolor{h}:=H\cap V$ is
   two-dimensional.
   Thus $h\in\Gr(2,V)$ and it satisfies $h\in\Omega_{\sI}\Edot^i$, for $i=1,2$.

   Now let $i\in\{3,4\}$.
   Since the $(a{+}1)$st part of $1^{a+1}+\nu^i$ is $1$ and $H\in\Omega_{1^{a+1}+\nu^i}\Fdot^i$,
   condition~\eqref{Eq:SchubertCondition} for $j=a{+}1$ is $\dim (H\cap F^i_{4+b-(2+a)+a+1-1})= a+1$.
   Furthermore, since $b{+}2=4{+}b{-}(2{+}a){+}a{+}1{-}1$, this becomes $\dim (H\cap F^i_{b+2})=a{+}1$.
   As $\dim(H)=a{+}2$, this intersection has codimension 1 in $H$, which implies that $\dim (h\cap F^i_{b+2})=1$.
   Because $V$ has codimension $b$ and $\Fdot^i$ is general, we have $\dim (V\cap F^i_{b+2})=2$ and thus $E^i_2= V\cap F^i_{b+2}$.
   Hence $h\in\Omega_{\sI}\Edot^i$.
   This shows that $h\in\Omega_{\sI^4}\calEdot$,  establishing part (1) of Definition~\ref{D:fibration}.\medskip

  {\bf Parts (2) and (3).}
   Define $\defcolor{W}:=F^3_{b+2}\cap F^4_{b+2}$.
   As the flags are in general position, $W\simeq\CC^b$ and $V{+}W=\CC^{4+b}$ is in direct sum.
   For $h\in\Omega_{\sI^4}\calEdot$ and for $i=1,\dotsc,s$ define $\defcolor{\Gdot^i(h)}:=W\cap(h{+}\Fdot^i)$, the (complete)
   flag in $W$ induced by $h{+}\Fdot^i$.
   Then the flags $\defcolor{\calGdot(h)}=(\Gdot^1(h),\dotsc,\Gdot^s(h))$ give an instance of the Schubert problem $\bnu$ in
   $\Gr(a,W)$.

   Let  $H\in\Omega_{\blambda}\calFdot$.
   Since $\dim (H\cap F^i_{b+2})=a{+}1$ for $i=3,4$ and $\dim (H)=a{+}2$, we have $\dim( H\cap W)=a$.
   Set $\defcolor{K}:= H\cap W\in\Gr(a,W)$.
   Then $H = h+K$ is in direct sum, where $h=H\cap V$.
   To show that part (2) of Definition~\ref{D:fibration} holds, we must show that $K\in\Omega_{\bnu}\calGdot(h)$.

 It is convenient to prove parts (2) and (3) together by showing that for each $i=1,\dotsc,s$, 
 \begin{equation}\label{Eq:equiv}
   K\ \in\ \Omega_{\nu^i}\Gdot^i(h)\ 
   \Longleftrightarrow\ h\oplus K\in\Omega_{\lambda^i}\Fdot^i\,.
 \end{equation}
 That is, $K$ satisfies the Schubert conditions~\eqref{Eq:SchubertCondition} for $\nu^i$ and $\Gdot^i(h)$ if and only if
 $h\oplus K$ satisfies the conditions~\eqref{Eq:SchubertCondition} for $\lambda^i$ and $\Fdot^i$.
 This involves careful counting and some linear algebra, using that the flags $\calFdot$ are in general position.
 As $h$ is fixed, let us write \defcolor{$\Gdot^i$} for $\Gdot^i(h)$.
 We will prove the assertion~\eqref{Eq:equiv} in three separate cases, $i\geq 5$, $i\in\{3,4\}$, and $i\in\{1,2\}$.\medskip

  Suppose that $i\geq 5$ and let $1\leq c\leq b$.
  Since $\Fdot^i$ is in general position with respect to the 2-plane $h$, we have $h\cap F^i_{c+2}=\{0\}$.
  Thus, $h{+}F^i_{c+2}$ is a direct sum and we have $\dim(h\oplus F^i_{c+2})=c{+}4$.
  Since $W=F^3_{b+2}\cap F^4_{b+2}$ has codimension 4 and is in general position with respect to $h$ and to $F^i_{c+2}$, we
  have $\dim(h\oplus F^i_{c+2})\cap W=c$.
  Hence $G^i_c=(h\oplus F^i_{c+2})\cap W$.
  We claim that if $K\in\Gr(a,W)$, then
  \begin{equation}\label{Eq:Claim_1}
     \dim (K\cap G^i_c)\ =\ \dim((h\oplus K)\cap F^i_{c+2})\,.
  \end{equation}
  To see this, set $d:=\dim((h\oplus K)\cap (h\oplus F^i_{c+2}))$, and observe that $h$ lies in the intersection.
  As $F^i_{c+2}$ has codimension 2 in $h\oplus F^i_{c+2}$,
  and $h$ is in direct sum with $(h\oplus K)\cap F^i_{c+2}$, which forces the inequality $d{-}2\leq\dim((h\oplus K)\cap F^i_{c+2})$.
  This implies that $d{-}2=\dim((h\oplus K)\cap F^i_{c+2})$.
  Similarly, as $K$ has codimension 2 in $h\oplus K$, we have $d{-}2=\dim ( K\cap(h+ F^i_{c+2}))$.
  Since $K\subset W$ and $G^i_c=W\cap(h\oplus F^i_{c+2})$, we have $K\cap G^i_c= K\cap(h+ F^i_{c+2})$,
  which establishes~\eqref{Eq:Claim_1}.

  To verify~\eqref{Eq:equiv}, let $j\leq a$.
  By letting $c=b{-}a{+}j{-}\nu^i_j$ in~\eqref{Eq:Claim_1}, we have 
\[
  \dim (K\cap G^i_{b-a+j-\nu^i_j})\ =\ \dim( (h\oplus K)\cap F^i_{b-a+j-\nu^i_j+2})\,.
\]
 As $\nu^i=\lambda^i$ and $b-a+j-\nu^i_j+2= (b{+}4)-(a{+}2) +j-\lambda^i_j$,
 we see that~\eqref{Eq:equiv} holds, by the Schubert
 conditions~\eqref{Eq:SchubertCondition}.\medskip

 Suppose now that $i\in\{3,4\}$.
 Let us investigate the flag $\Gdot^3$ (the same argument holds for $\Gdot^4$).
 Observe first that $\dim (h\cap F^3_{b+2})=1$, as $E^3_2=V\cap F^3_{b+2}$.
 Suppose that  $1\leq c\leq b$.
 Then $\dim ((h+F^3_{c+1})\cap F^3_{b+2})=c{+}2$, and thus  $\dim( (h+F^3_{c+1})\cap W)=c$, as  $W=F^3_{b+2}\cap F^4_{b+2}$ and
 the flags are general given that $h$ meets $F^3_{b+2}$ and $F^4_{b+2}$.
 Thus $G^3_c=(h+F^3_{c+1})\cap W$.
 Let $K\in\Gr(a,W)$.
 Then essentially the same arguments as for~\eqref{Eq:Claim_1} show that 
  \begin{equation}\label{Eq:Claim_2}
    \dim (K\cap G^3_c)\ =\ \dim((h\oplus K)\cap F^3_{c+1})\,.
  \end{equation}

   To verify~\eqref{Eq:equiv}, let $j\leq a$.
 Then~\eqref{Eq:Claim_2} with $c=b{-}a{+}j{-}\nu^3_j$ implies that 
\[
  \dim (K\cap G^3_{b-a+j-\nu^3_j})\ =\ \dim( (h\oplus K)\cap F^3_{b-a+j-\nu^3_j+1})\,.
\]
 Since $\lambda^3_j=1+\nu^3_j$ and $b-a+j-\nu^3_j+1= (4{+}b)-(2{+}a) +j-\lambda^3_j$, we see
 that~\eqref{Eq:SchubertCondition} holds for $j\leq a$. 
 For $j=a{+}1$, we have $\lambda^3_{a+1}=1$.
 Note that $\dim((h\oplus K)\cap F^3_{b+2})=a+1$ as 
 $K\subset W\subset F^3_{b+2}$ and $\dim (h\cap E^3_2)=1$, which shows that~\eqref{Eq:equiv} holds.\smallskip

 Suppose now that $i\in\{1,2\}$.
 As before, it suffices to show this for $i=1$.
 Since $E^1_2=F^1_2\subset V$, we have $1=\dim (h\cap E^1_2)=\dim ((h\oplus K)\cap F^1_2)$.
 Let $1\leq c\leq b$.
 Since $\dim (h\cap F^1_2)=1$ and $h\not\subset F^1_{c+3}$, we have $\dim(h+F^1_{c+3})=c{+}4$, and so $\dim((h+F^1_{c+3})\cap W)=c$.
 This uses again the general position of $\Fdot^1$ with $\Fdot^3$ and $\Fdot^4$.
 Thus $G^1_c=(h+F^1_{c+3})\cap W$.
 We claim that if $K\in\Gr(a,W)$, then
   \begin{equation}\label{Eq:Claim_3}
     1+\dim (K\cap G^1_c)\ =\ \dim((h\oplus K)\cap F^1_{c+3})\,.
   \end{equation}
   To see this, let $\defcolor{d}:=\dim( (h\oplus K)\cap(h+F^1_{c+3}))$. 
 Since $F^1_{c+3}$ has codimension 1 in $h+F^1_{c+3}$ and $h\not\subset F^1_{c+3}$, we have that 
 $d{-}1=\dim((h\oplus K)\cap F^1_{c+3})$.
  Similarly, as $K$ has codimension 2 in $h\oplus K$, we have that 
  $d{-}2=\dim (K\cap(h+F^1_{c+3}))$.
  Recalling that $K\cap G^1_c=K\cap(h+F^1_{c+3})$ establishes~\eqref{Eq:Claim_3}.

 We already observed that $h\oplus K$ satisfies  $\dim ((h\oplus K)\cap F^1_2)=1$, which is one of the
 conditions~\eqref{Eq:SchubertCondition} for $\Omega_{\lambda^1}\Fdot^1$.
 For the others, note that for $j=1,\dotsc,a$, $\nu^1_j=\lambda^1_{j+1}$.
 Then using~\eqref{Eq:Claim_3} for $c=b{-}a{+}j{-}\nu^1_j$,
\[
  1+\dim (K\cap G^1_{b-a+j-\nu^1_j})\ =\ \dim( (h\oplus K)\cap F^1_{b-a+j-\nu^1_j+3})\,,
\]
and $b-a+j-\nu^1_j+3= (b{+}4)-(a{+}2) +(j+1)-\lambda^1_{j+1}$.
This completes the proof of parts (2) and (3) in Definition~\ref{D:fibration}.\medskip

{\bf Part (4).}
 We have already established (4) for the flags $\calEdot$.
 Observe that for a fixed $h\in\Omega_{\sI^4}\calEdot$, every
 flag $\Gdot$ in $W$ occurs as $W\cap(h+\Fdot)$, for some flag $\Fdot\in\Fl_{4+b}$.
 Thus every collection of flags in $W$  occurs as $\calGdot(h)$, for a given $h$.
 This completes the proof. 
\end{proof}
%%%%%%%%%%%%%%%%%%%%%%%%%%%%%%%%%%%%%%%%%%%%%%%%%%%%%%%%%%%%%%%%%%%%%%%%%%%%%%%%%
%\newpage
   
%%%%%%%%%%%%%%%%%%%%%%%%%%%%%%%%%%%%%%%%%%%%%%%%%%%%%%%%%%%%%%%%%%%%%%%%%%%%%%%%%
\begin{proof}[Proof of Corollary~\ref{C:First_Family}]
 Let $\blambda$ be a Schubert problem from Table~\ref{Ta:107} 
 that is fibered over $\I^4$ with fiber $\bnu$ (a nontrivial Schubert problem in $\Gr(2,5)$), and let \defcolor{$\Gamma$} be
 the kernel of the homomorphism $\Gal_{\blambda}\twoheadrightarrow S_2$.
 By Proposition~\ref{P:Two}, $\Gal_{\bnu}\simeq S_{d(\bnu)}$.
 By Lemma~\ref{L:Fibered}, $\Gamma\subset S_{d(\bnu)}\times S_{d(\bnu)}$ is stable under the action of $S_2$ and it projects onto each
 factor.
 Let \defcolor{$e$} denote the identity element of $S_{d(\bnu)}$. 
 We will show that $\Gamma$ contains an element $(e,\sigma)$ with $\sigma$ a transposition.
 Conjugating by elements of $\Gamma$ and by $S_2$ shows that $\Gamma$ contains all elements  $(e,\tau)$ and
 $(\tau,e)$ for $\tau$ a transposition, and thus that $\Gamma= S_{d(\bnu)}\times S_{d(\bnu)}$, which will complete the proof.

 Let $\pi\colon\calX_{\blambda}\to(\Fl_9)^s$ be the branched cover~\eqref{Eq:TotalSpace}, \defcolor{${\mathbb L}$} the Galois closure of the
 extension $\CC(\calX_{\blambda})/\pi^*(\CC((\Fl_9)^s))$, and $\defcolor{{\mathbb K}}:={\mathbb L}^\Gamma$ the fixed field of $\Gamma$.
 Then ${\mathbb K}$ is a quadratic extension (as $\Gal_{\blambda}/\Gamma=S_2$), which we identify.
 Let  $\rho\colon \calZ_{\sI^4}\to(\Fl_9)^s$ be the family of auxiliary problems $\I^4$ constructed as in Theorem~\ref{T:First_Family}.
 (Over a general $\calFdot\in(\Fl_9)^s$, the fiber is $\Omega_{\sI^4}\calEdot\subset\Gr(2,V)$, where $V=F^1_2+F^2_2$
   and $\calEdot=(\Fdot^i\cap V\mid i=1,\dotsc,4)$.)
  Over the locus of flags in $(\Fl_9)^s$ in sufficiently general position, the map $\pi\colon\calX_{\blambda}\to(\Fl_9)^s$
  factors through $\rho$, giving a dominant rational map $p\colon \calX_{\blambda}\dashrightarrow \calZ_{\sI^4}$.
  As $\calX_{\blambda}$ is irreducible, $\calZ_{\sI^4}$ is irreducible.  
  Since $\rho\colon \calZ_{\sI^4}\to(\Fl_9)^s$ has degree 2, its Galois group is $S_2$, and we may identify ${\mathbb K}$ with the
  function field of $\calZ_{\sI^4}$.
  The existence of a rational factorization through an intermediate branched cover is a general fact when the
  Galois group is imprimitive, see~\cite{PirolaSchlesinger} and~\cite[Proposition~1]{SDSS}.
  The construction in Theorem~\ref{T:First_Family} identifies the intermediate branched cover.

 We consider a subfamily of $\pi\colon \calX_{\blambda}\to(\Fl_9)^s$ over which $p$ is a covering space and $\rho$ is trivial. 
 Fix 2-planes $\defcolor{E^1_2},\defcolor{E^2_2}$, and 7-planes $\defcolor{\Lambda^3_7},\defcolor{\Lambda^4_7}$ in linear general
 position in $\CC^9$. 
 Let $\defcolor{\calY}=\defcolor{\calY(E^1_2,E^1_2,\Lambda^3_7,\Lambda^4_7)}\subset(\Fl_9)^s$ be the space of flags $\calFdot$ where
 $F^1_2=E^1_2$, $F^2_2=E^2_2$, $F^3_7=\Lambda^3_7$, and $F^4_7=\Lambda^4_7$.
 Set $\defcolor{V}:=\langle E^1_2,E^2_2\rangle$, $\defcolor{E^3_2}:=V\cap\Lambda^3_7$, and $\defcolor{E^4_2}:=V\cap\Lambda^4_7$.
 Every $\calFdot\in\calY$ gives the same auxiliary Schubert problem $\I^4$ in $\Gr(2,V)$,
 \[
   \Omega_{\sI} E^1_2\, \cap\, 
   \Omega_{\sI} E^2_2\, \cap\, 
   \Omega_{\sI} E^3_2\, \cap\, 
   \Omega_{\sI} E^4_2\,.
 \]
 (This is $\Omega_{\sI^4}\calEdot$ where $\Edot^i$ contains $E^i_2$.) 
 This auxiliary problem has two solutions \defcolor{$h_1$} and \defcolor{$h_2$} with $h_1\oplus h_2=V$.
 The family $\calZ_{\sI^4}|_\calY\to\calY$ is constant with fiber $\{h_1,h_2\}$.
 The fibers $\calX_{\blambda}|_\calY\to\calY$ are solutions to the two instances of $\bnu$ given by flags
 $\calGdot(h_a)$ for $a=1,2$, which are in $\Gr(2,W)$ where $\defcolor{W}:=\Lambda_3\cap\Lambda_4\simeq\CC^5$.
 Consequently, the monodromy group of $\calX_{\blambda}|_\calY\to\calY$ is a subgroup of the group $\Gamma$, as
 $\calZ_{\sI^4}|_\calY\to\calY$ has trivial monodromy $\{e\}\subset S_2$. (Here, $e\in S_2$ is its identity element.)

 In instances of $\bnu$ the conditions $\I$ are imposed by the three-dimensional subspaces of the corresponding flags.
 These are $\defcolor{G^{ai}_3}:=G^i_3(h_a)$ for the appropriate index $i$.
 We will argue that there is a family of flags in $\calY$ so that the
 flags $G^{1i}_3$ (when $a=1$) are constant, but the flags $G^{2i}_3$ (when $a=2$) move and induce a simple
 transposition $\sigma$.

 A flag $M_2\subset M_4\subset W$, where $\dim( M_j)=j$ determines a pencil of $3$-planes in $W$,
 \[
     \PP(M_4/M_2)\ :=\ \{ G_3\in\Gr(3,W)\,\mid\, M_2\subset G_3\subset M_4\}\,.
 \]
 We will show that for every pencil of $3$-planes in $W$, there is a family of flags $\Fdot^i$ (drawn from $\calY$) such
 that $G^{1i}_3$ is constant in that family, but $G^{2i}_3$ moves in the pencil.
 We then invoke arguments from Remark~\ref{R:Galois_Comment} to complete the proof.

 From the proof of Theorem~\ref{T:First_Family}, there are three cases for the  subspace $G^{ai}_3$ in terms of $h_a$
 and the flag $\Fdot^i$,
 \begin{enumerate}
  \item $(h_a+F^i_5 ) \cap W$ for $i\geq 5$,
  \item $(h_a+F^i_4 ) \cap W$ for $i=3,4$, or
  \item $(h_a+F^i_6 ) \cap W$ for $i=1,2$.
 \end{enumerate}
 
 In case (1), the map $v\mapsto (v+ F^i_5 ) \cap W$ is an isomorphism $V\xrightarrow{\,\sim\,}W/(F^i_5 \cap W)$.
 Since $V=h_1\oplus h_2$, we may choose $F^i_5$ so that $G^{1i}_3$ and $G^{2i}_3$ are any two 3-planes in $W$ whose
 intersection is any given 1-dimensional subspace.
 This implies the claim about pencils.

 In case (2), we let $i=3$, as the argument for $i=4$ is similar.
 For $a=1,2$,  $h_a\cap F^3_7=\defcolor{\ell_a}$ is a 1-dimensional subspace with $E^3_2=\ell_1\oplus\ell_2$.
 As $W\subset F^3_7$ and has codimension 2,
  \begin{equation}\label{Eq:Ga3_3}
    G^{a3}_3\ =\ (h_a+F^3_4 ) \cap W\ =\ (h_a+F^3_4)\cap F^3_7 \cap W
    \ =\ (\ell_a+F^3_4)\cap W\,.
  \end{equation}
 Consequently, we have
  \begin{equation}\label{Eq:careful_Flags}
    M_2\ :=\ F^3_4\cap W\ \subset\ G^{13}_3\,,\, G^{23}_3\ \subset\ (E^3_2+F^3_4)\cap W\ =:\ M_4\,.
  \end{equation}
 Thus $ G^{13}_3$ and $G^{23}_3$ are distinct members of the pencil $\PP(M_4/M_2)$. 

Given any flag  $M_2\subset M_4\subset W$,  there is  a flag $\Fdot^3$ from $\calY$ (that is, with $\Lambda^3_7=F^3_7$) 
such that $M_2 =F^3_4\cap W$ and $M_4=(E^3_2+F^3_4)\cap W$.
If we replace  $F^3_4$ by another 4-plane $F'_4\subset F^3_4+\ell_1$ such that $M_2 =F'_4\cap W$, then 
$\ell_1+F'_4=\ell_1+F^3_4$ and thus  $G^{13}_3=(\ell_1+F'_4)\cap W$.
Note that $E^3_2+F'_4=E^3_2+F^3_4$ so that $(E^3_2+F'_4)\cap W= M_4$ and thus $M_2\subset(\ell_2+F'_4)\cap W\subset M_4$.
 Since we may choose $F'_4$ so that $\ell_2+F'_4$ contains almost any given point of $W$,
 there exists a choice of $F'_4$ such that $G^{23}_3\neq (\ell_2+F'_4)\cap W$, and in fact almost all
 $3$-planes in the pencil  $\PP(M_4/M_2)$ (except $G^{13}_3$) may occur in this way.
Thus there are flags $\Fdot^3$  so that~\eqref{Eq:careful_Flags} holds with $G^{a3}_3$ defined as in~\eqref{Eq:Ga3_3}.
  Moreover, $G^{13}_3$ and $G^{23}_3$ may  be any two distinct points in the pencil $\PP(M_4/M_2)$.
 This implies the claim about pencils in case (2).

 We argue case (3) for $i=1$.
 For $a\in\{1,2\}$, let $\defcolor{\ell_a}:=h_a\cap E^2_2$, which is 1-dimensional and have $\ell_1\oplus\ell_2=E^2_2=F^2_2$.
 Since $h_a$ meets $E^1_2=F^1_2\subset F^1_6$, $h_a+F^1_6=\ell_a+F^1_6$ is 7-dimensional, and we have $G^{a1}_3 = (\ell_a+F^1_6) \cap W$.
 We also have the containments
  \begin{equation}\label{Eq:careful_FlagsII}
    M_2\ :=\ F^1_6\cap W\ \subset\ G^{11}_3, G^{21}_3\ \subset\ (E^2_2+F^1_6)\cap W\ =:\ M_4\,,
  \end{equation}
 so that $ G^{11}_3$ and $G^{21}_3$ lie in this pencil determined by $F^1_6$.
 Any flag $\Fdot^1$ chosen from $\calY$ will have $L_1=F^1_2\subset F^1_6$.
 If  $F'_6$  is any 6-plane such that $E^1_2,M_2\subset F'_6\subset \ell_1+F^1_6$,
 then the 3-planes $(\ell_a+F'_6)\cap W$ will lie in the pencil~\eqref{Eq:careful_FlagsII}.
 As in case (2), $G^{11}_3=(\ell_1+F'_6)\cap W$, and all 3-planes except $G^{11}_3$ may occur as
 $(\ell_2+F'_6)\cap W$.
 This implies the claim about pencils in case (3).

 While $\calY$ is not equal to $(\Fl_9)^s$, if we take the union of $\calY(E^1_2,E^2_2,\Lambda^3_7,\Lambda^4_7)$ over all linear
 subspaces $E^1_2,E^2_2$ and $\Lambda^3_7,\Lambda^4_7$ of the given dimensions in linear general position, we obtain a dense open
 subset of  $(\Fl_9)^s$.
 We may thus assume that  $E^1_2,E^2_2$ and $\Lambda^3_7,\Lambda^4_7$ have been chosen so that for a general choice of flags
 $\calFdot\in\calY$ the flags $\calGdot(h_1)$ and $\calGdot(h_2)$ are each general and each $\Omega_{\bnu}\calGdot(h_a)$
 consists of $d(\bnu)$ points.
 
 Let $i\in\{1,\dotsc,s\}$ be an index corresponding to a condition $\I$ in $\bnu$.
  By Remark~\ref{R:Galois_Comment}, if $G^{2i}_3$ moves in a general pencil while the other flags in $\calGdot(h_2)$ remain
 fixed, the pencil of instances contains an instance with a unique double point and thus monodromy in the pencil around
 that instance will induce a simple transposition $\sigma$.
 If this may be done while keeping the other flags $\calGdot(h_1)$ fixed, we obtain the desired element $(e,\sigma)$ in the
 group $\Gamma$.
 This will prove the corollary.

 We explain how this may be accomplished.
 Let us fix all the flags in $\calFdot$, except possibly $\Fdot^i$, and let $U$ be the set of flags $\Fdot$ that may
 replace $\Fdot^i$ in $\calFdot$ such that the new $s$-tuple of flags lies in $\calY$ and the subspace $G^{2i}_3$ does not
 change.
 Replacing $U$ by a dense open subset, we may assume that flags in $U$ induce subspaces $G^{1i}_3$ so that
 $\Omega_{\bnu}\calGdot(h_1)$ consists of $d(\bnu)$ points.
 A general subspace of $G^{2i}_3$ of the appropriate dimension (1 if $i\geq 5$ and 2 if $i<5$) occurs as  $G^{1i}_3\cap G^{2i}_3$ for some
 flag in $U$.   
 Thus, given a general pencil $M_2\subset  G^{2i}_3\subset M_4$ as in Remark~\ref{R:Galois_Comment}, there are flags in
 $U$ with $G^{1i}_3\cap G^{2i}_3\subset M_2$ and $G^{1i}_3\subset M_4$.
 Consequently, monodromy in this pencil will induce a permutation $(e,\sigma)$ for $\sigma$ a simple transposition.
 The arguments above show that almost all subspaces in this pencil may be induced by flags in $\calY$, and thus flags
   in $\calY$ induce a permutation $(e,\sigma)$ with $\sigma$ a simple transposition.
  This completes the proof.
\end{proof}
%%%%%%%%%%%%%%%%%%%%%%%%%%%%%%%%%%%%%%%%%%%%%%%%%%%%%%%%%%%%%%%%%%%%%%%%%%%%%%%%% 

In addition to the five families of fibered Schubert problems in $\Gr(4,9)$ of Example~\ref{Ex:G25} fibered over $\I^4=2$ in $\Gr(2,4)$,  
there are two more families which are fibered over a Schubert problem in $\Gr(2,5)$ and come from a general
construction similar to the construction of Theorem~\ref{T:First_Family}.

%%%%%%%%%%%%%%%%%%%%%%%%%%%%%%%%%%%%%%%%%%%%%%%%%%%%%%%%%%%%%%%%%%%%%%%%%%%%%%%%%
\begin{theorem}\label{Th:Second_Family}
 Suppose that  $0<a<b$ and $\bnu$ is a Schubert problem in $\Gr(a,b)$.
 \begin{enumerate}
 \item[(i)]
   $\blambda= 
     \bigl((b{-}a{+}2,b{-}a{+}1,\nu^1), 1^{a+1}+\nu^2, 1^{a+1}+\nu^3,
       1^{a+1}+\nu^4, \nu^5, \dotsc, \nu^s\bigr)$
    is a Schubert problem in $\Gr(2{+}a,5{+}b)$ fibered over $\raisebox{-5.5pt}{\TI}\cdot\I^3=2$ in $\Gr(2,5)$ with fiber
    $\bnu$.  

   \item[(ii)] $\blambda{=}\bigl((b{-}a{+}2,\nu^1),(b{-}a{+}1, \nu^2), 1^{a+1}{+}\nu^3,
       1^{a+1}{+}\nu^4, 1^{a+1}{+}\nu^5, \nu^6, \dotsc, \nu^s\bigr)$
    is a Schubert problem in $\Gr(2{+}a,5{+}b)$ fibered over $\T\cdot\I^4=3$ in $\Gr(2,5)$ with fiber $\bnu$. 
 \end{enumerate}
\end{theorem}
%%%%%%%%%%%%%%%%%%%%%%%%%%%%%%%%%%%%%%%%%%%%%%%%%%%%%%%%%%%%%%%%%%%%%%%%%%%%%%%%%

For $a=2$ and $b=4$, these constructions give enriched Schubert problems in $\Gr(4,9)$ with fiber $\bnu=\I^4$.
The construction in Theorem~\ref{Th:Second_Family}(i) gives eight enriched problems with four solutions and Galois group  $S_2\wr S_2$, while
the construction in Theorem~\ref{Th:Second_Family}(ii) gives 15 enriched problems with six solutions and Galois group $S_2\wr S_3$.
Below we give one representative from each of these two families of enriched problems,
\[
  \raisebox{-14.5pt}{\includegraphics{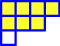}}\cdot
  \raisebox{-14.5pt}{\includegraphics{figures/211}}\cdot
  \raisebox{-14.5pt}{\includegraphics{figures/211}}\cdot
  \raisebox{-14.5pt}{\includegraphics{figures/211}}  \ =\ 4
  \qquad\mbox{and}\qquad
  \raisebox{-7.5pt}{\includegraphics{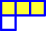}}\cdot
  \raisebox{-7.5pt}{\includegraphics{figures/41}}\cdot
  \raisebox{-14.5pt}{\includegraphics{figures/211}}\cdot
  \raisebox{-14.5pt}{\includegraphics{figures/211}}\cdot
  \raisebox{-14.5pt}{\includegraphics{figures/111}}  \ =\ 6\ .
\]

%%%%%%%%%%%%%%%%%%%%%%%%%%%%%%%%%%%%%%%%%%%%%%%%%%%%%%%%%%%%%%%%%%%%%%%%%%%%%%%%%
\begin{corollary}\label{C:Second_Family}
 The $23$ enriched Schubert problems in $\Gr(4,9)$ from Theorem~$\ref{Th:Second_Family}$ each have Galois group over $\CC$ as claimed.
\end{corollary}
%%%%%%%%%%%%%%%%%%%%%%%%%%%%%%%%%%%%%%%%%%%%%%%%%%%%%%%%%%%%%%%%%%%%%%%%%%%%%%%%%

%%%%%%%%%%%%%%%%%%%%%%%%%%%%%%%%%%%%%%%%%%%%%%%%%%%%%%%%%%%%%%%%%%%%%%%%%%%%%%%%%
\begin{proof}[Proof of Theorem~\ref{Th:Second_Family}]
 The proof is similar to that of Theorem~\ref{T:First_Family}.

 For (i), let $\calFdot\in(\Fl_{5{+}b})^s$ be general flags and suppose that $H\in\Omega_{\lambda^i}\Fdot^i$ for $i=1,\dotsc,4$.
 Genericity of $\calFdot$ and the Schubert condition~\eqref{Eq:SchubertCondition} for $i=1$ imply that
 $\dim (H\cap F_5^1)=\dim( H\cap F_4^1)=2$, as well as $\dim (H\cap F_2^1)=1$.
 Set $\defcolor{V}:=F_5^1$ and $h:=H\cap V\in\Gr(2,V)$ and for $i\in\{1,\dotsc,4\}$, let $\defcolor{\Edot^i}:=\Fdot^i\cap V$.
 Since $\Edot^1=F^1_1\subset\dotsb\subset F^1_5$, we have $h\subset E^1_4$ and $\dim (h\cap E^1_2)=1$, so that 
 $h\in \Omega_{\sTI}\Edot^1$.
 The conditions for $i=2,3,4$ imply that $\dim (H\cap F^i_{b+3})=a{+}1$.
 As $V$ has codimension $b$, $E^i_3=F^i_{b+3}\cap V$ and $\dim (H\cap V)=2$, this implies that $\dim (h\cap E^i_{3})=1$, 
 so that $h\in\Omega_{\sI}\Edot^i$.
 Thus $h\in\Omega_{\bmu}\calEdot$, where $\bmu=\raisebox{-5.5pt}{\TI}\cdot\I^3$.
 This establishes part (1) in the Definition~\ref{D:fibration} of a fibration.
 Since every possible collection of  flags $\calEdot$ in $V$ can occur, this also establishes the second half of
 part (4) in Definition~\ref{D:fibration}.

 We establish parts (2), (3), and the rest of (4).
 Let $h\in \Omega_{\bmu}\calEdot$.
 For $i\in\{2,3,4\}$ we have that $\dim(h+F^i_{b+3})=b+4$ as $1=\dim (h\cap E^i_3)=\dim( h\cap F^i_{b+3})$.
 Let $\defcolor{U(h)}$ be the intersection of the three hyperplanes $h+F^i_{b+3}$ for $i=2,3,4$.
 The general position of $\Fdot^i$ implies that $h=V\cap U(h)$, and so
 $\defcolor{W(h)}:=U(h)/h$ is canonically identified with $W=\CC^{b+5}/V$.

 The construction of $U(h)$ and $W$ 
 is explained by noting that if $H\in\Omega_{\blambda}\calFdot$ and $h=H\cap V$, then for each $i\in\{2,3,4\}$ %as 
 we have $\dim(H\cap F^i_{b+3})=a+1$, and thus $H\subset h+F^i_{b+3}$.
 Consequently, $H/h\in\Gr(a,W(h))$.
 For each $i\in\{1,\dotsc,s\}$, let us define
 \begin{equation}\label{Eq:G(h)}
   \defcolor{\Gdot^i(h)}\ :=\  \bigl( (h+\Fdot^i)\cap U(h)\bigr)/h\,.
 \end{equation}
 Any flag in $W(h)$ arises in this way, which establishes part (4) of Definition~\ref{D:fibration}. 
 For parts (2) and (3), we show that for all $H\in\Gr(a{+}2,b{+}5)$ with $h\subset H$ and $i\in\{1,\dotsc,s\}$
 \begin{equation}\label{Eq:needed_implications}
    H\ \in\ \Omega_{\lambda^i}\Fdot^i\ \Longleftrightarrow\ H/h\ \in\ \Omega_{\nu^i}\Gdot^i(h)\,.
 \end{equation}

 Let $i\geq 5$ and $1\leq c\leq b$.
 Then $\dim(h+F^i_{c+3})=c+5$, by the generality of $\Fdot$.
 As $U(h)$ has codimension 3 and $h\subset U(h)$, we have $G^i_c(h)=((h+F^i_{c+3})\cap U(h))/h$.
 Then for $j\in\{1,\dotsc,a\}$,
 \begin{eqnarray*}
    2+\dim\bigl(H/h\cap G^i_{b-a+j-\nu^i_j}(h)\bigr)&=&
    \dim(H\cap(h+ F^i_{b-a+j-\nu^i_j+3}))\\
    &=&    2+  \dim(H\cap F^i_{(b+5)-(a+2)+j-\nu^i_j})\,.
 \end{eqnarray*}
  Since $\lambda^i=\nu^i$ for $i\geq 5$, this shows~\eqref{Eq:needed_implications} for this case.
 
 For $i\in\{2,3,4\}$ it suffices to argue for $i=2$.
 The generality of $\Fdot^2$ implies that $h\cap F^2_{b+2}=\{0\}$ (but $\dim (h\cap F^2_{b+3})=1$).
 Let $1\leq c\leq b$.
 Then $\dim(h+F^2_{c+2})=c+4$.
 As $h+F^3_{b+3}$ and $h+F^4_{b+3}$ are hyperplanes, but $h+F^2_{c+2}\subset h+F^2_{b+3}$, we have that
 $\dim( (h+F^2_{c+2})\cap U(h))=c+2$, so that 
 $G^i_c(h)=((h+F^i_{c+2})\cap U(h))/h$.

 Let $H\in\Omega_{\lambda^2}\Fdot^2$ with $H\cap V=h$.
 Since for $j\in\{1,\dotsc,a\}$, $\lambda^2_j=1+\nu^2_j$, from~\eqref{Eq:SchubertCondition} we have
 \[
 j\ =\  \dim(H\cap F^2_{(b+5)-(a+2)+j-\lambda^2_j})\ =\ 
 \dim(H\cap F^2_{b-a+j-\nu^2_j+2}) \,.
 \]
 Thus $2+j=\dim(H\cap (h+F^2_{b-a+j-\nu^2_j+2}))$, and so
 $j=\dim (H/h \cap G^2_{b-a+j-\nu^2_j}(h))$, so that 
 $H/h\in\Omega_{\nu^2}\Gdot^2(h)$.
 These assertions about dimension are reversible, which proves~\eqref{Eq:needed_implications} for $i\in\{2,3,4\}$.

 Let $i=1$ and $1\leq c\leq b$.
 As $h\subset F^1_4\subset F^1_{c+5}$ and $\Fdot^2,\Fdot^3$ and $\Fdot^4$ are in general position with respect to
 $\Fdot^1$, $\dim((h+F^1_{c+5})\cap U(h))=c+2$ and thus  $G^1_c(h)=((h+F^i_{c+5})\cap U(h))/h$.
 Then we have that $\dim(H\cap F^1_{c+5})=2+\dim(H/h\cap G^1_c(h))$.
 Let $1\leq j\leq a$.
 As $\nu^1_j=\lambda^1_{j+2}$, we have
 \[
   j+2\ =\ \dim(H\cap F^1_{(b+5)-(a+2)+(j+2)-\lambda^1_{j+2}})
   \ =\ \dim(H\cap F^1_{b-a+j-\nu^1_{j}+5})\,,
 \]
 and thus $j=\dim (H/h\cap G^1_{b-a+j-\nu^1_{j}})$, which shows~\eqref{Eq:needed_implications} and completes the proof of
 parts (2) and (3) of Definition~\ref{D:fibration}, and thus establishes statement (i).\medskip

 For the statement (ii),  let $\calFdot\in(\Fl_{5{+}b})^s$ be general flags and  suppose that $H\in\Omega_{\lambda^i}\Fdot^i$ for
 $i\in\{1,\dotsc,5\}$.
 Then $\dim (H\cap F^1_2)=\dim (H\cap F^2_3)=1$.
 If $\defcolor{V}:=\langle F^1_2,F^2_3\rangle\simeq \CC^5$ then $\defcolor{h}:=H\cap V\in\Gr(2,V)$.
 If $\bmu=(\T,\I^4)$ and $\defcolor{\Edot^i}:=\Fdot^i\cap V$ for $i\in\{1,\dotsc,5\}$, then similar arguments as
 before show that $h\in\Omega_{\bmu}\calEdot$, and establish part (1) and the first half of part (4) of
 Definition~\ref{D:fibration}. 

 For parts (2), (3), and (4), we first identify $W(h)$ and the flags $\calGdot(h)$.
 For $h\in\Omega_{\bmu}\calEdot$, we define $\defcolor{U(h)}\simeq\CC^{b+2}$ to be the intersection of $h+F^i_{b+3}$ for
 $i\in\{3,4,5\}$.
 We set $\defcolor{W(h)}:=U(h)/h$ and then set $\defcolor{\Gdot^i(h)}:=((h+\Fdot^i)\cap U(h))/h$ for $i\in\{1,\dotsc,s\}$.
 As before, all flags in $W(h)$ may occur in this way, which proves part (4) of Definition~\ref{D:fibration}.
 Arguments using linear algebra, the Schubert conditions that $h$ satisfies, and the generality of the flags $\calFdot$
 show that for $1\leq c\leq b$ we have
 \begin{enumerate}
  \item $G^i_c(h)=((h+F^i_{c+3})\cap U(h))/h$ for $i\geq 6$,
  \item $G^i_c(h)=((h+F^i_{c+2})\cap U(h))/h$ for $i\in\{3,4,5\}$, and 
  \item $G^i_c(h)=((h+F^i_{c+4})\cap U(h))/h$ for $i\in\{1,2\}$.
 \end{enumerate}
 (These arguments are similar to those in (i) and in the proof of Theorem~\ref{T:First_Family}.)
 As before, from this it may be deduced that~\eqref{Eq:needed_implications} holds,  which establishes parts (2) and
 (3), and completes the proof.
\end{proof}
%%%%%%%%%%%%%%%%%%%%%%%%%%%%%%%%%%%%%%%%%%%%%%%%%%%%%%%%%%%%%%%%%%%%%%%%%%%%%%%%%

Theorems~\ref{T:First_Family} and~\ref{Th:Second_Family} are Theorems 3.8, 3.9, and 3.10 of the thesis~\cite{Williams}.

%%%%%%%%%%%%%%%%%%%%%%%%%%%%%%%%%%%%%%%%%%%%%%%%%%%%%%%%%%%%%%%%%%%%%%%%%%%%%%%%%
\begin{proof}[Proof of Corollary~\ref{C:Second_Family}]
  We build on ideas from the proof of Corollary~\ref{C:First_Family}, following its outline.
  In each case, we start with a subset $\calY\subset(\Fl_9)^s$ of flags for which the auxiliary Schubert problem $\bmu$ in $\Gr(2,V)$ is
  constant with solutions $\Omega_{\bmu}\calEdot$. 
  That is, some subspaces in some of the flags are fixed to induce a constant Schubert problem, while the others may be chosen freely
  given those constraints.
  The Schubert problem of the fiber above each $h\in\Omega_{\bmu}\calEdot$ is the problem $\I^4$ in $\Gr(2,W(h))$.
  We produce a subfamily $\calY'\subset\calY$ of instances that depends upon the choice of a general pencil and which
  induces a constant Schubert problem for all $h$
  except one, and for that one, three of the conditions are fixed while the fourth moves in a pencil.
  By the discussion following Figure~\ref{F:4lines} and similar to arguments given for Corollary~\ref{C:First_Family}, this
  implies that the Galois group is as claimed.

  An instance $\calFdot$ of a Schubert problem $\blambda$ in $\Gr(4,9)$ in Theorem~\ref{Th:Second_Family}(i) is fibered over
  $\raisebox{-5.5pt}{\TI}\cdot\I^3$ in $\Gr(2,F^1_5)$.
  Setting $V=F^1_4$, this reduces to $\I^4$ in $\Gr(2,V)$, given by the four 2-planes $E^1_2=F^1_2$ and $E^i_2:=F^i_7\cap V$
  for $i\in\{2,3,4\}$. 
  Fix $V\simeq\CC^4$ and $\CC^2\simeq E^i_2$ for $i\in\{1,\dotsc,4\}$ in general position in $V$.
  Let $\calY\subset(\Fl_9)^s$ be all flags $\calFdot$ such that  $F^1_2=E^1_2$, $F^1_4=V$,  and for  $i\in\{2,3,4\}$,
  $E^i_2:=F^i_7\cap V$.
  Then every flag in $\calY$ induces the same auxiliary problem $\Omega_{\sI^4}\calEdot$  in $\Gr(2,V)$,
  where $\Edot^i=\Fdot^i\cap V$ for $i\in\{1,\dotsc,4\}$.
  For a solution $h$ to  $\Omega_{\sI^4}\calEdot$, we define
  \[
    \defcolor{U(h)}\ :=\ (h+F^2_7)\cap (h+F^3_7)\cap (h+F^4_7)\,,
  \]
  which is a 6-plane that contains $h$, so that $\defcolor{W(h)}:=U(h)/h$ is 4-dimensional.
  Define the flags $\defcolor{\Gdot^i(h)}:= \bigl( (h+\Fdot^i)\cap U(h)\bigr)/h$, for $i\in\{1,\dotsc,s\}$.

    As the solutions $h_1$ and $h_2$ to  $\Omega_{\sI^4}\calEdot$ are independent, $V=h_1\oplus h_2$.
  This independence and the descriptions of the flags $\Gdot^i(h_a)$ enable arguments similar to those in the proof of
  Corollary~\ref{C:First_Family} 
  which show the existence of a subfamily $\calY'\subset\calY$ of instances such that the flag
  $\calGdot(h_1)$ is constant while $\calGdot(h_2)$ moves in a general pencil and hence induces a transposition.
 % While $W(h_1)\neq W(h_2)$, as explained in the proof of Theorem~\ref{Th:Second_Family}(i), both are canonically
 % isomorphic to $\CC^9/F^1_5$, which enables the comparison of $G^{1i}_3$ to $G^{2i}_3$ as in the proof of
 % Corollary~\ref{C:First_Family}.    

 Now let $\blambda$ be a Schubert problem in Theorem~\ref{Th:Second_Family}(ii) fibered over $\T\cdot\I^4$ with fiber
 $\I^4$.
 As the three solutions to $\T\cdot\I^4$ in $\Gr(2,5)$ are not linearly independent, we must modify our arguments.
 Let $F^1_2, F^2_3, F^3_7, F^4_7$, and $F^5_7$ be general linear subspaces of the indicated dimensions, and let
 $\calY\subset(\Fl_9)^s$ be the set of flags $\calFdot$ that are general given that those subspaces are
 fixed.
 Then if $\defcolor{V}:=\langle F^1_2, F^2_3\rangle$ and if we set $\defcolor{E^1_2}:=F^1_2$, $\defcolor{E^2_3}:=F^2_3$, and
 $\defcolor{E^i_3}:=F^i_7\cap V$ for $i\in\{3,4,5\}$, then we have the auxiliary problem
 \[
   \Omega_{\sT}E^1_2\,\cap\,
   \Omega_{\sI}E^2_3\,\cap\,
   \Omega_{\sI}E^3_3\,\cap\,
   \Omega_{\sI}E^4_3\,\cap\,
   \Omega_{\sI}E^5_3\,,
 \]
 which is constant for $\calFdot\in\calY$ and has three solutions $h_1,h_2,h_3$.
 For each $i\in\{1,2,3\}$, let $\defcolor{\ell_i}:= h_i\cap E^1_2$ and $\defcolor{m_i}:= h_i\cap E^2_3$, both of which are 1-dimensional.
 Observe that the span $M$ of $m_1,m_2,m_3$ is $E^2_3$, for otherwise $h_i\in\Omega_{\sT}M$, which contradicts the
 generality of $\calFdot$.
 Consequently,  the three solutions $h_1,h_2,h_3$ are linearly independent modulo $E^1_2$. 
 We will construct $\calY'\subset\calY$ as desired by fixing all subspaces in the flags in a given general
 $\calFdot\in\calY$ except one.
 We use the notation from the proof of Theorem~\ref{Th:Second_Family}.
 
 Suppose first that $\nu^i=\I$ for some $i\geq 6$.
 We may assume that $\nu^6=\I$, so that $\lambda^6=\I$, and for $h\in\{h_1,h_2,h_3\}$,
 $G^6_2(h)=(h+F^6_5)\cap W(h)$.
 The map $V\ni v\mapsto (v+F^6_5)\cap W(h)$ is a linear surjection from $V$ to $W(h)$ that depends upon $F^6_5$.
 In fact all surjections $V\twoheadrightarrow W(h)$ occur if we replace $F^6_5$ by other five-dimensional subspaces.
 Set $V':=\langle h_1,h_2\rangle\simeq\CC^4$.
 Then $\ell_3=h_3\cap V'$ as $E^1_2\subset V'$.
 Let $F^6_5(t)$ vary in a pencil so that the resulting maps to the $W(h_i)$ are constant on $V'$, but the map to $W(h_3)$ is
 not constant.
 This is possible as $h_3\not\subset V'$.
 Then the image of $h_3$ in $W(h_3)$ moves in a pencil containing the image of $\ell_3$, but the images of $h_1$ and $h_2$
 are fixed in this pencil.
 Letting $\Fdot^6(t)$ be a pencil of flags containing  $F^6_5(t)$ gives the desired family $\calY'\subset\calY$, and shows
 that the  kernel $\Gamma$ of the map  $\Gal_{\blambda}\twoheadrightarrow S_3$ has an element of the form
 $\{e\}\times\{e\}\times\sigma$, for $\sigma$ a  transposition, and thus $\Gal_{\blambda}=S_2\wr S_3$.

 Suppose now that $\nu^i=0$ for  $i\geq 6$.
 Then $\blambda$ is one of
 \[
 \sThI\cdot\shFI\cdot\raisebox{-2pt}{\shIII}\cdot\raisebox{-2pt}{\scTII}\cdot\raisebox{-2pt}{\scTII}
 \ \mbox{ or }\ 
 \sThI\cdot\shF\cdot\raisebox{-2pt}{\scTII}\cdot\raisebox{-2pt}{\scTII}\cdot\raisebox{-2pt}{\scTII}
 \ \mbox{ or }\ 
 \shTh\cdot\shFI\cdot\raisebox{-2pt}{\scTII}\cdot\raisebox{-2pt}{\scTII}\cdot\raisebox{-2pt}{\scTII}\,.
 \]
 We show that each has monodromy group $S_2\wr S_3$ by a direct symbolic computation that is documented in a Maple script that
 is available on our webpage\footnote{{\tt http://www.math.tamu.edu/\~{}sottile/research/stories/GIVIX}}.
 We fix flags $\Fdot^1,\dotsc,\Fdot^4$ and give a pencil $\Fdot^5(t)$ of flags, all in $\QQ^9$.
 For these flags, the auxiliary problem $\T\cdot\I^4$ has three solutions
 $h_1,h_2,h_3\in\Gr(2,\QQ^5)$, independent of $t$.
 Thus the monodromy group of this family over $\CC$ (the coordinate $t$) is a subgroup $\Gamma$ of the kernel
 $\Gal_{\blambda}\twoheadrightarrow S_3$.
 Note that $\Gamma\subset S_2\times S_2\times S_2$, where the $i$th factor is the monodromy in the copy of the Schubert
 problem $\I^4$ constituting the fiber over $h_i$.

 For each solution $h_i$, we solve the problem in the fiber over the field $\QQ(t)$, and compute its discriminant
 $\delta_i(t)\in\QQ[t]$.
 The local monodromy around any root of $\delta_i(t)$ is the transposition $(12)$ in the $i$th factor of
 $\Gamma$.
 The computation shows that these discriminants are pairwise relatively prime, which implies that $\Gamma$ contains elements
 $((12),e,e)$, $(e,(12),e)$, and  $(e,e,(12))$, and therefore  $\Gamma=S_2\times S_2\times S_2$ and
 $\Gal_\blambda= S_2\wr S_3$.
\end{proof}
%%%%%%%%%%%%%%%%%%%%%%%%%%%%%%%%%%%%%%%%%%%%%%%%%%%%%%%%%%%%%%%%%%%%%%%%%%%%%%%%%

%\newpage
%%%%%%%%%%%%%%%%%%%%%%%%%%%%%%%%%%%%%%%%%%%%%%%%%%%%%%%%%%%%%%%%%%%%%%%%%%%%%%%%%
\subsection{Fibrations of Type II}\label{S:typeII}

The remaining 28 essential enriched Schubert problems in $\Gr(4,9)$ are fibrations in an essentially different manner than
in~\S\ref{S:typeI}.
Each is a special case of one of four related general constructions.
It remains an open problem to find a unifying construction for all four, such as what was accomplished for type I
fibrations in the subsequent paper~\cite{SWY}.

%%%%%%%%%%%%%%%%%%%%%%%%%%%%%%%%%%%%%%%%%%%%%%%%%%%%%%%%%%%%%%%%%%%%%%%%%%%%%%%%%
\begin{theorem}\label{T:ThirdFamily}
  Suppose that $a < b$ and
  $\bnu=((b{-}a{-}1,\kappa),1^{a-1}+\rho,\nu^3,\nu^4,\dots,\nu^s)$ is a Schubert problem in $\Gr(a,b)$. 
 Then
 \[
    \blambda =  \bigl((b{-}a,b{-}a,\kappa)\,,\, 
     2^a+\rho\,,\, (b{-}a{+}1,\nu^3)\,,\, 1^{a+1}+\nu^4\,,\, \nu^5, \dotsc, \nu^s\bigr)
\]
is a Schubert problem in $\Gr(2{+}a, 4{+}b)$ that is fibered over $\I^4$ in $\Gr(2,4)$ with fiber $\bnu$.
\end{theorem}
%%%%%%%%%%%%%%%%%%%%%%%%%%%%%%%%%%%%%%%%%%%%%%%%%%%%%%%%%%%%%%%%%%%%%%%%%%%%%%%%%

As $\bnu$ is a Schubert problem in  $\Gr(a,b)$, we have that $\kappa,\rho$ are partitions with $\kappa_1,\rho_1\leq b{-}a{-}1$ and
$\kappa_a=\rho_a=0$.

When $a=2$ and $b=5$, Theorem~\ref{T:ThirdFamily} gives 18 enriched Schubert problems fibered over $\I^4=2$.
Since $b{-}a{-}1=2$ and $a{-}1=1$, the first condition in the fiber $\bnu$ is either $\T$ or
\raisebox{-5.5pt}{$\TI$} or \raisebox{-5.5pt}{$\mTT$} and the second is either $\T$ or $\I$.
The only such $\bnu$ with $d(\bnu)>1$ are
\[
  \T^2\cdot\I^2\ =\ 2\,,\quad
  \raisebox{-5.5pt}{\TI}\cdot\I^3\ =\ 2\,,\
  \quad\mbox{and}\quad\ 
    \T\cdot\I^4\ =\ 3\,.
\]
These give eleven, four, and four enriched Schubert problems in $\Gr(4,9)$, respectively.
Here is one Schubert problem from each of these families.
\[
  \raisebox{-7.5pt}{\ThTh}\cdot\raisebox{-7.5pt}{\TT}\cdot\raisebox{-7.5pt}{\FT}\cdot\raisebox{-14.5pt}{\bTII}\,,\qquad
  \raisebox{-14.5pt}{\ThThI}\cdot\raisebox{-7.5pt}{\TT}\cdot\raisebox{-7.5pt}{\FI}\cdot\raisebox{-14.5pt}{\bTII}\,,\qquad
  \raisebox{-7.5pt}{\ThTh}\cdot\raisebox{-7.5pt}{\TT}\cdot\raisebox{-7.5pt}{\FI}\cdot\raisebox{-14.5pt}{\bTII}\cdot\bI\,.
\]
For nearly the same reasons as Corollary~\ref{C:First_Family}, these have Galois groups $S_2\wr S_2$, $S_2\wr S_2$, and 
$S_3\wr S_2$, respectively.

%%%%%%%%%%%%%%%%%%%%%%%%%%%%%%%%%%%%%%%%%%%%%%%%%%%%%%%%%%%%%%%%%%%%%%%%%%%%%%%%%
\begin{corollary}\label{C:ThirdFamily}
  These 19 Schubert problems all have the claimed Galois groups over $\CC$.
\end{corollary}
%%%%%%%%%%%%%%%%%%%%%%%%%%%%%%%%%%%%%%%%%%%%%%%%%%%%%%%%%%%%%%%%%%%%%%%%%%%%%%%%%

%%%%%%%%%%%%%%%%%%%%%%%%%%%%%%%%%%%%%%%%%%%%%%%%%%%%%%%%%%%%%%%%%%%%%%%%%%%%%%%%%
\begin{proof}
  This may be proven using arguments similar to those in the proof of Corollary~\ref{C:First_Family}.
  Fix a four-dimensional subspace $V\subset\CC^9$ and four general 2-planes $E^1_2,E^2_2,E^3_2,E^4_2$ in $V$.
  Let $\calY\subset(\Fl_9)^s$ be the subset of flags $\calFdot$ such that
 \[
   F^3_2=E^3_2\,,\
   V=\langle F^3_2,F^2_4\rangle \cap \langle F^3_2,F^2_5\rangle\,,\
   E^1_2=F^1_4\cap V\,,\
   E^2_2=F^2_5\cap V\,,\ \mbox{ and }\ 
   E^4_2=F^4_7\cap V\,.
 \] 
 For a flag $\calFdot$ in $\calY$ and $i\in\{1,\dotsc,4\}$, let $\Edot^i=\Fdot^i\cap V$.

  As will be shown in the proof of Theorem~\ref{T:ThirdFamily}, for any flag $\calFdot$ in $\calY$, $\blambda$ one of the 19 problems,
  and $H\in \Omega_{\blambda}\calFdot$, the $2$-plane $h:=H\cap V$ is a solution to $\Omega_{\sI^4}\calEdot$, which is the problem of four
  lines
  \begin{equation}\label{Eq:anotherOne}
    \Omega_{\sI} E^1_2\ \cap\ 
    \Omega_{\sI} E^2_2\ \cap\ 
    \Omega_{\sI} E^3_2\ \cap\ 
    \Omega_{\sI} E^4_2\ .
  \end{equation}
  Let $h_1,h_2$ be the two solutions to~\eqref{Eq:anotherOne}.
  Define $\defcolor{W}\simeq \CC^5$ to be the span of the two spaces
  $G^1_2:=F^4_7\cap F^1_4$ and $G^2_3:=F^4_7 \cap F^2_5$, which are in direct sum.
  For each $i\in\{1,\dotsc,s\}$ and solution $h$ to~\eqref{Eq:anotherOne}, define $\Gdot^i(h)$ to be the flag $(h+\Fdot^i)\cap W$.
  (See~\eqref{Eq:FlagsII} for more information on these flags).
  
  Then, for every $i\in\{1,\dotsc,s\}$, there is a pencil in $\calY$ of flags such that the induced flags $\calGdot(h_1)$ are fixed,
  but $G^i_3(h_2)$ moves in a pencil while all other flags $\Gdot^j(h_2)$ with $j\neq i$ is fixed, and that pencil induces a simple
  transposition in $\Omega_{\bnu}\calGdot(h_2)$.
  As argued in the proof of Corollary~\ref{C:First_Family}, this suffices to complete the proof.
\end{proof}
%%%%%%%%%%%%%%%%%%%%%%%%%%%%%%%%%%%%%%%%%%%%%%%%%%%%%%%%%%%%%%%%%%%%%%%%%%%%%%%%%

%%%%%%%%%%%%%%%%%%%%%%%%%%%%%%%%%%%%%%%%%%%%%%%%%%%%%%%%%%%%%%%%%%%%%%%%%%%%%%%%%
\begin{proof}[Proof of {Theorem~\ref{T:ThirdFamily}}]
 First observe that as $\bnu$ is a Schubert problem in $\Gr(a,b)$, we have
  \[
    a(b-a)\ =\ |\bnu|\ =\ b{-}a{-}1{+}|\kappa| + a{-}1{+}|\rho|+\sum_{i=3}^s |\nu^i|
    \ =\  b{-}2+|\kappa|+|\rho|+\sum_{i=3}^s |\nu^i|\ .
  \]
  Then
  \begin{eqnarray*}
    |\blambda| &=& 2(b{-}a){+}|\kappa|+2a{+}|\rho|+(b{-}a{+}1){+}|\nu^3| + (a{+}1){+}|\nu^4| + \sum_{i=5}^s|\nu^i|\\
               &=& |\kappa|+|\rho|+\sum_{i=3}^s |\nu^i| \  + 2(b{-}a)+2a+(b{-}a{+}1) + (a{+}1)\\
               &=& a(b-a)-(b{-}2) +  2(b{-}a)+2a+(b{-}a{+}1) + (a{+}1)\,, 
   \end{eqnarray*}
  % simplify( a*(b-a)-(b-2) +  2*(b-a)+2*a+(b-a+1) + (a+1) - (a+2)*(b-a+2));
  which equals $(a+2)(b-a+2)$, so that $\blambda$ is a Schubert problem in $\Gr(a{+}2,b{+}4)$.
  
  Let $\calFdot=(\Fdot^1,\dotsc,\Fdot^s)$ be general flags in $\CC^{b+4}$.
  Let $\defcolor{V}:=\langle F^3_2,F^1_4\rangle\cap\langle F^3_2,F^2_b\rangle\simeq\CC^4$, and for $i\in\{1\dotsc,4\}$,
  define $\Edot^i:=\Fdot^i\cap V$.
  In particular $E^1_2=F^1_4\cap V$, $E^2_2=F^2_b\cap V$,   $E^3_2=F^3_2$, and  $E^4_2=F^4_{b+2}\cap V$.
  These give an instance of $\I^4$ in $\Gr(2,V)$.
  Indeed, from the definition~\eqref{Eq:SchubertCondition} of a Schubert variety, we see that if $H$ lies in 
  \begin{equation}\label{Eq:bigFour}
    \Omega_{(b-a)^2}\Fdot^1\,\cap\,
    \Omega_{2^a}\Fdot^2\,\cap\,
    \Omega_{(b-a+1)}\Fdot^3\,\cap\,
    \Omega_{1^{a+1}}\Fdot^4\,,
  \end{equation}
  then $h=H\cap V$ has dimension 2 and we have that $\dim( h\cap E^i_2)=1$ for $i\in\{1,\dotsc,4\}$.
  Thus $h\in\Omega_{\sI^4}\calEdot$ is a solution to the Schubert problem $\I^4=2$ given by the flags $\calEdot$.
  As any flags $\calEdot$ in $V$ may occur in this way, this shows part (1) and the first half of part (4) of
  Definition~\ref{D:fibration}.

  Set $\defcolor{G^1_2}:=F^4_{b+2}\cap F^1_4\simeq \CC^2$  and
  $\defcolor{G^2_{b-2}}:=F^4_{b+2}\cap F^2_{b}\simeq \CC^{b-2}$, and set
  $\defcolor{W}:= G^1_2\oplus G^2_{b-2} \simeq \CC^b$.
  If $H$ lies in the intersection~\eqref{Eq:bigFour} (which contains $\Omega_{\blambda}\calFdot$), then $H\cap W$ has
  dimension $a$ and lies in $\Omega_{(b-a-1)}G^1_2\cap\Omega_{1^{a-1}}G^2_{b-2}$.

  Now let $h$ be a solution to  $\Omega_{\sI^4}\calEdot$.
  For each $i\in\{1,\dotsc,s\}$, define $\Gdot^i(h)$ to be $(h+\Fdot^i)\cap W$.
  Then, except for $G^1_1,G^2_2,G^2_{b-1}$, and $G^2_{b-2}$, we claim that for $1\leq c\leq b$,
 \begin{equation}\label{Eq:FlagsII}
   \begin{tabular}{l}
    $G^i_c(h) = (h+F^i_{c+2}) \cap W$ for $i\geq 5$,\\
    $G^i_c(h) = (h+F^i_{c+1}) \cap W$ for $i=2,4$, and \rule{40pt}{0pt} \\
    $G^i_c(h) = (h+F^i_{c+3}) \cap W$ for $i=1,3$.
  \end{tabular}
 \end{equation}
 The arguments for~\eqref{Eq:FlagsII} when $i\geq 4$ or $i=3$ are the same as in the proof of Theorem~\ref{T:First_Family},
 except that the case $i=3$ here is case $i=1$ in that proof.
 
 Suppose that $i=1$.
 We defined $G^1_2:= F^1_4\cap W$ and we set $G^1_1:=F^1_3\cap W$.
 Let $3\leq c<b$.
 By the generality of $\calFdot$, $h\cap F^1_4=h\cap F^1_{c+3}$ has dimension 1, so that $\dim(h+F^1_{c+3})=c{+}4$.
 As $W$ has codimension 4, and is in linear general position modulo $G^1_2\subset  F^1_4$,~\eqref{Eq:FlagsII} follows for
 $i=1$.  

 Now suppose $i=2$ and let $1\leq c<b{-}2$.
 Since $\dim(h\cap F^2_b)=1$ and $h\cap F^2_{b-1}=\{0\}$ (by general position), we have
 $\dim((h+F^2_{c+1})\cap F^2_b)=c{+}2$.
 As $W=F^2_b\cap F^4_{b+2}$, we conclude that $c=\dim((h+F^2_{c+1}) \cap W)$, which implies~\eqref{Eq:FlagsII} for $i=2$.
 If we fix $h$ and consider all flags $\calFdot$ that induce $\calEdot, G^1_2$, and $G^2_{b-2}$, we obtain all flags $\calGdot$ in $W$
 with $\Gdot^1(h)$ containing $G^1_2$ and the same for $\Gdot^2(h)$ and $G^2_b$.
 This completes the proof of (4) in Definition~\ref{D:fibration}.

 We establish parts (2) and (3) in Definition~\ref{D:fibration} together by showing that for each $i\in\{1,\dotsc,s\}$, 
 \begin{equation}\label{Eq:equivII}
   K\ \in\ \Omega_{\nu^i}\Gdot^i(h)\ 
   \Longleftrightarrow\ h\oplus K\in\Omega_{\lambda^i}\Fdot^i\,.
 \end{equation}
  For $i\geq 3$, the argument is identical to the arguments given in the proof of Theorem~\ref{T:First_Family} (with the proviso that $i=3$
  here is the case $i=1$ in that proof).
  While it is nearly the same for $i=1,2$, we give the arguments.

  Suppose that $i=1$.
  As $\nu^1_1=b{-}a{+}1$, we show that that condition~\eqref{Eq:SchubertCondition}
  holds for $j=1$ in~\eqref{Eq:equivII}.
  If $H=K\oplus h$ lies in the intersection~\eqref{Eq:bigFour} (which contains $\Omega_{\blambda}\calFdot$), then $K=H\cap W$ lies in
  $\Omega_{(b-a-1)}G^1_2\subset \Omega_{\nu^1}\Gdot^1(h)$, and the reverse implication is similar.

  Let us complete the case of $i=1$. 
  We first observe that for $K\in\Gr(a,W)$ and $3\leq c\leq b$, we have
  $K\cap G^1_c(h)=K\cap W\cap(h+F^1_{c+3})=K\cap (h+F^1_{c+3})$.
  Furthermore, if we set $H:=h\oplus K$ then
  \[
    2+\dim(K\cap G^1_c(h))\ =\ 2 + \dim( K\cap (h+F^1_{c+3}))\ =\ \dim(H\cap(h+ F^1_{c+3}))\,.
  \]
  The last equality is because $h\cap K=\{0\}$.
  We also have that 
  \[
    1+\dim(H\cap F^1_{c+3})\ =\ \dim(H\cap(h+ F^1_{c+3}))\,,
  \]
  as $\dim(h\cap F^1_{c+3})=1$.
  Thus $\dim(H\cap F^1_c)=1+\dim(K\cap G^1_c(h))$.
  Now suppose that $j>1$.
  Then $\lambda^1_{j+1}=\nu^1_j$.  
  Then the condition~\eqref{Eq:SchubertCondition} for $j{+}1$ is
  \[
  1{+}j\ =\ \dim(H\cap F^1_{(b+4)-(a+2) +j+1-\lambda^1_{j+1}})
  \ =\ \dim(H\cap F^1_{b-a +j-\nu^1_j +3})\,,
  \]
  so that $j=\dim(K\cap G^1_{b-a +j-\nu^1_j})$, which proves~\eqref{Eq:equivII} when $i=1$.  

  Let $i=2$ and suppose that $1\leq c <b-2$.
  For $K\in\Gr(a,W)$, we again have that $K\cap G^2_c(h)= K\cap (h+F^2_{c+1})$.
  If we set $H:=K\oplus h$, then
  \[
  2+ \dim (K\cap G^2_c(h) )\ =\ 2+\dim (K\cap (h+F^2_{c+1}))\ =\ \dim (H\cap (h+F^2_{c+1}))\,,
  \]
  as $h\cap F^2_{c+1}=\{0\}$, we also have
  \[
  2+ \dim (H\cap F^2_{c+1}) )\ =\  \dim (H\cap (h+F^2_{c+1}))\,.
  \]
  We thus have $\dim (K\cap G^2_c(h) )=\dim (H\cap F^2_{c+1}) )$.
  For $1\leq j<a$, as $\lambda^2_j=1+\nu^2_j$, we have
  \[
  \dim(K\cap G^2_{b-a+j-\nu^2_j}(h))\ =\
  \dim(H\cap F^2_{b-a+j-\nu^2_j+1})\ =\
  \dim(H\cap F^2_{(b+4)-(a+2)+j-\lambda^2_j})\,.
  \]
  
  Finally, as $\nu^2_a=0$, the corresponding condition on $K$ in~\eqref{Eq:equivII} for $i=2$ is that $K\subset W$.
    As $\lambda^2_a=2$, the condition on $H=h\oplus K$ is  $\dim(H\cap F^2_{(b+4)-(a+2)+a-2})\geq a$,
    That is, $\dim(H\cap F_b)\geq a$.
      As $E^2_2=F^2_b\cap V$, we have $1=\dim(h\cap E^2_2)\leq\dim(h\cap F^2_b)$   and by case $j=a{-}1$ and $i=2$ of~\eqref{Eq:equivII},
      we have $a{+}1=\dim(K\cap G^2_{b-2}(h))<\dim(K\cap F^2_b)$ as $G^2_{b-2}(h)=F^4_{b+2}\cap F^2_b$.
      Thus $\dim(H\cap F^2_b)\geq a$, which 
 completes the proof.    
\end{proof} 
%%%%%%%%%%%%%%%%%%%%%%%%%%%%%%%%%%%%%%%%%%%%%%%%%%%%%%%%%%%%%%%%%%%%%%%%%%%%%%%%%

We give a related construction of enriched Schubert problems.

%%%%%%%%%%%%%%%%%%%%%%%%%%%%%%%%%%%%%%%%%%%%%%%%%%%%%%%%%%%%%%%%%%%%%%%%%%%%%%%%%
\begin{theorem}\label{Th:SecondTypeII}
  Suppose that $a < b$ and
  $\bnu=(\nu^1,(b{-}a{-}1,\kappa),1^{a-1}+\rho,\nu^4,\dots,\nu^s)$ is a Schubert problem in $\Gr(a,b)$. 
 Then
 \[
    \blambda =  \bigl((b{-}a{+}2, \nu^1)\,,\,  (b{-}a,b{-}a,\kappa)\,,\, 
     2^a+\rho\,,\, 1^{a+1}+\nu^4\,,\, 1^{a+1}+\nu^5\,,\, \nu^6, \dotsc, \nu^s\bigr)
\]
is a Schubert problem in $\Gr(2{+}a, 5{+}b)$ fibered over $\T\cdot\I^4$ in $\Gr(2,5)$ with fiber $\bnu$.
\end{theorem}
%%%%%%%%%%%%%%%%%%%%%%%%%%%%%%%%%%%%%%%%%%%%%%%%%%%%%%%%%%%%%%%%%%%%%%%%%%%%%%%%%

As $\bnu$ is a Schubert problem in $\Gr(a,b)$, $\kappa,\rho$ are partitions with $\kappa_1,\rho_1\leq b{-}a{-}1$ and
$\kappa_a=\rho_a=0$. 

When $a=2$ and $b=4$, Theorem~\ref{Th:SecondTypeII} gives five Schubert problems in $\Gr(4,9)$ with fiber $\I^4$.
These problems are
\[
    \raisebox{-7.5pt}{\includegraphics{figures/41}} \cdot
    \raisebox{-7.5pt}{\includegraphics{figures/22}} \cdot
    \raisebox{-7.5pt}{\includegraphics{figures/22}} \cdot
    \raisebox{-14.5pt}{\includegraphics{figures/211}} \cdot
    \raisebox{-14.5pt}{\includegraphics{figures/111}}\ ,
    \quad
    \includegraphics{figures/4} \cdot
    \raisebox{-7.5pt}{\includegraphics{figures/22}} \cdot
    \raisebox{-7.5pt}{\includegraphics{figures/22}} \cdot
    \raisebox{-14.5pt}{\includegraphics{figures/211}} \cdot
    \raisebox{-14.5pt}{\includegraphics{figures/111}}\cdot\bI \ ,
    \quad
    \raisebox{-7.5pt}{\includegraphics{figures/41}} \cdot
    \raisebox{-7.5pt}{\includegraphics{figures/22}} \cdot
    \raisebox{-7.5pt}{\includegraphics{figures/22}} \cdot
    \raisebox{-14.5pt}{\includegraphics{figures/111}}\cdot
    \raisebox{-14.5pt}{\includegraphics{figures/111}}\cdot\bI \ ,
\]
and two others.

%%%%%%%%%%%%%%%%%%%%%%%%%%%%%%%%%%%%%%%%%%%%%%%%%%%%%%%%%%%%%%%%%%%%%%%%%%%%%%%%%
\begin{corollary}\label{C:SecondTypeII}
  These five Schubert problems in $G(4,9)$ from  Theorem~\ref{Th:SecondTypeII} have Galois group $S_2\wr S_3$ over $\CC$.
\end{corollary}
%%%%%%%%%%%%%%%%%%%%%%%%%%%%%%%%%%%%%%%%%%%%%%%%%%%%%%%%%%%%%%%%%%%%%%%%%%%%%%%%%

We omit the argument; it is essentially the same as that for the family fibered over
$\T\cdot\I^4$  given in Corollary~\ref{C:Second_Family}.

%%%%%%%%%%%%%%%%%%%%%%%%%%%%%%%%%%%%%%%%%%%%%%%%%%%%%%%%%%%%%%%%%%%%%%%%%%%%%%%%%
\begin{proof}[Proof of Theorem~\ref{Th:SecondTypeII}]
  We derive the auxiliary Schubert problem $\T\cdot\I^4=3$ (part (1) of Definition~\ref{D:fibration})
  and sketch the derivation of the Schubert problem $\bnu$ in the fibers  (part (2) of Definition~\ref{D:fibration}), leaving the 
  remaining   verifications to the reader, as they are tediously similar to previous arguments.
  Let $\calFdot\in(\Fl_{b+5})^s$ be $s$ general flags and suppose that $H\in\Omega_{\blambda}\calFdot$.
  Then 
  \[
     \dim (H\cap F^1_2) =  1\,,\ 
     \dim (H\cap F^2_5) = 2\,,\ 
     \dim (H\cap F^3_{b+1})= a\,,\  \mbox{and}\ 
     \dim (H\cap F^i_{b+3}) = a{+}1\,,
   \]
  for $i\in\{4,5\}$.
  Let $\defcolor{V}:=\langle F^1_2,F^2_5 \rangle\cap \langle F^1_2,F^3_{b+1}\rangle\simeq\CC^5$ and let
  $\defcolor{\Edot^i}:= \Fdot^i\cap V$ for $i\in\{1,\dotsc,5\}$.
  We claim that $H\cap V\in\Gr(2,V)$ and is a solution of $\Omega_{\sT\cdot\sI^4}\calEdot$, that is,
  $\dim(H\cap V)=2$ and $H$ meets each of $E^1_2,E^2_3,\dotsc,E^5_3$  in a 1-dimensional subspace.

  Indeed, $E^1_2=F^1_2$, which meets $H$ in a 1-dimensional subspace.
  Also, $E^2_3=F^2_5\cap\langle F^1_2,F^3_{b+1}\rangle$---this has dimension 3 as 
  $\langle F^1_2,F^3_{b+1}\rangle$ has codimension 2 in $\CC^{b+5}$, and the flags are general.
  It meets $H$ in a 1-dimensional subspace, as $\dim( H\cap F^2_5)=2$ and
  $\langle F^1_2,F^3_{b+1}\rangle$ meets $H$ in dimension $a{+}1$ (codimension 1).
  Thus $V=E^1_2\oplus E^2_3$.
  As $H$ meets both $E^1_2$ and $E^2_3$  in 1-dimensional subspaces, $H\in\Gr(2,V)$ (that $H\cap V$ does not have dimension
  3 is due to the generality of $\calFdot$). 
  Then $E^3_3=F^3_{b+1}\cap\langle F^1_2,F^2_5 \rangle$ as $F^3_{b+1}$ has codimension 4.
  Since $H\cap\langle F^1_2,F^2_5 \rangle$ has dimension 3 and $H\cap F^3_{b+1}$ has codimension 2 in $H$, $\dim( H\cap E^3_3)=1$.
  Suppose that $i\in\{4,5\}$.
  Since $V$ has codimension $b$ in $\CC^{b+5}$, $\dim(F^i_{b+3}\cap V)=3$, so that $E^i_3=F^i_{b+3}\cap V$.
  As $H\cap F^i_{b+3}$ has codimension 1 in $H$, $\dim(H\cap E^i_3)=1$.
  This shows part (1) of Definition~\ref{D:fibration}, and the first half of part (4) follows by the same arguments as before.
  
  Given $h\in\Omega_{\sT\cdot\sI^4}\calEdot$, let $\defcolor{W(h)}\subset F^4_{b+3}\cap F^5_{b+3}$ be spanned by
  \[
    \defcolor{P}:=\langle h,F^2_5\rangle\cap F^4_{b+3}\cap F^5_{b+3}\simeq \CC^2
    \qquad\mbox{and}\qquad
    \defcolor{Q}:=\langle h,F^3_{b+1}\rangle\cap F^4_{b+3}\cap F^5_{b+3}\simeq \CC^{b-2}\,.
  \]
  These subspaces $P$ and $Q$ are in direct sum, $W(h)=P\oplus Q$. 
  (The dimension and direct sum claims follow as $h$ meets each of $F^2_5$ and $F^3_{b+1}$ in a 1-dimensional
  subspace, and the generality of the flags $\calFdot$).

  When $h\subset H\in\Omega_{\blambda}\calFdot$,
  $\dim(H\cap W(h))=a$.
  Indeed, $H$ meets $\langle h,F^2_5\rangle$ in a 3-plane and each of $\langle h,F^3_{b+1}\rangle$,  $F^{4}_{b+3}$, and
  $F^{5}_{b+3}$ in codimension 1 (has dimension $a{+}1$).
  Thus $\dim(H\cap P)=3{-}1{-}1=1$ and $\dim(H\cap Q)=a{+}1{-}1{-}1=a{-}1$.

  For $i\in\{1,\dotsc,s\}$, set $\defcolor{\Gdot^i(h)}:=(h+\Fdot^i)\cap W(h)$.
  Then the proof of the rest of Definition~\ref{D:fibration} is similar to arguments given before, particularly those for
  Theorem~\ref{T:ThirdFamily}.
\end{proof}
%%%%%%%%%%%%%%%%%%%%%%%%%%%%%%%%%%%%%%%%%%%%%%%%%%%%%%%%%%%%%%%%%%%%%%%%%%%%%%%%%

%\subsection{Two more families}
The five remaining enriched Schubert problems in $\Gr(4,9)$ come from two additional general constructions of
enriched problems.
While we explain their statements and the resulting Schubert problems in $\Gr(4,9)$, we will only sketch their proofs, and
only give the formal statement about their Galois groups.

%%%%%%%%%%%%%%%%%%%%%%%%%%%%%%%%%%%%%%%%%%%%%%%%%%%%%%%%%%%%%%%%%%%%%%%%%%%%%%%%%
\begin{theorem}\label{Th:SecondTypeIII}
  Suppose that  $a<b{-}1$ and
  \[
    \bnu\ =\ ((b{-}a{-}1, \kappa), (b{-}a{-}2, \rho), (2)^{a-1}+ \sigma, \nu^4,\dots,\nu^s)
  \]
  is a Schubert problem in $\Gr(a,b)$. 
  Then
  \[
    \blambda\ =\ \Bigl((b{-}a{-}1)^2,\kappa), (b{-}a{-}1)^2,\rho),
        \bigl( (3)^a, 2 \bigr) + \sigma, (1)^{a+1} + \nu^4, \nu^5, \dots, \nu^s  \Bigr)
  \]
  is a Schubert problem in $\Gr(2{+}a, 4{+}b)$ that is fibered over $\I^4$ in $\Gr(2,4)$ with fiber $\bnu$.
\end{theorem}
%%%%%%%%%%%%%%%%%%%%%%%%%%%%%%%%%%%%%%%%%%%%%%%%%%%%%%%%%%%%%%%%%%%%%%%%%%%%%%%%%

As  $\bnu$ is a Schubert problem in $\Gr(a,b)$, 
$\kappa,\rho,\sigma$ are partitions with $\kappa_1\leq b{-}a{-}1$,  $\rho_1,\sigma_1\leq b{-}a{-}2$,
and $\kappa_a=\rho_a=\sigma_a=0$. 

When $a=2$ and $b=5$, Theorem~\ref{Th:SecondTypeIII} gives two enriched Schubert problems fibered over $\I^4=2$ with fiber
$\T^2 \cdot \I^2$.
These problems are
\[
    \raisebox{-7.5pt}{\includegraphics{figures/22}} \cdot
    \raisebox{-7.5pt}{\includegraphics{figures/22}} \cdot
    \raisebox{-14.5pt}{\includegraphics{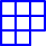}} \cdot
    \raisebox{-14.5pt}{\includegraphics{figures/211}}  
   \qquad
    \mbox{and}
    \qquad
   \raisebox{-7.5pt}{\includegraphics{figures/22}} \cdot
   \raisebox{-7.5pt}{\includegraphics{figures/22}} \cdot
   \raisebox{-14.5pt}{\includegraphics{figures/332}} \cdot
   \raisebox{-14.5pt}{\includegraphics{figures/111}} \cdot \bI\,.
\]
These correspond to $\bnu=(\T,\I,\T,\I)$ and $\bnu=(\T,\I,\T,0,\I)$, respectively.
For these, $\kappa=\rho=\sigma=0$.

%%%%%%%%%%%%%%%%%%%%%%%%%%%%%%%%%%%%%%%%%%%%%%%%%%%%%%%%%%%%%%%%%%%%%%%%%%%%%%%%%
\begin{corollary}
 These two problems have Galois group  $S_2 \wr S_2$ over $\CC$.
\end{corollary}
%%%%%%%%%%%%%%%%%%%%%%%%%%%%%%%%%%%%%%%%%%%%%%%%%%%%%%%%%%%%%%%%%%%%%%%%%%%%%%%%%

%%%%%%%%%%%%%%%%%%%%%%%%%%%%%%%%%%%%%%%%%%%%%%%%%%%%%%%%%%%%%%%%%%%%%%%%%%%%%%%%%
\begin{proof}[Proof sketch for {Theorem~\ref{Th:SecondTypeIII}}]
We derive the auxiliary Schubert problem $\I^4=2$ and sketch the derivation of the Schubert problem $\bnu$ in the
fibers, and leaving remaining verifications of Definition~\ref{D:fibration} to the reader.
Let $\calFdot\in(\Fl_{4+b})^s$ be general and suppose that $H\in\Omega_{\blambda}\calFdot$.
As $\Omega_{\blambda}\calFdot$ is a subset of
\[
   \Omega_{(b-a-1)^2}\Fdot^1\cap
   \Omega_{(b-a-1)^2}\Fdot^2\cap
   \Omega_{(3^a,2)}\Fdot^3\cap
   \Omega_{1^{a+1}}\Fdot^4\,,
\]
we have
 \begin{eqnarray*}
   &\dim(H\cap F^1_5)= 2\,,\ 
   \dim(H\cap F^2_5)= 2\,,\ &\\
   &\dim(H\cap F^3_{b+1})= a+1\,,\ 
   \dim(H\cap F^3_{b-1})= a\,,\ \mbox{ and }\
  \dim(H\cap F^4_{b+2})= a+1\,.&
 \end{eqnarray*}
 For $i\in\{1,2\}$, let $\defcolor{E_2^i}:=F^i_5\cap F^3_{b+1}$ and set $\defcolor{V}:=\langle E_2^1,E_2^2\rangle \simeq \CC^4$.
 Since $H\cap F^3_{b+1}$ has codimension 1 in $H$, $\dim(H\cap E^i_2)=1$ for $i\in\{1,2\}$, and so
 $\defcolor{h}:=H\cap V\in\Gr(2,V)$.

 For $i\in\{1,\dotsc,4\}$, set $\defcolor{\Edot^i}:=\Fdot^i\cap V$.
 As $\CC^4\simeq V\subset F^3_{b+1}$, $\dim( V\cap F^3_{b-1})\geq 2$, and it equals 2 by the general position of $\Fdot^1,\Fdot^2,\Fdot^3$.
 The dimension of $F^4_{b+2}\cap V$ is also 2, by the general position of  $\Fdot^1,\Fdot^2,\Fdot^3,\Fdot^4$. 
 Then $E^3_2= F^3_{b-1}\cap V$ and $E^4_2:=F^4_{b+2}\cap V$, and $h$ meets each in 1-dimensional subspaces.
 Thus, 
 \begin{equation}\label{Eq:TypeIIIint}
    h\ \in\ \Omega_{\sI} E^1_2 \,\cap\,
	\Omega_{\sI} E^2_2 \,\cap\,
	\Omega_{\sI} E^3_2 \,\cap\,
	\Omega_{\sI} E^4_2 \, \ =\ \Omega_{\sI^4}\calEdot\,.
 \end{equation}
 This shows part (1) of Definition~\ref{D:fibration} and the first half of part (4).
 
 Let $\defcolor{A}:= F^3_{b-1}\cap F^4_{b+2}\simeq\CC^{b-3}$ and $\defcolor{B}:=F^1_5\cap  F^4_{b+2}\simeq\CC^3$ and set
 $\defcolor{W}:=\langle A,B\rangle\simeq\CC^b$.
 Then $W$ is complimentary to $V$.
 Note that $\dim(H\cap A)=a{-}1$ and $\dim(H\cap B)=1$, so that $\dim(H\cap W)=a$.

 Let us define flags $\calGdot(h)$ in $W$.
 Set $\defcolor{\Gdot^1(h)}:= \Fdot^1 \cap W$, $\defcolor{\Gdot^3(h)}:= \Fdot^3 \cap W$, and
 for $i\neq 1,3$, set $\defcolor{\Gdot^i(h)}:=(h+\Fdot^i)\cap W$.  
 For a given $h\in\Omega_{\sI^4}\calEdot$, all general instances of flags in $W$ occur as $\calGdot(h)$, for flags
 $\calFdot$ that induce $V,W$ and $\calEdot$.
 This proves part (4) of  Definition~\ref{D:fibration}.
 For parts (2) and (3),  linear-algebra arguments will suffice as before.
\end{proof}
%%%%%%%%%%%%%%%%%%%%%%%%%%%%%%%%%%%%%%%%%%%%%%%%%%%%%%%%%%%%%%%%%%%%%%%%%%%%%%%%%

We give the second additional construction of enriched Schubert problems.

%%%%%%%%%%%%%%%%%%%%%%%%%%%%%%%%%%%%%%%%%%%%%%%%%%%%%%%%%%%%%%%%%%%%%%%%%%%%%%%%%
\begin{theorem}\label{Th:SecondTypeIV}
 Suppose that $a < b$ and
 $\bnu=( (2)^{a-1}, (b{-}a{-}1), \nu^3, \nu^4, \dots,\nu^s)$ is a Schubert problem in $\Gr(a,b)$. 
 Then
\[
  \blambda\ =\ \Bigl( (b{-}a{+}1, (2)^a), (b{-}a{-}1)^2 , 
    (1)^{a+1} + \nu^3, (1)^{a+1} + \nu^4, \nu^5, \dots, \nu^s  \Bigr)
\]
is a Schubert problem in $\Gr(2{+}a, 4{+}b)$ that is fibered over $\I^4$ in $\Gr(2,4)$ with fiber $\bnu$.
\end{theorem}
%%%%%%%%%%%%%%%%%%%%%%%%%%%%%%%%%%%%%%%%%%%%%%%%%%%%%%%%%%%%%%%%%%%%%%%%%%%%%%%%%

When $a=2$ and $b=5$, Theorem~\ref{Th:SecondTypeIV} gives three enriched Schubert problems fibered over $\I^4=2$ with fiber
$\T^2 \cdot \I^2$.
These problems are
\[
\raisebox{-14.5pt}{\includegraphics{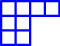}} \cdot
\raisebox{-7.5pt}{\includegraphics{figures/22}} \cdot
\raisebox{-14.5pt}{\includegraphics{figures/111}} \cdot
\raisebox{-14.5pt}{\includegraphics{figures/111}} \cdot \bI \cdot \bI
\]
and the two obtained by replacing one or two occurrences of\/ $\raisebox{-2pt}{\shIII}\cdot\sI$ with
$\raisebox{-2pt}{\scTII} \,$.

%%%%%%%%%%%%%%%%%%%%%%%%%%%%%%%%%%%%%%%%%%%%%%%%%%%%%%%%%%%%%%%%%%%%%%%%%%%%%%%%%
\begin{corollary}
 These three problems have Galois group  $S_2 \wr S_2$ over $\CC$.
\end{corollary}
%%%%%%%%%%%%%%%%%%%%%%%%%%%%%%%%%%%%%%%%%%%%%%%%%%%%%%%%%%%%%%%%%%%%%%%%%%%%%%%%%

%%%%%%%%%%%%%%%%%%%%%%%%%%%%%%%%%%%%%%%%%%%%%%%%%%%%%%%%%%%%%%%%%%%%%%%%%%%%%%%%%
\begin{proof}[Proof sketch for {Theorem~\ref{Th:SecondTypeIV}}]
We sketch the derivation of the auxiliary Schubert problem $\I^4=2$ and the Schubert problem $\bnu$ in the
fibers, and leave the verifications of Definition~\ref{D:fibration} to the reader.
Let $\calFdot\in(\Fl_{4+b})^s$ be general and suppose that $H\in\Omega_{\blambda}\calFdot$. 
Then
 \begin{eqnarray*}
  & \dim(H\cap F^1_2)= 1\,,\
  \dim(H\cap F^1_{b+1})= a+1\,,\ 
  \dim(H\cap F^2_5)= 2\,,&\\ 
  &\dim(H\cap F^3_{b+2})\ = \ \dim(H\cap F^4_{b+2})= a+1\,.&
 \end{eqnarray*}
 Let $\defcolor{E_2^1}:=F^1_2$ and $\defcolor{E_2^2}:=F^2_5\cap F^1_{b+1}$ and set
 $\defcolor{V}:=\langle E_2^1,E_2^2\rangle \simeq \CC^4$.
 Define $\Edot^1:=\Fdot^i\cap V$ for $i\in\{1,\dotsc,4\}$
 Then $E^i_2= F^i_{b+2}\cap V$ for $i\in\{3,4\}$ and $h:=H\cap V$ has dimension 2, and also 
 \begin{equation}\label{Eq:TypeIVint}
   h\ \in\ \Omega_{\sI} E^1_2 \,\cap\,
	\Omega_{\sI} E^2_2 \,\cap\,
	\Omega_{\sI} E^3_2 \,\cap\,
	\Omega_{\sI} E^4_2 \,  \ =\ \Omega_{\sI^4}\calEdot\,.
 \end{equation}
 This is the auxiliary problem and as before establishes part (1) and the first half of part (4)  of
 Definition~\ref{D:fibration}. 
 
 Set $\defcolor{W}:=F^3_{b+2}\cap F^4_{b+2}\simeq\CC^b$ and $\defcolor{\Gdot^1}:=\Fdot^1 \cap W$.
 For $h$ in the intersection~\eqref{Eq:TypeIVint}, let $\defcolor{\Gdot^i(h)}:=(h+\Fdot^i)\cap W$ for $j\in\{2,\dotsc,s\}$.
 All general instances of flags $\calGdot(h)$ arise from flags $\calFdot$ with $F^1_2,F^1_{b+1},F^2_5, F^3_{b+2}$, and
 $F^4_{b+2}$ fixed (and thus $V$, $W$, and the intersection~\eqref{Eq:TypeIVint} are fixed).
 This proves part (4) of  Definition~\ref{D:fibration}. 
 For parts (2) and (3),  linear-algebra arguments will suffice as before.
\end{proof}

%%%%%%%%%%%%%%%%%%%%%%%%%%%%%%%%%%%%%%%%%%%%%%%%%%%%%%%%%%%%%%%%%%%%%%%%%%%%%%%%%
%
\section{Conclusion}
We studied the Galois groups of all $81,533$ nontrivial Schubert problems in $\Gr(4,9)$.
Of the $31,806$ essential problems, 149 had Galois group that did not contain the alternating group, and we
identified the Galois group of each of these 149 enriched problems.
We discussed several methods to study Schubert Galois groups, including Vakil's algorithm and computing Frobenius elements.
We also introduced a new structure, a fibration of Schubert problems, that explained the Galois groups of the 149 enriched
problems.
This is both a first step towards the inverse Galois problem in Schubert calculus and points towards a potential
classification of enriched Schubert problems.

%%%%%%%%%%%%%%%%%%%%%%%%%%%%%%%%%%%%%%%%%%%%%%%%%%%%%%%%%%%%%%%%%%%%%%%%%%%%%

\providecommand{\bysame}{\leavevmode\hbox to3em{\hrulefill}\thinspace}
\providecommand{\MR}{\relax\ifhmode\unskip\space\fi MR }
\providecommand{\MRhref}[2]{%
  \href{http://www.ams.org/mathscinet-getitem?mr=#1}{#2}
}
\providecommand{\href}[2]{#2}

%%%%%%%%%%%%%%%%%%%%%%%%%%%%%%%%%%%%%%%%%%%%%%%%%%%%%%%%%%%%%%%%%%%%%%%%%%%%%


\begin{thebibliography}{10}

\bibitem{AR}
C.~Am\'endola, J.~Lindberg, and J.I. Rodriguez, \emph{Solving parameterized
  polynomial systems with decomposable projections}, {\tt
  arXiv.org/1612.08807}, 2021.

\bibitem{BMMT}
E.~Becker, M.G. Marinari, T.~Mora, and C.~Traverso, \emph{The shape of the
  {S}hape {L}emma}, Proceedings ISSAC-94, 1993, pp.~129--133.

\bibitem{BCDLT}
H.~Bercovici, B.~Collins, K.~Dykema, W.~S. Li, and D.~Timotin,
  \emph{Intersections of {S}chubert varieties and eigenvalue inequalities in an
  arbitrary finite factor}, J. Funct. Anal. \textbf{258} (2010), no.~5,
  1579--1627.

\bibitem{BdCS}
C.~J. Brooks, A.~Mart\'\i n~del Campo, and F.~Sottile, \emph{Galois groups of
  {S}chubert problems of lines are at least alternating}, Trans. Amer. Math.
  Soc. \textbf{367} (2015), 4183--4206.

\bibitem{SDSS}
T.~Brysiewicz, J.~I. Rodriguez, F.~Sottile, and T.~Yahl, \emph{Solving
  decomposable sparse systems}, Numerical Algorithms \textbf{88} (2021).

\bibitem{BS_homotopy}
C.I. Byrnes and P.K. Stevens, \emph{Global properties of the root-locus map},
  Feedback control of linear and nonlinear systems, Lecture Notes in Control
  and Inform. Sci., vol.~39, Springer, 1982, pp.~9--29.

\bibitem{Cameron}
P.J. Cameron, \emph{Permutation groups}, London Mathematical Society Student
  Texts, vol.~45, Cambridge University Press, Cambridge, 1999.

\bibitem{Ch1864}
M.~Chasles, \emph{Construction des coniques qui satisfont {\`a} cinque
  conditions}, C. R. Acad. Sci. Paris \textbf{58} (1864), 297--308.

\bibitem{Singular}
W.~Decker, G.-M. Greuel, G.~Pfister, and H.~Sch\"onemann, \emph{{\sc Singular}
  {4-1-1} --- {A} computer algebra system for polynomial computations}, {\tt
  http://www.singular.uni-kl.de}, 2018.

\bibitem{Ekedahl}
T.~Ekedahl, \emph{An effective version of {H}ilbert's irreducibility theorem},
  S\'eminaire de {T}h\'eorie des {N}ombres, {P}aris 1988--1989, Progr. Math.,
  vol.~91, Birkh\"auser, Boston, MA, 1990, pp.~241--249.

\bibitem{Esterov}
A.~Esterov, \emph{Galois theory for general systems of polynomial equations},
  Compos. Math. \textbf{155} (2019), no.~2, 229--245.

\bibitem{Fu97}
W.~Fulton, \emph{Young tableaux}, London Mathematical Society Student Texts,
  vol.~35, Cambridge University Press, Cambridge, 1997.

\bibitem{M2}
D.~R. Grayson and M.~E. Stillman, \emph{Macaulay2, a software system for
  research in algebraic geometry}, Available at {\tt
  http://www.math.uiuc.edu/Macaulay2/}.

\bibitem{Ha79}
J.~Harris, \emph{Galois groups of enumerative problems}, Duke Math.~J.
  \textbf{46} (1979), 685--724.

\bibitem{HRS}
J.~D. Hauenstein, J.~I. Rodriguez, and F.~Sottile, \emph{Numerical computation
  of {G}alois groups}, Found. Comput. Math. \textbf{18} (2018), 867--890.

\bibitem{HaHS}
J.D. Hauenstein, N.~Hein, and F.~Sottile, \emph{A primal-dual formulation for
  certifiable computations in {S}chubert calculus}, Found. Comput. Math.
  \textbf{16} (2016), no.~4, 941--963.

\bibitem{HS_Sq}
N.~Hein and F.~Sottile, \emph{A lifted square formulation for certifiable
  {S}chubert calculus}, J. Symb. Comp. \textbf{79} (2017), no.~part 3,
  594--608.

\bibitem{Hermite}
C.~Hermite, \emph{Sur les fonctions alg{\'e}briques}, CR Acad. Sci.(Paris)
  \textbf{32} (1851), 458--461.

\bibitem{Jacobson85}
Nathan Jacobson, \emph{Basic {A}lgebra {I}}, 2 ed., W.H. Freeman, 1985.

\bibitem{J1870}
C.~Jordan, \emph{Trait\'e des substitutions et des \'equations alg\'ebrique},
  Gauthier-Villars, Paris, 1870.

\bibitem{KL72}
S.~Kleiman and D.~Laksov, \emph{Schubert calculus}, Amer. Math. Monthly
  \textbf{79} (1972), 1061--1082.

\bibitem{K74}
S.~L. Kleiman, \emph{The transversality of a general translate}, Compos. Math.
  \textbf{28} (1974), 287--297.

\bibitem{Lang}
Serge Lang, \emph{Algebra}, third ed., GTM, vol. 211, Springer-Verlag, New
  York, 2002.

\bibitem{LMSVV}
A.~Leykin, A.~Mart\'\i n~del Campo, F.~Sottile, R.~Vakil, and J.~Verschelde,
  \emph{Numerical {S}chubert calculus via the {L}ittlewood-{R}ichardson
  homotopy algorithm}, Math. Comp. \textbf{90} (2021), 1407--1433.

\bibitem{LS09}
A.~Leykin and F.~Sottile, \emph{Galois groups of {S}chubert problems via
  homotopy computation}, Math. Comp. \textbf{78} (2009), no.~267, 1749--1765.

\bibitem{MSJ}
A.~Mart\'\i n~del Campo and F.~Sottile, \emph{Experimentation in the {S}chubert
  calculus}, Schubert Calculus, Osaka 2012 (H.~Naruse, T.~Ikeda, M.~Masuda, and
  T.~Tanisaki, eds.), Advanced Studies in Pure Mathematics, vol.~71,
  Mathematical Society of Japan, 2016, pp.~295--336.

\bibitem{PirolaSchlesinger}
G.~P. Pirola and E.~Schlesinger, \emph{Monodromy of projective curves}, J.
  Algebraic Geom. \textbf{14} (2005), no.~4, 623--642.

\bibitem{Sch1886c}
H.~Schubert, \emph{Anzahl-{B}estimmungen f{\"u}r lineare {R}{\"a}ume beliebiger
  {D}imension}, Acta. Math. \textbf{8} (1886), 97--118.

\bibitem{OS}
L.L. Scott, \emph{Representations in characteristic {$p$}}, The {S}anta {C}ruz
  {C}onference on {F}inite {G}roups ({U}niv. {C}alifornia, {S}anta {C}ruz,
  {C}alif., 1979), Proc. Sympos. Pure Math., vol.~37, Amer. Math. Soc.,
  Providence, R.I., 1980, pp.~319--331.

\bibitem{So97}
F.~Sottile, \emph{Enumerative geometry for the real {G}rassmannian of lines in
  projective space}, Duke Math. J. \textbf{87} (1997), no.~1, 59--85.

\bibitem{SW_double}
F.~Sottile and J.~White, \emph{Double transitivity of {G}alois groups in
  {S}chubert calculus of {G}rassmannians}, Algebr. Geom. \textbf{2} (2015),
  no.~4, 422--445.


\bibitem{SWY}
Frank Sottile, Robert Williams, and Li~Ying, \emph{Galois groups of composed
  {S}chubert problems}, Facets of algebraic geometry. {V}ol. {II}, London Math.
  Soc. Lecture Note Ser., vol. 473, Cambridge Univ. Press, Cambridge, 2022,
  pp.~336--366.

\bibitem{GGEGA}
F.~Sottile and T.~Yahl, \emph{Galois groups in enumerative geometry and
  applications}, 2021, {\tt arXiv:2108.07905}.

\bibitem{Va06a}
R.~Vakil, \emph{A geometric {L}ittlewood-{R}ichardson rule}, Ann. of Math. (2)
  \textbf{164} (2006), no.~2, 371--421, Appendix A written with A. Knutson.

\bibitem{Va06b}
\bysame, \emph{Schubert induction}, Ann. of Math. (2) \textbf{164} (2006),
  no.~2, 489--512.

\bibitem{Williams}
R.L. Williams, \emph{Restrictions on {G}alois groups of {S}chubert problems},
  Ph.D. thesis, Texas A\&M University, 2017.

\end{thebibliography}
\end{document}